\documentclass[12pt, a4paper]{amsart}
\usepackage{amsmath}
\usepackage{geometry,amsthm,graphics,tabularx,amssymb,shapepar, eucal}
\usepackage{latexsym,amsfonts,amssymb,amsmath,amsxtra,bbm,color,mathrsfs}
\usepackage{amscd}
\usepackage{mathdots}
\allowdisplaybreaks[4]

\usepackage{hyperref}
\usepackage{pdfsync}
\usepackage[all]{xypic}

\usepackage{verbatim}

\newcommand{\BC}{{\mathbb {C}}}

\newcommand{\BQ}{{\mathbb {Q}}}

\newcommand{\CA}{{\mathcal {A}}}

\newcommand{\CC}{{\mathcal {C}}}

\newcommand{\CE}{{\mathcal {E}}}

\newcommand{\CG}{{\mathcal {G}}}
\newcommand{\CH}{{\mathcal {H}}}

\newcommand{\CL}{{\mathcal {L}}}
\newcommand{\CM}{{\mathcal {M}}}
\newcommand{\CN}{{\mathcal {N}}}
\newcommand{\CO}{{\mathcal {O}}}
\newcommand{\CP}{{\mathcal {P}}}

\newcommand{\CR}{{\mathcal {R}}}

\newcommand{\CX}{{\mathcal {X}}}

\newcommand{\RB}{{\mathrm {B}}}
\newcommand{\RC}{{\mathrm {C}}}

\newcommand{\RE}{{\mathrm {E}}}

\newcommand{\RH}{{\mathrm {H}}}
\newcommand{\RI}{{\mathrm {I}}}

\newcommand{\RL}{{\mathrm {L}}}
\newcommand{\RM}{{\mathrm {M}}}
\newcommand{\RN}{{\mathrm {N}}}

\newcommand{\RP}{{\mathrm {P}}}
\newcommand{\RQ}{{\mathrm {Q}}}

\newcommand{\RT}{{\mathrm {T}}}
\newcommand{\RU}{{\mathrm {U}}}

\newcommand{\RW}{{\mathrm {W}}}

\newcommand{\RZ}{{\mathrm {Z}}}

\newcommand{\Aut}{{\mathrm{Aut}}}

\newcommand{\GL}{{\mathrm{GL}}}

\newcommand{\Hom}{{\mathrm{Hom}}}

\newcommand{\PGL}{{\mathrm{PGL}}}

\renewcommand{\Re}{{\mathrm{Re}}}

\newcommand{\rk}{{\mathrm{k}}}

\newcommand{\tr}{{\mathrm{tr}}}

\newcommand{\n}{\mathfrak n}

\renewcommand{\l}{\mathfrak l}

\renewcommand{\rk}{\mathrm k}

\newcommand{\Z}{\mathbb{Z}}
\newcommand{\C}{\mathbb{C}}
\newcommand{\R}{\mathbb R}

\newcommand{\K}{\mathbb{K}}

\newcommand{\A}{\mathbb{A}}

\newcommand{\ve}{{\vee}}

\newcommand{\abs}[1]{\lvert#1\rvert}

\newcommand{\la}{\langle}
\newcommand{\ra}{\rangle}

\newcommand{\be}{\begin {equation}}
\newcommand{\ee}{\end {equation}}
\newcommand{\bee}{\begin {equation*}}
\newcommand{\eee}{\end {equation*}}

\theoremstyle{Theorem}

\theoremstyle{Theorem}

\newtheorem*{theoremA'}{Theorem A'}

\theoremstyle{Theorem}

\theoremstyle{Theorem}
\newtheorem{prp}{Proposition}[section]

\newtheorem{lemp}[prp]{Lemma}
\newtheorem{thmp}[prp]{Theorem}

\theoremstyle{Plain}

\newtheorem{remarkp}[prp]{Remark}

\theoremstyle{Definition}

\newtheorem{dfnp}[prp]{Definition}
\newtheorem{examplep}[prp]{Example}

\numberwithin{equation}{section}

\begin{document}

	\title[Period Relations]{Period Relations for Rankin-Selberg convolutions for $\GL(n)\times\GL(n)$}
	
	\author[Y. Jin]{Yubo Jin}
	\address{Institute for Advanced Study in Mathematics, Zhejiang University\\
		Hangzhou, 310058, China}\email{yubo.jin@zju.edu.cn}

	\author[J.-S. Li]{Jian-Shu Li}
	\address{Institute for Advanced Study in Mathematics, Zhejiang University\\
		Hangzhou, 310058, China}\email{jianshu@zju.edu.cn}
	
	\author[D. Liu]{Dongwen Liu}
	\address{School of Mathematical Sciences,  Zhejiang University\\
		Hangzhou, 310058, China}\email{maliu@zju.edu.cn}
	
	\author[B. Sun]{Binyong Sun}
	\address{Institute for Advanced Study in Mathematics and New Cornerstone Science Laboratory, Zhejiang University\\
		Hangzhou, 310058, China}\email{sunbinyong@zju.edu.cn}

	\subjclass[2020]{Primary 11F67; Secondary 11F70, 11F75, 22E45}
	\keywords{Deligne's conjecture, L-function, Rankin-Selberg convolution, period relation,  Eisenstein cohomology}  

	\maketitle
	
	\begin{abstract}
    Using the modular symbol approach, we prove the rationality and period relations for certain critical values of Rankin-Selberg L-functions for $\GL(n)\times\GL(n)$ over any number field that contains a CM field.
	\end{abstract}
	
	\tableofcontents
	
	\section{Introduction}

    One of the central problems in number theory is the study of special values of L-functions. Deligne's celebrated conjecture (\cite{D}) predicts the rationality of critical values of L-functions attached to pure motives. In \cite{B}, Blasius established relations between Deligne's periods of certain pure motives and their Artin twists, and further proposed a conjecture on period relations among different critical values of L-functions. The main goal of this paper is to investigate the rationality and period relations for Rankin–Selberg L-functions. Within the framework of the Langlands program, our results are consistent with the conjectures of Deligne and Blasius.
	
The cases of $\GL_n\times\GL_{n-1}$ and $\GL_n\times\GL_n$ are of fundamental importance in Rankin-Selberg theory, and many problems for general Rankin-Selberg convolutions reduce to these two cases. The period relations for the $\GL_n\times\GL_{n-1}$ case have been studied by many authors. See \cite{Sh59, Ma72, Ma73, Ma76, Sh76, Sh77, Sh78, Har83, Hi94, Rag11} for $n=2$, and see \cite{Sch2, Ma98} for $n=3$. For general $n$, much work has been done since the landmark work of Kazhdan, Mazur and Schmidt in \cite{KMS00}; see \cite{Ma05,KS13,Rag1,Rag2,GH,HL,GS20,GL,GHL,LLS, HMN25}. Much less is known for the $\GL_n\times\GL_n$ case; see \cite{Hi94,Sh76,Sh78} for $n=2$ and \cite{Gr} for general $n$. The aim of this paper is to prove some period relations for the $\GL_n\times\GL_n$ case as a continuation of \cite{LLS}. We follow the well-established modular symbol approach used in the aforementioned works. In particular, our treatment of the $\GL_n\times\GL_n$ case follows that of \cite{Hi94, Gr}, and the results of this paper can be viewed as a refinement of those found there.

	\subsection{Representations and coefficient systems}
	
	Let $\rk$ be a number field that contains a CM field. Denote by $\rk_v$ the completion of $\rk$ at a place $v$ with the normalized absolute value $\abs{\cdot}_v$. Write
	\[
	\rk_{\infty}:=\rk\otimes_{\BQ}\R=\prod_{v|\infty}\rk_v\hookrightarrow\rk\otimes_{\BQ}\C=\prod_{\iota\in\CE_{\rk}}\C,
	\]
	where $\CE_{\rk}$ is the set of field embeddings $\iota:\rk\hookrightarrow\C$. For each $\iota\in\CE_{\rk}$, we write $\overline{\iota}$ for its complex conjugate. Write $\mathbb{A}=\mathbb{A}_{\mathrm{f}}\times\rk_{\infty}$ for the adèle ring of $\rk$ where $\mathbb{A}_{\mathrm{f}}$ denotes the ring of finite adèles, and set $\abs{\cdot}_{\A}:=\prod_v\abs{\cdot}_v$.

    Given a positive integer $n$, let $\GL_n$ denote the general linear group over $\rk$ and $1_n$ the identity matrix of size $n$. The center of $\GL_n$ is denoted by $\RZ_n$. Let $\RP_n$ be the  block upper-triangular parabolic subgroup of $\GL_n$ with Levi component $\GL_{n-1} \times \GL_1$.

For each archimedean place $v$, denote by $K_{n,v}$ the standard maximal compact subgroup of $\GL_n(\rk_v)$, which is  a  compact unitary group. Write  \[
\widetilde{K}_{n,v}:=K_{n,v}\cdot\RZ_n(\rk_v)\subset \GL_n(\rk_v).
\]
Denote by $\mathfrak{g}_{n,v}$ the complexified Lie algebra of $\GL_n(\rk_v)$. Set
	\[
	\mathfrak{g}_{n,\infty}:=\bigoplus_{v|\infty}\mathfrak{g}_{n,v},\qquad  K_{n,\infty}:=\prod_{v|\infty}K_{n,v},\qquad\widetilde{K}_{n,\infty}:=\prod_{v|\infty}\widetilde{K}_{n,v}.
	\]

  As usual,  a weight
\be\label{writemu}
\mu := \{\mu^{\iota}\}_{\iota \in \CE_{\rk}} := \{(\mu_1^{\iota}, \mu_2^{\iota}, \dots, \mu_n^{\iota})\}_{\iota \in \CE_{\rk}} \in (\Z^n)^{\CE_{\rk}}
\ee
is called dominant if
\[
\mu_1^{\iota} \geq \mu_2^{\iota} \geq \dots \geq \mu_n^{\iota}, \qquad \text{for all } \iota \in \CE_{\rk}.
\]
For such a dominant weight, we denote by $F_{\mu}$ the (unique up to isomorphism) irreducible holomorphic finite-dimensional representation of $\GL_n(\rk \otimes_{\BQ} \C)$ of highest weight $\mu$. 

Let $\Sigma$ be an irreducible smooth automorphic representation of $\GL_n(\A)$ that is cohomological in the following sense: 
there exists a dominant weight $\mu$ as above  such that  the total relative Lie algebra cohomology (see \cite{Clo, HM62, VZ84})
\be\label{totalcoh0}
\RH^*(\mathfrak{g}_{n,\infty},\widetilde{K}_{n,\infty}; F_{\mu}^\vee  \otimes \Sigma_\infty)\neq \{0\}\quad (\textrm{`${}^{\vee}$' indicates the contragredient}),
\ee
where $\Sigma_{\infty} := \widehat{\otimes}_{v \mid \infty} \Sigma_v$ is  the infinite part of $\Sigma$ ($\widehat{\otimes}$ indicates the completed inductive tensor product, see \cite[Definition 43.5]{Tr}). Here $\Sigma_v$ denotes the local $v$-component of $\Sigma$ (similar notation of subscripts `$v$', `$\infty$', and `$\mathrm f$' respectively for the local $v$-component, the infinite part,  and the  finite part will be used without explanation). 
We call $F_{\mu}$ the coefficient system of $\Sigma$. 

Let $\Sigma'$ be another cohomological irreducible smooth automorphic representation of $\GL_n(\A)$, with coefficient system $F_\nu$, where 
\begin{equation}\label{writenu}
\nu := \{\nu^{\iota}\}_{\iota \in \CE_{\rk}} := \{(\nu_1^{\iota}, \nu_2^{\iota}, \dots, \nu_n^{\iota})\}_{\iota \in \CE_{\rk}} \in (\Z^n)^{\CE_{\rk}}
\end{equation}
is another dominant weight. Set $$\Pi := \Sigma \,\widehat{\otimes}\,\Sigma'.$$ It is an   irreducible smooth automorphic representation of $\GL_n(\A) \times \GL_n(\A)$ which is cohomological with coefficient system $F_{\mu} \otimes F_{\nu}$ (in the sense that the obvious analogue of \eqref{totalcoh0} holds for $\Pi$).

	Let $\chi:\rk^{\times}\backslash\A^{\times}\to\C^{\times}$ be a Hecke character. When no confusion occurs, we identify a character of a group with the complex vector space $\C$ equipped with the corresponding action of the group. Denote by $\omega_{\Pi}$ the product of the central characters of $\Sigma$ and $\Sigma'$, and put 
    \be\label{eta}
    \eta:=\omega_{\Pi}\chi^n.
    \ee
    Consider the degenerate principal series representation
	\begin{equation}\label{Ieta}
		I_{\eta}:=\mathrm{Ind}^{\GL_n(\A)}_{\RP_n(\A)}(\abs{\det}_{\A}^{-\frac{1}{2}}\otimes\eta^{-1}\abs{\cdot}_{\A}^{\frac{n-1}{2}})={^{\mathrm{u}}\mathrm{Ind}}^{\GL_n(\A)}_{\RP_n(\A)}(\mathbf{1}\otimes\eta^{-1}),
	\end{equation}
	where $\det$ indicates the determinant character and $\mathbf{1}$ indicates the trivial representation. Here and henceforth, $\mathrm{Ind}$ and ${^{\mathrm{u}}\mathrm{Ind}}$ stand for the normalized and un-normalized smooth parabolic inductions respectively. We have the usual factorizations
    \begin{equation}\label{factorization}
			\chi=\otimes_v\chi_v=\chi_{\mathrm{f}}\otimes\chi_{\infty},\quad 
			\eta=\otimes_v\eta_v=\eta_{\mathrm{f}}\otimes\eta_{\infty},\quad 
			I_{\eta}=\widehat\otimes_v'I_{\eta,v}=I_{\eta,\mathrm{f}}\otimes I_{\eta,\infty}. 
	\end{equation}

	Write
	\[
	\begin{aligned}
		\mathrm{d}\chi:=\mathrm{d}\chi_{\infty}&:=\{\chi_{\iota}\}_{\iota\in\CE_{\rk}}\in\rk\otimes_{\BQ}\C=\C^{\CE_{\rk}},\\	
		\mathrm{d}\eta:=\mathrm{d}\eta_{\infty}&:=\{\eta_{\iota}\}_{\iota\in\CE_{\rk}}\in\rk\otimes_{\BQ}\C=\C^{\CE_{\rk}},
	\end{aligned}
	\]
	for the complexified differentials of $\chi_{\infty}$ and $\eta_{\infty}$,  respectively. We assume $\chi$ is algebraic in the sense that $\mathrm{d}\chi\in\Z^{\CE_{\rk}}$. Denote by $F_{\chi}$ the holomorphic character of $\GL_n(\rk\otimes_{\BQ}\C)$ extending the character $\chi_{\infty}\circ\det:\GL_n(\rk_{\infty})\to\C^{\times}$. Note that $\eta$ is also algebraic with
	\[
	\eta_{\iota}=\sum_{i=1}^n(\mu_i^{\iota}+\nu_i^{\iota})+n\chi_{\iota},\qquad\text{for all }\iota\in\CE_{\rk}.
	\]

The representation  $I_{\eta,\infty}$ is said to be regular algebraic if its infinitesimal character coincides with that of an irreducible holomorphic finite-dimensional representation of $\GL_n(\rk \otimes_{\BQ} \C)$. This finite-dimensional representation is called the coefficient system of  $I_{\eta,\infty}$ and is denoted by $F_{\eta}$.

  We say that the algebraic Hecke character $\chi$ is $(\mu,\nu)$-balanced if $I_{\eta,\infty}$ is regular algebraic and
       \begin{equation}\label{balanced}
		\left(F_{\mu}\otimes F_{\nu}\otimes F_{\eta}\otimes F_{\chi}\right)^{\GL_n(\rk\otimes_{\BQ}\C)}\neq\{0\}\qquad (\textrm{the invariant space}).
	\end{equation}

    It is easy to see that $I_{\eta,\infty}$ is regular algebraic if and only if 
	\begin{equation}\label{regular}
\eta_{\iota}(\eta_{\iota}-n)\geq 0, \qquad \textrm{for all $\iota\in\mathcal{E}_{\rk}$.}
	\end{equation}
    In this paper, we further impose the condition
	\begin{equation}\label{CM}
		\min\{\eta_{\iota},\eta_{\overline{\iota}}\}\leq 0\quad\text{and}\quad\max\{\eta_{\iota},\eta_{\overline{\iota}}\}\geq n,\quad \textrm{for all $\iota\in\CE_{\rk}$},  
	\end{equation}
	which is labelled as Case $(\pm)$ in \cite{local}. Note that \eqref{CM} implies the inequalities in \eqref{regular} and can happen only when $\rk$ contains a CM field, as assumed from the beginning. Denote by $\mathrm{B}(\mu,\nu)$ the set of all such $(\mu,\nu)$-balanced characters that also satisfy \eqref{CM}.

	We say that an algebraic Hecke character $\chi$ is $(\mu,\nu)$-critical if $s=0$ is not a pole of the archimedean local L-function $\RL(s,\Pi_{\infty}\times\chi_{\infty})$ or $\RL(1-s,\Pi_{\infty}^{\vee}\times\chi_{\infty}^{-1})$. Here $\RL(s,\Pi_{\infty}\times\chi_{\infty}):=\RL(s,\Sigma_{\infty}\times\Sigma_{\infty}'\times\chi_{\infty})$ denotes the archimedean Rankin-Selberg L-function, and similar notations apply to other (local or global) Rankin-Selberg L-functions. Denote by $\mathrm{C}(\mu,\nu)$ the set of all $(\mu,\nu)$-critical characters. It is shown in \cite[Lemma 2.7]{local} that $\RB(\mu,\nu)\subset\mathrm{C}(\mu,\nu)$ when $\Pi_{\infty}$ is generic. We remark that $\RB(\mu,\nu)\neq\RC(\mu,\nu)$ in general; see \cite[Section 2.4.3]{local} for a counterexample. The method of the present paper does not apply to the characters in $\mathrm{C}(\mu,\nu)\setminus\RB(\mu,\nu)$. 
	
	\begin{remarkp}
		Denote by $\RB(\mu,\nu)^{\natural}$ (resp. $\RC(\mu,\nu)^{\natural}$) the subset of $\RB(\mu,\nu)$ (resp. $\RC(\mu,\nu)$) consisting of all $(\mu,\nu)$-balanced characters (resp. $(\mu,\nu)$-critical characters) of the form $\chi=\chi_0\abs{\cdot}_{\A}^{t}$, where $t\in\Z$ and $\chi_0$ is a finite order character of $\rk^\times\backslash \A^\times$. If $\Pi_{\infty}$ is generic, then
        \[
        \RB(\mu,\nu)^{\natural}\neq\emptyset\quad \Rightarrow \quad \RB(\mu,\nu)^{\natural}=\RC(\mu,\nu)^{\natural}.
        \]
	\end{remarkp}

	\subsection{The main theorem}

   We investigate the special value $\RL(0,\Pi \times \chi)$ when $\chi \in \RB(\mu,\nu)$ via the method of modular symbols. Throughout this paper, all global L-functions are taken to be completed. Let $\rk_1$ be the maximal CM subfield of $\rk$. Fix a CM type of $\rk_1$, that is, for every archimedean place $w$ of $\rk_1$, we fix a field embedding $\iota_w : \rk_{1} \hookrightarrow \C$ that induces $w$.  
   This determines a CM type of $\rk$ such that for every archimedean place $v$ of $\rk$ over $w$ we have a continuous field embedding 
   $\iota_v: \rk_v\hookrightarrow \C$ with $\iota_v|_{\rk_1} = \iota_w$. Then $\CE_{\rk_v} := \{\iota_v, \overline{\iota}_v\}$ is the set of topological isomorphisms $\rk_{v} \xrightarrow{\sim} \C$, which is naturally viewed as a subset of $\CE_{\rk}$. Here $\overline{\iota}_v$ denotes the complex conjugate of $\iota_v$. 
  
  For all $\chi \in \RB(\mu,\nu)$ and all archimedean places $v$ of $\rk$, we set \[
   \varepsilon_{\Pi,\chi,v}:=\begin{cases}
       1,\qquad & \textrm{if $\eta_{\iota_v} \leq 0$};\\
       -1,\qquad & \textrm{if $\eta_{\iota_v} \geq n$}.
   \end{cases}
   \]
   Denote $\mathrm{i}:=\sqrt{-1}\in \C$. 
Define
\begin{equation}\label{omegainfty}
\begin{aligned}
    \Omega_{\infty}(\Pi,\chi):=\delta_{\infty}(\omega_{\Pi}\chi^n)^{n-1}\cdot\left(\prod_{v|\infty}\varepsilon_{\Pi,\chi,v}\right)^{\frac{n(n-1)}{2}}\cdot\prod_{v|\infty}\prod_{i+k\leq n}\mathrm{i}^{\sum_{\iota\in\CE_{\rk_v}}(\mu_i^{\iota}+\nu_k^{\iota}+\chi_{\iota})-1},
\end{aligned}
\end{equation}
    where $\delta_{\infty}(\omega_{\Pi}\chi^n)$ is defined in \eqref{omegaeta}. 
   
     Note that $\chi \in \RB(\mu,\nu)$ implies that the character $\eta:=\omega_{\Pi} \chi^n$ is critical, namely $s=0$ is not a pole of $\RL(s,\eta_{\infty})$ or $\RL(1-s,\eta_\infty^{-1})$. Our main theorem is stated as follows.

	\begin{thmp}\label{mainthm}
		Let $\Pi=\Sigma\,\widehat{\otimes}\,\Sigma'$ be an irreducible smooth automorphic representation of $\GL_n(\A)\times\GL_n(\A)$ that is cohomological with coefficient system $F_{\mu}\otimes F_{\nu}$. Assume that $\Sigma$ is tamely isobaric (Definition \ref{def:isobaric}) with generic infinite part, and $\Sigma'$ is cuspidal. Then 
			\begin{equation}\label{thmquotient}
			\frac{\RL(0,\Pi\times\chi)}{\mathrm{c}(\omega_{\Pi}\chi^n)\cdot
				\CG(\chi)^{\frac{n(n-1)}{2}}\cdot\Omega_{\infty}(\Pi,\chi)\cdot\Omega(\Pi)}\in\BQ(\Pi,\chi)
		\end{equation}
        for all  $\chi\in\RB(\mu,\nu)$, where
		\begin{itemize}
			\item $\BQ(\Pi,\chi)$ is the compositum of the rationality fields of $\Pi$ and $\chi$;
			\item $\omega_\Pi$ is the product of the central characters of $\Sigma$ and $\Sigma'$;
			\item $\CG(\chi)\in\C^{\times}$ is the Gauss sum of $\chi$ defined in 
			\eqref{globalgauss};
            \item $\mathrm{c}(\omega_{\Pi}\chi^n):=\mathrm{c}^+(\omega_{\Pi}\chi^n)\cdot\RL(0,\omega_{\Pi_{\infty}}\chi_{\infty}^n)$, with
            $\mathrm{c}^+(\omega_{\Pi}\chi^n)\in\C^{\times}$ the Deligne period associated to $\omega_{\Pi}\chi^n$ defined in \eqref{cplus} and $\omega_{\Pi_{\infty}}$ the archimedean component of $\omega_\Pi$;
			\item $\Omega(\Pi)\in\C^{\times}$ is the Whittaker period of $\Pi$ defined in Section \ref{sec:whittaker}. 
		\end{itemize}
		Moreover, the quotient \eqref{thmquotient} is $\mathrm{Aut}(\C)$-equivariant in the sense that
		\begin{equation}\label{thmsigma}
			\begin{aligned}
				&\sigma\left(\frac{\RL(0,\Pi\times\chi)}{\mathrm{c}(\omega_{\Pi}\chi^n)\cdot
					\CG(\chi)^{\frac{n(n-1)}{2}}\cdot\Omega_{\infty}(\Pi,\chi)\cdot\Omega(\Pi)}\right)\\
				=&	\frac{\RL(0,{^{\sigma}\Pi}\times{^{\sigma}\chi})}{\mathrm{c}(\omega_{{^{\sigma}\Pi}}\,{^{\sigma}\chi}^n)\cdot
					\CG({^{\sigma}\chi})^{\frac{n(n-1)}{2}}\cdot\Omega_{\infty}({^{\sigma}\Pi},{^{\sigma}\chi})\cdot\Omega({^{\sigma}\Pi})},
			\end{aligned}
		\end{equation}
        for all $\sigma\in\mathrm{Aut}(\C)$. Here ${^{\sigma}\chi}$ is the usual $\sigma$-twist of $\chi$, and the $\sigma$-twist ${^{\sigma}\Pi}$ of $\Pi$ is defined in Section \ref{sec:sigmapi}.
	\end{thmp}

In view of the fact that the sum $\eta_{\iota} + \eta_{\overline{\iota}}$ is independent of $\iota\in \CE_\rk$ by the purity lemma (\cite{We55} and \cite[Lemma 7]{R-Hecke}), we define 
\[
\RB(\mu,\nu)^+ := \{\chi \in \RB(\mu,\nu) : \eta_{\iota} + \eta_{\overline{\iota}} \geq n\}\subset\RB(\mu,\nu).
\]
We will prove Theorem \ref{mainthm} for $\chi\in\RB(\mu,\nu)^+$ using the modular symbol approach, and the remaining case will be proved  by applying the functional equation. The proof of Theorem \ref{mainthm} will depend on our previous works \cite{local, dual, JT26}.
    

    We remark that Theorem \ref{mainthm} is compatible with \cite{R-imaginary}, where the rationality of the quotient of successive critical L-values is studied, and with \cite{RS08}, where the period relations for Whittaker periods are established in the spirit of \cite{B}. When $n = 1$, we have $\Omega(\Pi) = 1$, and Theorem \ref{mainthm} reduces to Deligne's conjecture for Hecke characters, which was completely proved in \cite{Ku24}. We therefore assume $n \geq 2$ in the rest of the paper. When $n = 2$, the result was obtained in \cite{Hi94}. For general $n$, a form of the rationality result \eqref{thmquotient} was obtained by Grenié in \cite[Theorem 2]{Gr}, without the explicit determination of the constants $\CG(\chi)^{\frac{n(n-1)}{2}}$ and $\Omega_{\infty}(\Pi,\chi)$. The appearance of the term $\CG(\chi)^{\frac{n(n-1)}{2}}$ is  due to \cite{RS08}. The determination of the term $\Omega_{\infty}(\Pi,\chi)$ is one of the two main contributions of the present article.

     Compared with the $\GL_n\times\GL_{n-1}$ case, the main difficulty for the $\GL_n\times\GL_n$ case lies in proving the rationality of certain Eisenstein cohomology spaces (Theorem \ref{thm:Eis}), which is the other main contribution of the present article. The proof  follows the general strategy of Harder (\cite{Har87, Har90}) and the approach of  Grenié (\cite[Proposition 12]{Gr}). The rationality of the `restriction map to boundary cohomology' is a crucial step in Harder's strategy, which is missing from the proof of \cite[Proposition 12]{Gr}. The main goal of Sections \ref{sec:2}–\ref{sec:5} is to complete this step (see Proposition \ref{mainprop}).  We will introduce the result on Eisenstein cohomology in more detail in Section \ref{intro:Eis}.

    We make two improvements over the result in \cite{LLS}. One is that we drop the assumption that $\chi\in\RB(\mu,\nu)^{\natural}$ and the other is that the cohomological tamely isobaric representations considered in the current paper are more general than those treated in \cite{LLS}. Both improvements can also be obtained for the $\GL_n\times\GL_{n-1}$ case without difficulty. We mention that allowing $\Sigma$ to be isobaric 
    is a common idea for deducing rationality results for general Rankin-Selberg L-functions from known ones. See \cite{Sch2,Ma98,Ma05,GH,HL,GS20,GL,GHL,Chen}. All these works carry out induction arguments starting from the rationality results for $\GL_n\times\GL_{n-1}$ Rankin-Selberg L-functions. We expect that our result for the $\GL_n\times\GL_n$ case may serve as another starting point for studying general Rankin-Selberg L-functions by induction.

	\subsection{The Eisenstein cohomology}\label{intro:Eis}

	We now introduce our main result on rationality of the Eisenstein cohomology.

Let $\eta:\rk^{\times}\backslash\A^{\times}\to\C^{\times}$ be an algebraic Hecke character such that
\begin{equation}\label{etaassumption}
	\min\{\eta_{\iota},\eta_{\overline{\iota}}\}\leq 0
    \quad\textrm{and}\quad \eta_{\iota}+\eta_{\overline{\iota}}\geq n
\end{equation}
for all $\iota\in\CE_{\rk}$. Note that \eqref{etaassumption} implies \eqref{CM}.  Let $I_{\eta}$ be the induced representation defined in \eqref{Ieta} with coefficient system  $F_\eta$ as before. For every $q\in\Z$, we consider the relative Lie algebra cohomology
	\[
	\RH^q(\mathfrak{g}_{n,\infty},\widetilde{K}_{n,\infty};I_{\eta}\otimes F^\vee_{\eta}).
	\]
    There is a canonical isomorphism
	\begin{equation}\label{canI}
		\begin{aligned}	&\imath_{\mathrm{can}}:I_{\eta,\mathrm{f}}\otimes\RH^q(\mathfrak{g}_{n,\infty},\widetilde{K}_{n,\infty};I_{\eta,\infty}\otimes F^\vee_{\eta})\,\xrightarrow{\sim}\RH^q(\mathfrak{g}_{n,\infty},\widetilde{K}_{n,\infty};I_{\eta}\otimes F^\vee_{\eta}).
		\end{aligned}
	\end{equation}
	As a consequence of Delorme's lemma (\cite[Theorem III.3.3]{BW}) we have that 
	\[
\RH^q(\mathfrak{g}_{n,\infty},\widetilde{K}_{n,\infty};I_{\eta,\infty}\otimes F^\vee_{\eta})=\{0\}\qquad\text{ for all }q<c_n:=\frac{[\rk:\BQ]}{2}\cdot(n-1), 
	\]
	and
	\be\label{dimhcn}
	\mathrm{dim}\,\RH^{c_n}(\mathfrak{g}_{n,\infty},\widetilde{K}_{n,\infty};I_{\eta,\infty}\otimes F^\vee_{\eta})=1.
	\ee
	
	Define 
	\[
	\mathcal{X}_n:=\GL_n(\rk)\backslash\GL_n(\A)/\widetilde{K}_{n,\infty}.
	\]
	For every open compact subgroup $K_{\mathrm{f}}$ of $\GL_n(\A_{\mathrm{f}})$, the representation $F^\vee_{\eta}$ defines a sheaf on $\mathcal{X}_n/K_{\mathrm{f}}$ as usual, which is still denoted by $F^\vee_{\eta}$. For $q\in\Z$, consider the sheaf cohomology space
	\[
	\RH^q(\mathcal{X}_n,F^\vee_{\eta}):=\lim_{\substack{\longrightarrow\\K_{\mathrm{f}}}}\RH^q(\mathcal{X}_n/K_{\mathrm{f}},F^\vee_{\eta}),
	\]
	where $K_{\mathrm{f}}$ runs over the directed system of open compact subgroups of $\GL_n(\A_{\mathrm{f}})$.

The Eisenstein series \eqref{Eisenstein} defines a $\GL_n(\A)$-equivariant continuous linear embedding (see Proposition \ref{prp-Eisenstein})
	\[
	\mathrm{Eis}_{\eta}:I_{\eta}\to\mathcal{A}(\GL_n(\rk)\backslash\GL_n(\A)),
	\]
	where $\mathcal{A}(\GL_n(\rk)\backslash\GL_n(\A))$ denotes the space of  smooth automorphic forms. On the cohomological level, this induces for each $q\in\Z$ an Eisenstein map
    \begin{eqnarray}\label{Eis}
\nonumber \mathrm{Eis}_{\eta}:\RH^q(\mathfrak{g}_{n,\infty},\widetilde{K}_{n,\infty};I_{\eta}\otimes F^\vee_{\eta})&\to &\RH^q(\mathfrak{g}_{n,\infty},\widetilde{K}_{n,\infty};\mathcal{A}(\GL_n(\rk)\backslash\GL_n(\A))\otimes F^\vee_{\eta})\\
  &\xrightarrow{\sim} & \RH^q(\mathcal{X}_n,F^\vee_{\eta}).
\end{eqnarray}

	For every $\sigma\in\mathrm{Aut}(\C)$, there is a unique Hecke character ${^{\sigma}\eta}={^{\sigma}\eta_{\mathrm{f}}}\otimes{^{\sigma}\eta_{\infty}}$ with ${^{\sigma}\eta_{\mathrm{f}}}=\sigma\circ\eta_{\mathrm{f}}$ and infinity type 
	\[
	\mathrm{d}{^{\sigma}\eta}:=\{{^{\sigma}\eta}_{\iota}\}_{\iota\in\mathcal{E}_{\rk}}=\{\eta_{\sigma^{-1}\circ\iota}\}_{\iota\in\mathcal{E}_{\rk}}.
	\]
    The purity lemma implies that the character ${^{\sigma}\eta}$ also satisfies \eqref{etaassumption}.
	The preceding discussion applies equally to ${^{\sigma}\eta}$.
    
    The $\sigma$-linear isomorphism 
	\[
	\sigma:F^\vee_{\eta}\to F^\vee_{{^{\sigma}\eta}}
	\]
	defined in \eqref{sigmaF} induces a $\sigma$-linear map
	\begin{equation}
		\sigma:\RH^q(\mathcal{X}_n,F^\vee_{\eta})\to\RH^q(\mathcal{X}_n,F^\vee_{{^{\sigma}\eta}}).
	\end{equation}
	On the other hand, we have a $\sigma$-linear isomorphism
	\be\label{sigmai}
	\sigma:I_{\eta,{\mathrm{f}}}\to I_{{^{\sigma}\eta},{\mathrm{f}}},\qquad \varphi\mapsto\sigma\circ\varphi
	\ee
	for the induced representations,  and we will define a $\sigma$-linear isomorphism
	\be\label{sigmah}
	\sigma:	\RH^{c_n}(\mathfrak{g}_{n,\infty},\widetilde{K}_{n,\infty};I_{\eta,\infty}\otimes F^\vee_{\eta})\to	\RH^{c_n}(\mathfrak{g}_{n,\infty},\widetilde{K}_{n,\infty};I_{{^{\sigma}\eta},\infty}\otimes F^\vee_{{^{\sigma}\eta}})
	\ee
	in \eqref{sigmaI}. See Section \ref{sec:lie} for details. In view of the canonical isomorphism \eqref{canI}, the $\sigma$-linear isomorphisms \eqref{sigmai} and \eqref{sigmah} yield a $\sigma$-linear isomorphism
	\[
	\sigma:	\RH^{c_n}(\mathfrak{g}_{n,\infty},\widetilde{K}_{n,\infty};I_{\eta}\otimes F^\vee_{\eta})\to	\RH^{c_n}(\mathfrak{g}_{n,\infty},\widetilde{K}_{n,\infty};I_{{^{\sigma}\eta}}\otimes F^\vee_{{^{\sigma}\eta}}).
	\]

 Following the general strategy of Harder (\cite{Har87, Har90}; see also \cite[Proposition 12]{Gr}), 
 we will prove the following theorem on the rationality of Eisenstein cohomology.

	\begin{thmp}\label{thm:Eis}
		Keep the above notation. The Eisenstein map $\mathrm{Eis}_{\eta}$ is $\mathrm{Aut}(\C)$-equivariant. That is, the diagram
		\begin{equation}\label{maindiagram}
			\begin{CD}
				\RH^{c_n}(\mathfrak{g}_{n,\infty},\widetilde{K}_{n,\infty};I_{\eta}\otimes F^\vee_{\eta}) @>\mathrm{Eis}_{\eta}>> \RH^{c_n}(\mathcal{X}_n,F^\vee_{\eta})\\
				@V\sigma VV @VV\sigma V\\
				\RH^{c_n}(\mathfrak{g}_{n,\infty},\widetilde{K}_{n,\infty};I_{^{\sigma}\eta}\otimes F^\vee_{^{\sigma}\eta}) @>\mathrm{Eis}_{^{\sigma}\eta}>>\RH^{c_n}(\mathcal{X}_n,F^\vee_{^{\sigma}\eta})
			\end{CD}
		\end{equation}
		commutes for every $\sigma\in\mathrm{Aut}(\C)$. 
	\end{thmp}
	
	As an application of Theorem \ref{thm:Eis}, we provide another proof of Deligne's conjecture for Hecke characters in \cite{Deligne} following a strategy initiated by Harder and Schappacher. To this end, we emphasize that the proof based on Theorem \ref{thm:Eis} is independent of \cite{Ku24}.

    \section{The Eisenstein series and constant terms}\label{sec:2}

    Sections \ref{sec:2}--\ref{sec:5} are devoted to proving Theorem \ref{thm:Eis} on the rationality of the Eisenstein cohomology. In this section, we review the Eisenstein series defined by degenerate principal series representations and compute their constant terms.

    Let $\RB_n = \RT_n \RN_n$ be the Borel subgroup of $\GL_n$ consisting of upper-triangular matrices, where $\RT_n$ is the diagonal torus and $\RN_n$ is its unipotent radical. Analogously, $\overline{\RB}_n = \RT_n \overline{\RN}_n$ is the opposite Borel subgroup of lower-triangular matrices with unipotent radical $\overline{\RN}_n$.

	\subsection{The Eisenstein series}	\label{sec:Eis}

    Let $\eta : \rk^{\times} \backslash \A^{\times} \to \C^{\times}$ be an algebraic Hecke character satisfying \eqref{etaassumption}. For every $s \in \C$, we set $\eta_s := \eta \abs{\cdot}_{\A}^{ns}$. Recall the assumption that $n \geq 2$. We consider the degenerate principal series representation
\begin{equation}\label{Ietas}
I_{\eta,s} := \mathrm{Ind}_{\RP_n(\A)}^{\GL_n(\A)} \bigl( \abs{\det\cdot}_{\A}^{-\frac{1}{2}} \otimes \eta_s^{-1} \abs{\cdot}_{\A}^{\frac{n-1}{2}} \bigr) = {}^{\mathrm{u}}\mathrm{Ind}_{\RP_n(\A)}^{\GL_n(\A)} (\mathbf{1} \otimes \eta_s^{-1}).
\end{equation}
When specialized to $s = 0$, this is the representation $I_{\eta} = \widehat{\otimes}_{v}' I_{\eta,v}$ from the Introduction.

Denote by $K_n := \GL_n(\widehat{\CO}) \cdot K_{n,\infty}$ a maximal compact subgroup of $\GL_n(\A)$, where $\widehat{\CO}$ is the profinite completion of the ring of integers $\CO$ of $\rk$. For every $\varphi \in I_{\eta}$ and $s \in \C$, we define
\be\label{varphis}
\varphi_s(g) := |p_{n,n}|_{\A}^{-ns} \cdot \varphi(g),
\ee
where $g = pk \in \RP_n(\A) K_n$ with $p = [p_{i,j}]_{1 \leq i,j \leq n} \in \RP_n(\A)$ and $k \in K_n$. Then clearly $\varphi_s \in I_{\eta,s}$. We form the Eisenstein series
\begin{equation}\label{Eisensteinseries}
\RE(g; \varphi_s, \eta) = \sum_{\gamma \in \RP_n(\rk) \backslash \GL_n(\rk)} \varphi_s(\gamma g).
\end{equation}
The series converges absolutely for $\mathrm{Re}(s)$ (the real part) sufficiently large and admits a meromorphic continuation to the whole complex plane. See \cite{BL24, La2, Lap08, MW} for the general theory of Eisenstein series.

    \subsection{Intertwining operators} \label{sec2.2}
	
	Let $\psi_{\R}:\R\to\C^{\times}$ be the unitary character defined by $x\mapsto e^{2\pi\mathrm{i}x}$. Fix an additive character of $\rk\backslash\A$ as the composition
	\be \label{additivechar}
	\psi:\rk\backslash\A\xrightarrow{\tr_{\rk/\BQ}}\BQ\backslash\A_{\BQ}\xrightarrow{\textrm{quotient map}}\BQ\backslash\A_{\BQ}/\widehat{\Z}=\Z\backslash\R\xrightarrow{\psi_{\R}}\C^{\times},
	\ee
	where $\mathbb{A}_{\BQ}$ is the adèle ring of $\BQ$; $\tr_{\rk/\BQ}$ is the trace map; and $\widehat{\Z}$ is the profinite completion of $\Z$. On $\A$ we choose the self-dual measure with respect to $\psi$ so that $\rk\backslash\A$ has total volume $1$. For each place $v$, we fix a Haar measure $\mathrm{d}x_v$ on $\rk_v$ such that
	\begin{itemize}
		\item $\CO_v$ has volume $1$ if $v$ is non-archimedean, where $\CO_v$ denotes the ring of integers of $\rk_v$; and
		\item the product of all the measures $\{\mathrm{d}x_v\}_v$ equals the self-dual Haar measure on $\mathbb{A}$.
	\end{itemize}
For each algebraic subgroup $\RN'$ of $\RN_n$ that is stable under the adjoint action of $\RT_n$, the Haar measure on $\RN'(\A)$ (resp. $\RN'(\rk_v)$) is fixed to be the product of the Haar measures on $\A$ (resp. $\rk_v$). 

Denote by $\RW_n$ the (relative) Weyl group of $\GL_n$ with respect to $\RT_n$, which is also identified with the group of permutation matrices in $\GL_n(\rk)$. Let $k=1,2,\dots,n$. Define 
a permutation matrix
	\[
	w_k=\begin{bmatrix}
		1_{k-1} & & \\
		& & 1\\
		&1_{n-k} & 
	\end{bmatrix}\in \RW_n\subset \GL_n(\rk).
	\]
For all $s\in \C$, define a character
	\begin{eqnarray*}
		\Lambda_{\eta,s}^{(k)}: & \RT_n(\A)& \to \C, \\
		&\mathrm{diag}\left(t_1,t_2,\dots,t_n\right)&\mapsto \eta^{-1}(t_k)\cdot|t_k|_{\A}^{-ns-k+n} \cdot|t_{k+1}|^{-1}_{\A}\cdots|t_n|^{-1}_{\A},
	\end{eqnarray*}
	and a principal series representation
	\begin{equation} \label{IB}
		I_{\eta,s}^{\RB,(k)}:={^{\mathrm{u}}\mathrm{Ind}_{\RB_n(\A)}^{\GL_n(\A)}}\Lambda_{\eta,s}^{(k)}. 
	\end{equation}
	We simply denote \[
    \Lambda_{\eta,s}:=\Lambda_{\eta,s}^{(n)}\quad\textrm{ and }\quad I_{\eta,s}^{\RB}:=I_{\eta,s}^{\RB,(n)}.
    \]
    Note that  
	$$I_{\eta,s}\subset I_{\eta,s}^{\RB}.$$

    Define the   global intertwining operator
	\begin{equation}\label{gi}
		\begin{aligned}
			\mathcal{M}(s,w_k):I_{\eta,s}&\to I_{\eta,s}^{\RB,(k)},\\
			\varphi_s&\mapsto \left(g\mapsto \int_{\RN_n(\A)\cap w_k\overline{\RN}_n(\A)w_k^{-1}}\varphi_s(w_k^{-1}ug)\mathrm{d}u\right),
		\end{aligned}
	\end{equation}
   and for every place $v$ of $\rk$ define   the local 
    intertwining operator
	\begin{equation}\label{gi2}
		\begin{aligned}
			\mathcal{M}_v(s,w_k):I_{\eta,s,v}&\to I_{\eta,s,v}^{\RB,(k)},\\
			\varphi_{s,v}&\mapsto \left(g\mapsto \int_{\RN_n(\rk_v)\cap w_k\overline{\RN}_n(\rk_v)w_k^{-1}}\varphi_{s,v}(w_k^{-1}ug)\mathrm{d}u \right).
		\end{aligned}
	\end{equation}
  Recall that the above integrals absolutely converge when $\Re(s)$ is sufficiently large and admit  meromorphic continuations to the whole complex plane. The maps \eqref{gi} and \eqref{gi2}
	 agree with the inclusion maps when $k=n$.

    Define the normalized local intertwining operator
	\begin{equation}\label{normalize}
		\CN_v(s,w_k):= \frac{\RL(ns,\eta_v)}{\RL(ns-n+k,\eta_v)}\CM_v(s,w_k)
	\end{equation}
so that 
    \be\label{facm}
	\mathcal{M}(s,w_k)=\frac{\RL(ns-n+k,\eta)}{\RL(ns,\eta)}\cdot \otimes_v\mathcal{N}_v(s,w_k).
    \ee
    
	\subsection{The constant term}\label{sec:2.2}
	
 In this subsection we calculate the constant term of $\RE(g;\varphi_s,\eta)$ along the Borel subgroup (Proposition \ref{prp-constant}):
	\[
	\RE_{\mathrm{B}}(g;\varphi_s,\eta):=\int_{\mathrm{N}_n(\mathrm{k})\backslash \mathrm{N}_n(\mathbb{A})}\RE(ug;\varphi_s,\eta)\mathrm{d}u.
	\]
    This is well known to experts (see \cite[II.1.7]{MW} and \cite{Sh10}), but we will repeat the computation following \cite{HM15} to make our presentation self-contained.
    
	We introduce some notation that will be used in the computation. Let $e_i$ ($1\leq i\leq n$) be the algebraic characters of $\RT_n$ defined over $\rk$ given by
	\[
e_i\left(\mathrm{diag}\left(t_1,\dots,t_n\right)\right)=t_i.
	\]
	They form a basis of the abelian group of algebraic characters of $\RT_n$ defined over $\rk$. The sets of positive roots and simple roots for $\GL_n$ are respectively  given by
	\[
	\begin{aligned}
		\Sigma_n^+&:=\{e_i-e_j:i, j =1,2,\dots, n, \, i<j\},\ \textrm{and}\\
		\Delta_n&:=\{\alpha_1:=e_1-e_2,\alpha_2:=e_2-e_3,\dots,\alpha_{n-1}:=e_{n-1}-e_n\}.
	\end{aligned}
	\]

	\begin{lemp}\label{lemct}
	For all $\varphi \in I_{\eta}$ and $g\in \GL_n(\A)$, the following equality (of meromorphic functions in the complex variable $s$) holds: 
        \begin{equation}
		\label{constant2}
		\RE_{\mathrm{B}}(g;\varphi_s,\eta)=\sum_{k=1}^n	(\mathcal{M}(s,w_k)\varphi_s)(g).
	\end{equation} 
    	\end{lemp}
	
	\begin{proof}
		        By the Bruhat decomposition, we have
		\[
		\mathrm{GL}_n=\bigsqcup_{\substack{w\in \mathrm{W}_n\\w(\Delta_n\backslash\{\alpha_{n-1}\})\subset\Sigma_n^+}}\mathrm{P}_nw^{-1}\RB_n=\bigsqcup_{\substack{w\in \RW_n\\w(\Delta_n\backslash\{\alpha_{n-1}\})\subset\Sigma_n^+}}\RP_nw^{-1}(\RN_n\cap w\overline{\RN}_n w^{-1}).
		\]
		Unfolding the Eisenstein series, we find that
		\[
		\begin{aligned}
			\RE_{\RB}(g;\varphi_s,\eta)&=\int_{\RN_n(\mathrm{k})\backslash \RN_n(\mathbb{A})}\sum_{\gamma\in \RP_n(\mathrm{k})\backslash\mathrm{GL}_n(\mathrm{k})}\varphi_s(\gamma ug)\mathrm{d}u\\
			&=\int_{\RN_n(\mathrm{k})\backslash \RN_n(\mathbb{A})}\sum_{\substack{w\in \RW_n\\w(\Delta_n\backslash\{\alpha_{n-1}\})\subset\Sigma_n^+}}\sum_{\gamma\in (\RN_n\cap w\overline{\RN}_nw^{-1})(\mathrm{k})}\varphi_s(w^{-1}\gamma ug)\mathrm{d}u\\
			&=\sum_{\substack{w\in \RW_n\\w(\Delta_n\backslash\{\alpha_{n-1}\})\subset\Sigma_n^+}}\int_{(\RN_n\cap w\RN_nw^{-1})(\mathrm{k})\backslash \RN_n(\mathbb{A})}\varphi_s(w^{-1}ug)\mathrm{d}u.
		\end{aligned}
		\]
        
		The inner integral equals
		\[
		\begin{aligned}
			&\int_{(\RN_n\cap w\RN_nw^{-1})(\mathrm{k})\backslash \RN_n(\mathbb{A})}\varphi_s(w^{-1}ug)\mathrm{d}u\\
			=&\int_{\RN_n(\mathbb{A})\cap w\overline{\RN}_n(\mathbb{A})w^{-1}}\int_{(\RN_n\cap w\RN_nw^{-1})(\mathrm{k})\backslash(\RN_n\cap w\RN_nw^{-1})(\mathbb{A})}\varphi_s(w^{-1}u_1u_2g)\mathrm{d}u_1\mathrm{d}u_2\\
			=&\int_{\RN_n(\mathbb{A})\cap w\overline{\RN}_n(\mathbb{A})w^{-1}}\varphi_s(w^{-1}ug)\mathrm{d}u.
		\end{aligned}
		\]
Since
    \[
\{w\in\RW_n\,:\, w(\Delta_n\backslash\{\alpha_{n-1}\})\subset\Sigma_n^+\}=\{w_1, w_2, \dots, w_n\},
\]
the lemma follows.
	\end{proof}

   Let $v$ be a finite place such that $\eta_v$ is unramified. Denote by $\varphi_{s,v}^{(k)\circ}$ the unique spherical vector of $I_{\eta,s,v}^{\RB,(k)}$ such that $\varphi_{s,v}^{(k)\circ}(1_n)=1$. Write $\varphi_{s,v}^{\circ}:=\varphi_{s,v}^{(n)\circ}$ for simplicity and note that $\varphi_{s,v}^{\circ}\in I_{\eta,s, v}$. 

 \begin{prp} \label{prp-constant}
Let $S$ be a finite set of places including all archimedean places such that $\eta_v$ is unramified for all $v\notin S$. Assume that $\varphi_{s}=\otimes_v\varphi_{s,v}\in I_{\eta,s}$ with $\varphi_{s,v}=\varphi_{s,v}^{\circ}$ for all $v\notin S$. Then
	\begin{equation} \label{constantterm}
		\begin{aligned}
			\RE_{\RB}(\cdot;\varphi_s,\eta)=\sum_{k=1}^n\frac{\RL(ns-n+k,\eta)}{\RL(ns,\eta)}\cdot\left(\otimes_{v\in S}\CN_v(s,w_k)\varphi_{s,v}\right)\otimes \left(\otimes_{v\notin S}\varphi_{s,v}^{(k)\circ}\right).
		\end{aligned}
	\end{equation}
	\end{prp}

\begin{proof}
    We have Langlands' formula (\cite{L71}; see also \cite{Kim05, Sh10})
	\begin{equation}\label{LF}
	\CM_v(s,w_k)\varphi^{\circ}_{s,v}= 	\prod_{\substack{\alpha\in\Sigma_n^+\\w_k(\alpha)\notin\Sigma_n^+}}\frac{\RL(0,\Lambda_{\eta,s,v}\varrho_v^{-1}\circ\alpha^{\vee})}{\RL(1,\Lambda_{\eta,s,v}\varrho_v^{-1}\circ\alpha^{\vee})}\cdot \varphi^{(k)\circ}_{s,v},
	\end{equation}
	where  
    $\alpha^{\vee}$ denotes the coroot of $\alpha$ and $\varrho_v$ is the character
    \[
    \varrho_v:\RT_n(\rk_v)\to\C,\qquad\mathrm{diag}\left(t_1,\dots,t_n\right)\mapsto\prod_{i=1}^n|t_i|_v^{\frac{n+1}{2}-i}.
    \]

    Note that
    \[
    \{\alpha\in\Sigma_n^+\,:\, w_k(\alpha)\notin\Sigma_n^+\}=\{e_k-e_n,\,e_{k+1}-e_n,\,\dots,\,e_{n-1}-e_n\}.
    \]
 	For $\alpha=e_i-e_n$ with $ k\leq i<n$, we have that 
	\[
	\Lambda_{\eta,s,v}\varrho_v^{-1}\circ\alpha^{\vee}=\eta_v\abs{\cdot}_v^{ns-n+i}.
	\]
	Hence
	\[
		\begin{aligned}
			\CM_v(s,w_k)\varphi^{\circ}_{s,v}&= 
			\prod_{i=k}^{n-1}\frac{\RL(ns-n+i,\eta_v)}{\RL(ns-n+i+1,\eta_v)}\cdot\varphi^{(k)\circ}_{s,v}=
			\frac{\RL(ns-n+k,\eta_v)}{\RL(ns,\eta_v)}\cdot\varphi^{(k)\circ}_{s,v},
		\end{aligned}
	\]
	and the global intertwining operator can be written as 
	\[	\begin{aligned}
			\CM(s,w_k)&\varphi_s=\frac{\RL(ns-n+k,\eta)}{\RL(ns,\eta)}\cdot\left(\otimes_{v\in S}\CN_v(s,w_k)\varphi_{s,v}\right)\otimes \left(\otimes_{v\notin S}\varphi_{s,v}^{(k)\circ}\right).
		\end{aligned}
	\]
    This completes the proof of the proposition in view of \eqref{constant2}.
 \end{proof}

\subsection{The Eisenstein map}
    
	Denote by $\mathcal{A}(\GL_n(\rk)\backslash\GL_n(\A))$ the space of smooth automorphic forms (\cite{Gro23}). Both $I_{\eta}$ and $\mathcal{A}(\GL_n(\rk)\backslash\GL_n(\A))$ are equipped with
	the inductive topology in the category of locally convex complex topological vector  spaces and are quasi-Casselman-Wallach representations of $\GL_n(\A)$ (see \cite[Section 3.2]{LS19}). In view of the following proposition, we define an Eisenstein map
		\begin{equation}
		\label{Eisenstein}
		\mathrm{Eis}_{\eta}:I_{\eta}\to\mathcal{A}(\GL_n(\rk)\backslash\GL_n(\A)),\qquad
		\varphi\mapsto\RE(g;\varphi_s,\eta)|_{s=0}.
	\end{equation}

	\begin{prp}\label{prp-Eisenstein}
		The Eisenstein series $\RE(g;\varphi_s,\eta)$ is holomorphic at $s=0$ for all $\varphi\in I_{\eta}$ and $g\in \GL_n(\A)$, and the Eisenstein map \eqref{Eisenstein} defines a $\GL_n(\A)$-equivariant continuous linear embedding from  $I_{\eta}$ into the space $\mathcal{A}(\GL_n(\rk)\backslash\GL_n(\A))$.
	\end{prp}
	
	\begin{proof}
		It is well known that the poles of the Eisenstein series $\RE(g;\varphi_s,\eta)$ coincide with those of the constant term $\RE_{\RB}(g;\varphi_s,\eta)$ (see \cite[Lemma 6.2]{La2} and also \cite[Lemma 2.3]{M07}). 
		The same argument as in the proof of \cite[Lemmas 6.13 and 8.4]{HM15} shows that
        \[
        \CN_v(s,w_k)\varphi_{v,s} \qquad(k=1,2,\dots,n,\ \varphi_v\in I_{\eta,v})
        \]
       is holomorphic at $s=0$, where $\varphi_{v,s}\in I_{\eta,s,v}$ is defined in a similar way to \eqref{varphis}.
        
        Note that the assumption \eqref{etaassumption} implies that $\RL(s,\eta)$ has no pole and $\RL(0,\eta)\neq 0$ (see \cite[Lemma VII. 13.3]{Neu99}), and consequently   the quotient 
		\[
		\frac{\RL(ns-n+k,\eta)}{\RL(ns,\eta)}\qquad (k=1,2,\dots,n)
		\]
		is holomorphic at $s=0$. Therefore by \eqref{facm} and Lemma  \ref{lemct}, the Eisenstein series $\RE(g;\varphi_s,\eta)$ is holomorphic at $s=0$ for all  $\varphi\in I_{\eta}$ and $g\in \GL_n(\A)$. Then it is clear that   \eqref{Eisenstein} is a well-defined $\GL_n(\A)$-equivariant linear map,  which is continuous by \cite[Theorem 2.2]{Lap08}. 
        
     Note that the condition   \eqref{CM} implies that         
 \[
        \Lambda_{\eta,s}^{(i)}\neq \Lambda_{\eta,s}^{(j)} 
        \qquad \textrm{ for all $s\in \C$ and $1\leq i< j\leq n$},
        \]
        and hence 
\[
\bigoplus_{k=1}^nI_{\eta}^{\RB,(k)}\subset\CA(\RB_n(\rk)\RN_n(\A)\backslash\GL_n(\A))\quad(\textrm{the space of smooth automorphic forms}).
\]
In view of Lemma \ref{lemct}, this 
   further implies that the map  \eqref{Eisenstein} is injective. 
	\end{proof}

	From now on, we omit $s$ from the notation when we specialize to the case $s=0$.

	\section{Cohomology for induced representations}
	\label{sec:lie}

	As in the Introduction, let $F_{\eta}$ be the irreducible holomorphic finite-dimensional representation of $\GL_n(\rk\otimes_{\BQ}\C)$ whose infinitesimal character equals that of $I_{\eta,\infty}$, and denote by $F_{\eta}^{\vee}$ its contragredient.  In this section, we study the relative Lie algebra cohomology spaces
	\[
	\RH^q(\mathfrak{g}_{n,\infty},\widetilde{K}_{n,\infty};I_{\eta,\infty}^{\RB,(k)}\otimes F^\vee_{\eta}),\qquad k=1,2,\dots, n, \ q\in \Z.
	\]
	 
	\subsection{Finite-dimensional representations}\label{secfeta}
	
	Recall that for any dominant weight $\mu\in(\Z^n)^{\CE_{\rk}}$, we denote by $F_{\mu}$ the (unique up to isomorphism) irreducible holomorphic finite-dimensional representation of $\GL_n(\rk\otimes_{\BQ}\C)$ of highest weight $\mu$. We realize $F_{\mu}^{\vee}$ as the algebraic induction
	\begin{equation}
		\label{algebraicinduction}
		 	F_{\mu}^{\vee}={^{\mathrm{alg}}\mathrm{Ind}^{\GL_n(\rk\otimes_{\BQ}\C)}_{\RB_n(\rk\otimes_{\BQ}\C)}}\chi_{-\mu}\subset {^{\mathrm{alg}}\mathrm{Ind}^{\GL_n(\rk\otimes_{\BQ}\C)}_{\RN_n(\rk\otimes_{\BQ}\C)}}\C,
	\end{equation}
	where  $\chi_{-\mu}$ is the algebraic character of $\RT_n(\rk\otimes_{\BQ}\C)$ corresponding to the weight $-\mu$. 
    
Define an action
\[
\mathrm{Aut}(\C)\curvearrowright {^{\mathrm{alg}}\mathrm{Ind}^{\GL_n(\rk\otimes_{\BQ}\C)}_{\RN_n(\rk\otimes_{\BQ}\C)}}\C,\quad (\sigma, f)\mapsto {^\sigma f}:=\left(g\mapsto \sigma(f(\sigma^{-1}(g)))\right),
\]
where $\mathrm{Aut}(\C)$ acts on $\GL_n(\rk\otimes_{\BQ}\C)$ through its action on the second factor of $\rk\otimes_{\BQ}\C$. 
This induces a  $\sigma$-linear isomorphism
	\begin{equation}\label{sigmaF}
\sigma:F^{\vee}_{\mu}\to F^{\vee}_{{^{\sigma}\mu}},
	\end{equation}
	where  $\sigma\in\mathrm{Aut}(\C)$, and 
	\[
	{^{\sigma}\mu}:=\{\mu^{\sigma^{-1}\circ\iota}\}_{\iota\in\CE_{\rk}}.
	\]

	Fix an identification 
$F^{\vee}_{^{\sigma}\eta}=F^{\vee}_{^{\sigma}\mu}$, where $\mu=\{\mu^{\iota}\}_{\iota\in\CE_{\rk}}$ is given by
	\[
	\mu^{\iota}=\begin{cases}
		(-\eta_{\iota},0,\dots,0),&\eta_{\iota}\leq 0;\\
		(-1,\dots,-1,n-1-\eta_{\iota}),&\eta_{\iota}\geq n.
	\end{cases}
	\]
	Applying the above discussion to the representation $F_{\eta}^{\vee}$, we obtain a $\sigma$-linear isomorphism
	\begin{equation}\label{sigmaFeta}
		\sigma:F_{\eta}^{\vee}\to F_{{^{\sigma}\eta}}^{\vee}.
	\end{equation}

Fix a highest weight vector 
\begin{equation}\label{vetavee}
v_{\eta}^{\vee}\in (F_{\eta}^{\vee})^{\RN_n(\rk\otimes_\BQ \C)}
\end{equation}
normalized such that $v_{\eta}^{\vee}(w_0)=1$, where $w_0\in\GL_n(\rk)\subset\GL_n(\rk\otimes_{\BQ}\C)$ is the anti-diagonal permutation matrix.

	\subsection{Cohomology for principal series representations}
	
	The cohomology for induced representations can be calculated by Delorme's lemma \cite[Theorem III.3.3]{BW} (see also \cite[Proposition 2.9]{Chen} and \cite[Theorem 9.2.1]{HRbook}). We review the computations in what follows. 
	
	Let $\widetilde{\mathfrak{k}}_{n,\infty}$ be the complexified Lie algebra of $\widetilde{K}_{n,\infty}$ and set $\mathfrak{p}_{n,\infty}:=\mathfrak{g}_{n,\infty}/\widetilde{\mathfrak{k}}_{n,\infty}$. By definition, the relative Lie algebra cohomology 
	\[
	\RH^q(\mathfrak{g}_{n,\infty},\widetilde{K}_{n,\infty};I_{\eta,\infty}^{\RB,(k)}\otimes F^\vee_{\eta})\qquad (k=1,2,\dots,n,\ q\in \Z)
	\]
	is the degree $q$ cohomology of the complex
	\[
	\begin{aligned}
		\cdots\xrightarrow{d}	\,&\mathrm{Hom}_{\widetilde{\mathfrak{k}}_{n,\infty}}(\wedge^{q-1}\mathfrak{p}_{n,\infty}, I_{\eta,\infty}^{\RB,(k)}\otimes F^\vee_{\eta})\\
		\xrightarrow{d}\,&\mathrm{Hom}_{\widetilde{\mathfrak{k}}_{n,\infty}}(\wedge^{q}\mathfrak{p}_{n,\infty}, I_{\eta,\infty}^{\RB,(k)}\otimes F^\vee_{\eta})\\
		\xrightarrow{d}\,&\mathrm{Hom}_{\widetilde{\mathfrak{k}}_{n,\infty}}(\wedge^{q+1}\mathfrak{p}_{n,\infty}, I_{\eta,\infty}^{\RB,(k)}\otimes F^\vee_{\eta})\to\cdots.
	\end{aligned}
	\]
	See \cite[Section I.1]{BW} for the explicit formula of the coboundary map $d$. 
	
	Denote by $\mathfrak{b}_{n,\infty}$, $\mathfrak{n}_{n,\infty}$, $\mathfrak{t}_{n,\infty}$ and $\mathfrak{t}_{n,\infty}^{K}$  the complexified Lie algebras of $\RB_n(\rk_{\infty})$, $\RN_n(\rk_{\infty})$, $\RT_n(\rk_{\infty})$ and $\RT_n(\rk_{\infty})\cap\widetilde{K}_{n,\infty}$ respectively. 
     The obvious isomorphisms 
	\[
	\mathfrak{p}_{n,\infty}\cong\mathfrak{b}_{n,\infty}/\mathfrak{t}_{n,\infty}^{K}\cong\mathfrak{t}_{n,\infty}/\mathfrak{t}_{n,\infty}^{K}\oplus\mathfrak{n}_{n,\infty}
	\]
    yield isomorphisms
    \begin{eqnarray*}
  &&\mathrm{Hom}_{\widetilde{\mathfrak{k}}_{n,\infty}}\left(\wedge^q\mathfrak{p}_{n,\infty}, I_{\eta,\infty}^{\RB,(k)}\otimes F^\vee_{\eta}\right)\\
  &\xrightarrow{\sim}&\mathrm{Hom}_{\mathfrak{t}_{n,\infty}^{K}}\left(\wedge^q(\mathfrak{b}_{n,\infty}/\mathfrak{t}_{n,\infty}^{K}),\Lambda_{\eta,\infty}^{(k)}\otimes F^\vee_{\eta}\right)\qquad (\textrm{by  Frobenius reciprocity})\\
  &\xrightarrow{\sim}&\bigoplus_{0\leq p\leq q}\mathrm{Hom}_{\mathfrak{t}_{n,\infty}^{K}}\left(\wedge^{q-p}(\mathfrak{t}_{n,\infty}/\mathfrak{t}_{n,\infty}^{K}),\mathrm{Hom}(\wedge^{p}\mathfrak{n}_{n,\infty},F^\vee_{\eta})\otimes\Lambda_{\eta,\infty}^{(k)}\right).
    \end{eqnarray*}

 By   Kostant's theorem \cite{Ko61}, $\RH^{p}(\mathfrak{n}_{n,\infty},F^\vee_{\eta})$ is naturally identified as a direct summand  of $\mathrm{Hom}(\wedge^{p}\mathfrak{n}_{n,\infty},F^\vee_{\eta})$. Then by the same proof as that of \cite[Proposition 2.9]{Chen}, the projection map 
    \[
    \mathrm{Hom}(\wedge^{p}\mathfrak{n}_{n,\infty},F^\vee_{\eta})\twoheadrightarrow \RH^{p}(\mathfrak{n}_{n,\infty},F^\vee_{\eta})
    \]
    and the inclusion map 
    \[
   \RH^{p}(\mathfrak{n}_{n,\infty},F^\vee_{\eta})\hookrightarrow  \mathrm{Hom}(\wedge^{p}\mathfrak{n}_{n,\infty},F^\vee_{\eta})
    \]
    respectively induce an isomorphism 
  \begin{equation}\label{generaldelorme}
	\begin{aligned}
		&\RH^{q}(\mathfrak{g}_{n,\infty},\widetilde{K}_{n,\infty};I_{\eta,\infty}^{\RB,(k)}\otimes F^\vee_{\eta})\\
		\xrightarrow{\sim}\bigoplus_{0\leq p\leq q}&\RH^{q-p}(\mathfrak{t}_{n,\infty},\RT_{n}(\rk_{\infty})\cap\widetilde{K}_{n,\infty};\RH^{p}(\mathfrak{n}_{n,\infty},F^\vee_{\eta})\otimes\Lambda_{\eta,\infty}^{(k)})
	\end{aligned}
	\end{equation}
	and its inverse.  In particular, for the bottom degree $q=c_n$, we have an isomorphism 
	\begin{equation}\label{delorme'}
		\RH^{c_n}(\mathfrak{g}_{n,\infty},\widetilde{K}_{n,\infty};I_{\eta,\infty}^{\RB,(k)}\otimes F^\vee_{\eta})\xrightarrow{\sim}\RH^0(\mathfrak{t}_{n,\infty};\RH^{c_n}(\mathfrak{n}_{n,\infty},F^\vee_{\eta})\otimes\Lambda_{\eta,\infty}^{(k)}).
	\end{equation}
	
	As mentioned above, the nilpotent cohomology $\RH^p(\mathfrak{n}_{n,\infty},F^\vee_{\eta})$ is calculated by Kostant's theorem \cite{Ko61} (see also \cite[Section 2.7]{R-imaginary}). We need to make this explicit. 

    For each $\iota\in\CE_{\rk}$, write $\C_{\iota}:=\C$, viewed as a $\rk$-algebra via $\iota$. Denote by $\mathfrak{n}_n^{\iota}$  the Lie algebra of the complex Lie group $\RN_n(\C_{\iota})$. 
    We fix a basis 
	\begin{equation}\label{basis}
		\{e_{i,j}^{\iota}\}_{1\leq i<j\leq n,\,\iota\in\CE_{\rk}}\quad \textrm{ of }\quad \mathfrak{n}_{n,\infty}=\oplus_{\iota\in\CE_{\rk}}\mathfrak{n}_n^{\iota},
	\end{equation}
	where  $e_{i,j}^{\iota}$ is the $n\times n$ elementary matrix with $1$ at the $(i,j)$-th entry and $0$ elsewhere,  viewed as an element of $\mathfrak{n}_{n}^{\iota}$.
	
	Denote by \[
    \RW_{n,\infty}=\prod_{\iota\in\mathcal{E}_{\rk}}\RW_n^{\iota}
    \]
    the (absolute) Weyl group of $\GL_n(\rk_{\infty}\otimes_{\R}\C)=\prod_{\iota\in\CE_{\rk}}\GL_n(\C_{\iota})$ with respect to $\RT_n(\rk_{\infty}\otimes_{\R}\C)=\prod_{\iota\in\CE_{\rk}}\RT_n(\C_{\iota})$, where $\RW_n^{\iota}$ is the Weyl group of $\GL_n(\C_{\iota})$ with respect to $\RT_n(\C_{\iota})$ and is identified with the group of permutation matrices in $\GL_n(\C_{\iota})$. Moreover, we identify the group of permutation matrices in $\GL_n(\C_{\iota})$ with the permutation group on $n$ letters $\{1,\dots,n\}$ as follows. A permutation matrix $w^{\iota}=[w^{\iota}_{i,j}]_{1\leq i,j\leq n}\in\GL_n(\C_{\iota})$ is identified with a permutation $w^{\iota}$ on $n$ letters such that $w^{\iota}(j)=i$ whenever $w^{\iota}_{i,j}=1$. For $w=\{w^{\iota}\}_{\iota\in\mathcal{E}_{\rk}}\in\RW_{n,\infty}$, we set
	\[
	\frak n_w : = \bigoplus_{\iota\in\CE_{\rk}} \bigoplus_{\substack{1\leq i<j\leq n, \\ 
			(w^{\iota})^{-1}(i)>(w^{\iota})^{-1}(j)}} \BC e_{i, j}^{\iota},
	\]
	which is of dimension $l(w)$ (the length of $w$).
		Then Kostant's theorem reads
	\begin{equation}\label{kostant}
		\RH^p(\mathfrak{n}_{n,\infty},F^\vee_{\eta})=\bigoplus_{\substack{w\in\RW_{n,\infty}\\l(w)=p}} \wedge^p \frak{n}_w^* \otimes \BC w.v_{\eta}^{\vee}, 
	\end{equation}
	where the superscript `$\ast$' indicates the dual space, 
		and $v_{\eta}^{\vee}\in F_{\eta}^{\vee}$ is the fixed highest weight vector in \eqref{vetavee}. Here $\n_w$ has a unique $\RT_n(\rk_{\infty}\otimes_{\R}\C)$-stable complement  in  $\n_{n,\infty}$ so that the one-dimensional space $\wedge^p \frak{n}_w^*$ is obviously viewed as a subspace of $\wedge^p \frak{n}_{n,\infty}^*$.

\begin{examplep}
    Let $\rk=\BQ(\sqrt{-1})$ so that $\CE_{\rk}=\{\iota,\overline{\iota}\}$ with $\iota:\BQ(\sqrt{-1})\to\C$ sending $\sqrt{-1}$ to $\sqrt{-1}$. Let $n=3$ and $\eta_{\iota}=4,\eta_{\overline{\iota}}=0$. For the Weyl element
    \[
    w:=\left(\begin{pmatrix}
        1 & 2 & 3
\end{pmatrix},1\right)\in\RW_n^{\iota}\times\RW_n^{\overline{\iota}},
    \]
    we have that 
    \[
    \mathfrak{n}_w=\C e_{1,2}^{\iota}\oplus\C e_{1,3}^{\iota} 
    \]
    and
    \[
    e_{1,2}^{\iota\ast}\wedge e_{1,3}^{\iota\ast}\otimes w.v_{\eta}^{\vee}
    \]
    is a cocycle representing a nonzero class in $\RH^2(\mathfrak{n}_{3,\infty},F_{\eta}^{\vee})$.
\end{examplep}
    
	Let $w^{(k)}:=\{w^{(k)\iota}\}_{\iota\in\mathcal{E}_{\rk}}\in\RW_{n,\infty}$ be the 
	element defined by permutations
	\begin{equation}\label{wk}
	w^{(k)\iota}:=\begin{cases}
		\begin{pmatrix}
			k & k+1 & \cdots & n
		\end{pmatrix},&\eta_{\iota}\leq 0;\\
		\begin{pmatrix}
			1 & 2 & \cdots & k
		\end{pmatrix}^{-1}
		, & \eta_{\iota}\geq n.
	\end{cases}
	\end{equation}
	Then $w^{(k)}$ is the unique  element of $\RW_{n,\infty}$ of length $c_n$ such that $\RT_n(\rk_{\infty})$ acts on the one-dimensional space $\wedge^{c_n}\frak n_{w^{(k)}}^*\otimes \BC w^{(k)}.v_{\eta}^{\vee}$ via $(\Lambda_{\eta,\infty}^{(k)})^{-1}$. Therefore,
	\[
	\RH^{q}(\mathfrak{g}_{n,\infty},\widetilde{K}_{n,\infty};I_{\eta,\infty}^{\RB,(k)}\otimes F^\vee_{\eta})=\{0\}\qquad\text{ if }q<l(w^{(k)})=c_n,
	\]
	and  \eqref{delorme'} reads
	\begin{equation}\label{delorme}
		\begin{aligned}
			\RH^{c_n}(\mathfrak{g}_{n,\infty},\widetilde{K}_{n,\infty};I_{\eta,\infty}^{\RB,(k)}\otimes F^\vee_{\eta})\xrightarrow{\sim} \wedge^{c_n}\frak n_{w^{(k)}}^*\otimes \BC w^{(k)}.v_{\eta}^{\vee}.
					\end{aligned}
	\end{equation}

	\subsection{Rational structures}\label{sec:rs}
	
	We are going to impose a rational structure on the one-dimensional space
	\[
	\RH^{c_n}(\mathfrak{g}_{n,\infty},\widetilde{K}_{n,\infty};I_{\eta,\infty}^{\RB,(k)}\otimes F^\vee_{\eta})\qquad(k=1,2,\dots,n).
	\]
	
	We have an identification $\mathfrak{n}_{n,\infty}=\n_n(\rk)\otimes_\BQ \C$, where $\n_n(\rk)$ denotes the space of $n\times n$ upper-triangular nilpotent matrices with coefficients in $\rk$. The action of  $\mathrm{Aut}(\C)$  on the second factor $\C$  induces an action 
    \be\label{actn}
    \mathrm{Aut}(\C)\curvearrowright \mathfrak{n}_{n,\infty}, \quad (\sigma, x)\mapsto \sigma(x). 
    \ee
     With respect to the fixed basis \eqref{basis} of $\mathfrak{n}_{n,\infty}$, this is explicitly given by
	\[
	\begin{aligned}
		\sigma\left(	\sum_{\substack{1\leq i<j\leq n\\\iota\in\CE_{\rk}}}C_{i,j}^{\iota}e_{i,j}^{\iota}\right)=\sum_{\substack{1\leq i<j\leq n\\\iota\in\CE_{\rk}}}\sigma(C_{i,j}^{\iota})e_{i,j}^{\sigma\circ\iota},\qquad C_{i,j}^{\iota}\in\C.
	\end{aligned}
	\]
	Together with \eqref{sigmaFeta}, we obtain a $\sigma$-linear isomorphism
	\begin{equation}\label{sigmaN}
		\sigma:\RH^q(\mathfrak{n}_{n,\infty},F_{\eta}^{\vee})\to\RH^q(\mathfrak{n}_{n,\infty},F_{{^{\sigma}\eta}}^{\vee})\qquad (q\in\Z)
	\end{equation}
    which sends the cohomology class of a cocycle $\phi\in \Hom_\C(\wedge^q \n_{n,\infty}, F_\eta^\vee)$ to the cohomology class of the composition of 
    \[
      \wedge^q \n_{n,\infty}\xrightarrow{\sigma^{-1}} \wedge^q \n_{n,\infty}\xrightarrow{\phi}F_\eta^\vee \xrightarrow{\sigma} F_{{^{\sigma}\eta}}^{\vee}.
    \]
    	
        We also  have a $\sigma$-linear isomorphism
	\begin{equation}\label{sigmaLambda}
		\sigma:\Lambda^{(k)}_{\eta,\infty}\to{^{\sigma}\Lambda^{(k)}_{\eta,\infty}}
	\end{equation}
	that sends $1$ to $1$. Both \eqref{sigmaN} and \eqref{sigmaLambda} are equivariant with respect to the isomorphism
	\[
	\sigma:\RT_n(\rk\otimes_{\BQ}\C)\to\RT_n(\rk\otimes_{\BQ}\C).
	\]
	Therefore, in view of \eqref{delorme'}, they induce a $\sigma$-linear isomorphism
	\begin{equation}\label{sigmacohomologyinfty}
		\sigma:	\RH^{c_n}(\mathfrak{g}_{n,\infty},\widetilde{K}_{n,\infty};I_{\eta,\infty}^{\RB,(k)}\otimes F^\vee_{\eta})\to	\RH^{c_n}(\mathfrak{g}_{n,\infty},\widetilde{K}_{n,\infty};I_{{^{\sigma}\eta},\infty}^{\RB,(k)}\otimes F^\vee_{{^{\sigma}\eta}}).
	\end{equation}
	
	It is clear  that the diagram
	\begin{equation}\label{delormesigma}
	\begin{CD}
		\RH^{c_n}(\mathfrak{g}_{n,\infty},\widetilde{K}_{n,\infty};I_{\eta,\infty}^{\RB,(k)}\otimes F^\vee_{\eta}) @>\eqref{delorme}>>
		\wedge^{c_n}\frak n_{w^{(k)}}^* \otimes \BC w^{(k)}.v_{\eta}^{\vee} 
		\\
		@V\sigma VV @VV\sigma V\\
		\RH^{c_n}(\mathfrak{g}_{n,\infty},\widetilde{K}_{n,\infty};I_{{^{\sigma}\eta},\infty}^{\RB,(k)}\otimes F^\vee_{{^{\sigma}\eta}}) @>\eqref{delorme}>> \wedge^{c_n}\frak n_{{}^\sigma w^{(k)}}^* \otimes \BC \,{}^\sigma w^{(k)}.v_{{}^\sigma\eta}^{\vee}
	\end{CD}
	\end{equation}
	commutes for all $\sigma\in\mathrm{Aut}(\C)$. Here 
    \begin{equation}\label{w317}
{^{\sigma}w^{(k)}}:=\left\{w^{(k)\sigma^{-1}\circ\iota}\right\}_{\iota\in\CE_{\rk}}\in\RW_{n,\infty}
    \end{equation}
    and the right vertical arrow is induced by \eqref{actn}  and \eqref{sigmaFeta}. 
	
 In view of the canonical isomorphism
	\[
	\imath_{\mathrm{can}}:I_{\eta,\mathrm{f}}^{\RB,(k)}\otimes	\RH^{c_n}(\mathfrak{g}_{n,\infty},\widetilde{K}_{n,\infty};I_{\eta,\infty}^{\RB,(k)}\otimes F^\vee_{\eta})\xrightarrow{\sim}	\RH^{c_n}(\mathfrak{g}_{n,\infty},\widetilde{K}_{n,\infty};I_{\eta}^{\RB,(k)}\otimes F^\vee_{\eta}),
	\]	
	 for every $\sigma\in\mathrm{Aut}(\C)$ we define a $\sigma$-linear isomorphism
	\begin{equation}\label{sigmacohomology}
		\sigma:	\RH^{c_n}(\mathfrak{g}_{n,\infty},\widetilde{K}_{n,\infty};I_{\eta}^{\RB,(k)}\otimes F^\vee_{\eta})\to	\RH^{c_n}(\mathfrak{g}_{n,\infty},\widetilde{K}_{n,\infty};I_{{^{\sigma}\eta}}^{\RB,(k)}\otimes F^\vee_{{^{\sigma}\eta}})
	\end{equation}
	such that the diagram
	\[
	\begin{CD}
		I_{\eta,\mathrm{f}}^{\RB,(k)}\otimes	\RH^{c_n}(\mathfrak{g}_{n,\infty},\widetilde{K}_{n,\infty};I_{\eta,\infty}^{\RB,(k)}\otimes F^\vee_{\eta}) @>\imath_{\mathrm{can}}>> \RH^{c_n}(\mathfrak{g}_{n,\infty},\widetilde{K}_{n,\infty};I_{\eta}^{\RB,(k)}\otimes F^\vee_{\eta})\\
		@V\sigma\otimes \sigma VV @VV\sigma V\\
		I_{{^{\sigma}\eta},\mathrm{f}}^{\RB,(k)}\otimes \RH^{c_n}(\mathfrak{g}_{n,\infty},\widetilde{K}_{n,\infty};I_{{^{\sigma}\eta},\infty}^{\RB,(k)}\otimes F^\vee_{{^{\sigma}\eta}}) @>\imath_{\mathrm{can}}>> \RH^{c_n}(\mathfrak{g}_{n,\infty},\widetilde{K}_{n,\infty};I_{{^{\sigma}\eta}}^{\RB,(k)}\otimes F^\vee_{{^{\sigma}\eta}})
	\end{CD}
	\]
	commutes, where the first map $\sigma$ in the left vertical arrow is given by 
	\begin{equation}\label{sigmaIf}
	\sigma:I_{\eta,\mathrm{f}}^{\RB,(k)}\to I_{{^{\sigma}\eta},\mathrm{f}}^{\RB,(k)},\qquad \varphi\mapsto\sigma\circ\varphi,
	\end{equation}
	and the second $\sigma$ is given in \eqref{sigmacohomologyinfty}. 
    
	\subsection{Cohomology for degenerate principal series representations}\label{dp}
	
	We end this section by discussing the cohomology space 
	\[
	\RH^{c_n}(\mathfrak{g}_{n,\infty},\widetilde{K}_{n,\infty};I_{\eta,\infty}\otimes F_{\eta}^{\vee}).
	\]

	\begin{lemp}\label{isoiib}
		The natural embedding $I_{\eta,\infty}\hookrightarrow I_{\eta,\infty}^{\RB}$ induces an isomorphism
		\[
		\RH^{c_n}(\mathfrak{g}_{n,\infty},\widetilde{K}_{n,\infty};I_{\eta,\infty}\otimes F_{\eta}^{\vee})\cong\RH^{c_n}(\mathfrak{g}_{n,\infty},\widetilde{K}_{n,\infty};I^{\RB}_{\eta,\infty}\otimes F_{\eta}^{\vee}).
		\]
	\end{lemp}
	
	\begin{proof}
		By induction in stages, the embedding $I_{\eta,\infty}\hookrightarrow I_{\eta,\infty}^{\RB}$ can be understood as
		\[
		\begin{aligned}
			I_{\eta,\infty}&={^{\mathrm{u}}\mathrm{Ind}}^{\GL_n(\rk_{\infty})}_{\RP_n(\rk_{\infty})}(\mathbf{1}\otimes\eta_{\infty}^{-1})\hookrightarrow{^{\mathrm{u}}\mathrm{Ind}}^{\GL_n(\rk_{\infty})}_{\RP_n(\rk_{\infty})}\left(({^{\mathrm{u}}\mathrm{Ind}}^{\GL_{n-1}(\rk_{\infty})}_{\RB_{n-1}(\rk_{\infty})}\mathbf{1})\otimes\eta_{\infty}^{-1}\right)=I_{\eta,\infty}^{\RB}.
		\end{aligned}
		\]
		Here the first `$\mathbf{1}$' is the trivial representation of $\GL_{n-1}(\rk_{\infty})$, while the second `$\mathbf{1}$' is the trivial representation of $\RB_{n-1}(\rk_{\infty})$.
		
		Write $\RP_n=\RM_n\RU_n$ for the Levi decomposition of $\RP_n$, where $\RM_n=\GL_{n-1}\times\GL_1$ is the Levi component and $\RU_n$ is the unipotent radical. Denote by $\mathfrak{m}_{n,\infty}$ and $\mathfrak{u}_{n,\infty}$  the complexified Lie algebras of $\RM_n(\rk_{\infty})$ and $\RU_n(\rk_{\infty})$ respectively. As in the preceding discussion of principal series representations,  Delorme's lemma (\cite[Theorem III.3.3]{BW}) gives the following isomorphisms
		\[
		\RH^{c_n}(\mathfrak{g}_{n,\infty},\widetilde{K}_{n,\infty};I_{\eta,\infty}\otimes F_{\eta}^{\vee})\cong\RH^0(\mathfrak{m}_{n,\infty};\RH^{c_n}(\mathfrak{u}_{n,\infty},F_{\eta}^{\vee})\otimes(\mathbf{1}\otimes\eta^{-1}_\infty))
		\]
		and
		\[
		\begin{aligned}
			&\,\RH^{c_n}(\mathfrak{g}_{n,\infty},\widetilde{K}_{n,\infty};I^{\RB}_{\eta,\infty}\otimes F_{\eta}^{\vee})\\
			=&\,\RH^{c_n}\left(\mathfrak{g}_{n,\infty},\widetilde{K}_{n,\infty};{^{\mathrm{u}}\mathrm{Ind}}^{\GL_n(\rk_{\infty})}_{\RP_n(\rk_{\infty})}\left(({^{\mathrm{u}}\mathrm{Ind}}^{\GL_{n-1}(\rk_{\infty})}_{\RB_{n-1}(\rk_{\infty})}\mathbf{1})\otimes\eta_{\infty}^{-1}\right)\otimes F_{\eta}^{\vee}\right)\\
			\cong&\,\RH^0\left(\mathfrak{m}_{n,\infty};\RH^{c_n}(\mathfrak{u}_{n,\infty},F_{\eta}^{\vee})\otimes\left(({^{\mathrm{u}}\mathrm{Ind}}^{\GL_{n-1}(\rk_{\infty})}_{\RB_{n-1}(\rk_{\infty})}\mathbf{1})\otimes\eta^{-1}_\infty\right)\right).
		\end{aligned}
		\]
		Note that both spaces above have dimension one. The inclusion $\mathbf{1}\hookrightarrow{^{\mathrm{u}}\mathrm{Ind}}^{\GL_{n-1}(\rk_{\infty})}_{\RB_{n-1}(\rk_{\infty})}\mathbf{1}$ induces an isomorphism
		\[
		\begin{aligned}
			&\,\RH^0(\mathfrak{m}_{n,\infty};\RH^{c_n}(\mathfrak{u}_{n,\infty},F_{\eta}^{\vee})\otimes(\mathbf{1}\otimes\eta^{-1}_\infty))\\
			\cong&\,\RH^0\left(\mathfrak{m}_{n,\infty};\RH^{c_n}(\mathfrak{u}_{n,\infty},F_{\eta}^{\vee})\otimes\left(({^{\mathrm{u}}\mathrm{Ind}}^{\GL_{n-1}(\rk_{\infty})}_{\RB_{n-1}(\rk_{\infty})}\mathbf{1})\otimes\eta^{-1}_\infty\right)\right)
		\end{aligned}
		\]
		and thus the isomorphism in the lemma.
	\end{proof}
	
	Let $\sigma\in\mathrm{Aut}(\C)$. In view of the above lemma and Section \ref{sec:rs}, we obtain a $\sigma$-linear isomorphism
	\begin{equation}\label{sigmaI}
		\sigma:\RH^{c_n}(\mathfrak{g}_{n,\infty},\widetilde{K}_{n,\infty};I_{\eta,\infty}\otimes F_{\eta}^{\vee})\to\RH^{c_n}(\mathfrak{g}_{n,\infty},\widetilde{K}_{n,\infty};I_{{^{\sigma}\eta},\infty}\otimes F_{{^{\sigma}\eta}}^{\vee}).
	\end{equation}
As in \eqref{sigmaIf}, we have a $\sigma$-linear isomorphism
    \[
    \sigma:I_{\eta,\mathrm{f}}\to I_{{^{\sigma}\eta},\mathrm{f}},\qquad \varphi\mapsto\sigma\circ\varphi.
    \]
    Finally, we define a $\sigma$-linear isomorphism
	\[
		\sigma:\RH^{c_n}(\mathfrak{g}_{n,\infty},\widetilde{K}_{n,\infty};I_{\eta}\otimes F_{\eta}^{\vee})\to\RH^{c_n}(\mathfrak{g}_{n,\infty},\widetilde{K}_{n,\infty};I_{{^{\sigma}\eta}}\otimes F_{{^{\sigma}\eta}}^{\vee})
	\]
	such that the diagram
	\begin{equation}
		\label{diagrameis}
		\begin{CD}
			I_{\eta_{\mathrm{f}}}\otimes\RH^{c_n}(\mathfrak{g}_{n,\infty},\widetilde{K}_{n,\infty};I_{\eta,\infty}\otimes F_{\eta}^{\vee})@>\imath_{\mathrm{can}}>>\RH^{c_n}(\mathfrak{g}_{n,\infty},\widetilde{K}_{n,\infty};I_{\eta}\otimes F_{\eta}^{\vee})\\
			@V\sigma\otimes \sigma VV @VV\sigma V\\
			I_{^{\sigma}\eta_{\mathrm{f}}}\otimes\RH^{c_n}(\mathfrak{g}_{n,\infty},\widetilde{K}_{n,\infty};I_{{^{\sigma}\eta},\infty}\otimes F_{{^{\sigma}\eta}}^{\vee})@>\imath_{\mathrm{can}}>>\RH^{c_n}(\mathfrak{g}_{n,\infty},\widetilde{K}_{n,\infty};I_{{^{\sigma}\eta}}\otimes F_{{^{\sigma}\eta}}^{\vee})
		\end{CD}
	\end{equation}
	commutes.

	\section{The boundary cohomology}\label{sec4}
	
	In this section, we review the sheaf cohomology for the ad\'elic symmetric space of $\GL_n$, especially the structure of the boundary cohomology following \cite[Theorem I]{Har90}, \cite[Section 4.2]{HRbook} and \cite[Section 2.6]{R-imaginary}.

	
	Recall that $\mathcal{X}_n =\GL_n(\rk)\backslash\GL_n(\A)/\widetilde{K}_{n,\infty}$ and $K_{\mathrm{f}}\subset\GL_n(\A_{\mathrm{f}})$ is an open compact subgroup. Let $\overline{\mathcal{X}}_n(K_{\mathrm{f}})$ be the Borel-Serre compactification of $\mathcal{X}_n/K_{\mathrm{f}}$ with the boundary $\partial(\mathcal{X}_n/K_{\mathrm{f}})$ stratified as $\partial(\mathcal{X}_n/K_{\mathrm{f}})=\bigsqcup_{\RP}\partial_{\RP}(\mathcal{X}_n/K_{\mathrm{f}})$, where $\RP$ runs over the $\GL_n(\rk)$-conjugacy classes of proper parabolic subgroups of $\GL_n$ defined over $\rk$.
	
	For any dominant weight $\mu\in(\Z^n)^{\CE_{\rk}}$, the representation $F_{\mu}^{\vee}$ defines a sheaf on $\mathcal{X}_n/K_{\mathrm{f}}$, which is still denoted by $F^\vee_{\mu}$. For each $q\in\Z$, we consider the sheaf cohomology $\RH^q(\mathcal{X}_n/K_{\mathrm{f}},F^\vee_{\mu})$ and the boundary cohomology $\RH^q(\partial_{\RP}(\mathcal{X}_n/K_{\mathrm{f}}),F^\vee_{\mu})$. Passing to the limit, we define
	\be\label{defbcoh}
	\begin{aligned}
		\RH^{q}(\mathcal{X}_n,F^\vee_{\mu})	&:=\lim_{\substack{\longrightarrow\\K_{\mathrm{f}}}}\RH^{q}(\mathcal{X}_n/K_{\mathrm{f}},F^\vee_{\mu}),\\
         \RH^{q}(\partial_{\RP}\mathcal{X}_n,F^\vee_{\mu})	&:=\lim_{\substack{\longrightarrow\\K_{\mathrm{f}}}}\RH^{q}(\partial_{\RP}(\mathcal{X}_n/K_{\mathrm{f}}),F^\vee_{\mu}).
	\end{aligned}
	\ee
	The sheaf cohomology can be calculated by the Betti (singular) complex, and the $\sigma$-linear isomorphism $\sigma:F^\vee_{\mu}\to F^\vee_{^{\sigma}\mu}$ defines a $\sigma$-linear map of the corresponding sheaves which induces $\sigma$-linear maps
	\begin{equation} \label{sheafsigma}
		\begin{aligned}
			\sigma:\RH^{q}(\mathcal{X}_n,F^\vee_{\mu})&\to\RH^{q}(\mathcal{X}_n,F^\vee_{^{\sigma}\mu}),\\
              \sigma:\RH^{q}(\partial_\RP\mathcal{X}_n,F^\vee_{\mu})&\to\RH^{q}(\partial_\RP\mathcal{X}_n,F^\vee_{^{\sigma}\mu}).
		\end{aligned}
	\end{equation}
	Both the domain and the codomain of the maps in \eqref{sheafsigma} are naturally smooth representations of $\GL_n(\A_{\mathrm{f}})$, and the maps in \eqref{sheafsigma} are $\GL_n(\A_{\mathrm{f}})$-equivariant.

	Applying the preceding discussion to the representation $F_{\eta}^{\vee}$, we now consider the boundary cohomology $\RH^q(\partial_{\RB_n}\mathcal{X}_n,F^\vee_{\eta})$ along the Borel stratum. Note that the homotopy equivalence between the underlying spaces induces an identification (see \cite[Section 1.1.1]{Har90} and \cite[Proposition 3.1]{Har82})
	\[
\RH^q(\partial_{\RB_n}\mathcal{X}_n,F^\vee_{\eta})=\lim_{\substack{\longrightarrow\\K_{\mathrm{f}}}}\RH^q(\RB_n(\rk)\backslash\GL_n(\A)/\widetilde{K}_{n,\infty}K_{\mathrm{f}},F^\vee_{\eta}).
	\]
    Calculating $\RH^q(\partial_{\RB_n}\mathcal{X}_n,F^\vee_{\eta})$ via the de Rham resolution yields a canonical isomorphism
\begin{equation}\label{boundaryZ}
\RH^q(\partial_{\RB_n}\mathcal{X}_n,F^\vee_{\eta})\cong\RH^q(\mathfrak{g}_{n,\infty},\widetilde{K}_{n,\infty};\CC^{\infty}(\RB_n(\rk)\backslash\GL_n(\A))\otimes F_{\eta}^{\vee}).
\end{equation}
Here
\[
\CC^{\infty}(\RB_n(\rk)\backslash\GL_n(\A)):=\bigcup_{K_{\mathrm{f}}}\CC^{\infty}(\RB_n(\rk)\backslash\GL_n(\A)/K_{\mathrm{f}})
\]
denotes the space of smooth functions on $\RB_n(\rk)\backslash\GL_n(\A)$, where $K_{\mathrm{f}}$ runs over all open compact subgroups of $\GL_n(\A_{\mathrm{f}})$. We have the following proposition, which is essentially due to \cite[Theorem I]{Har90}. The proof will be included for the convenience of the reader.

\begin{prp}
For every $1\leq k\leq n$, the inclusion $I_{\eta}^{\RB,(k)}\hookrightarrow\CC^{\infty}(\RB_n(\rk)\backslash\GL_n(\A))$ induces an embedding
    \begin{equation}\label{prp4.1}
	\RH^{c_n}(\mathfrak{g}_{n,\infty},\widetilde{K}_{n,\infty};I_{\eta}^{\RB,(k)}\otimes F^\vee_{\eta})\hookrightarrow\RH^{c_n}(\partial_{\RB_n}\CX_n,F^\vee_{\eta}),
	\end{equation}
	which is $\mathrm{Aut}(\C)$-equivariant in the sense that the diagram
	\begin{equation}\label{boundaryautC}
		\begin{CD}
			\RH^{c_n}(\mathfrak{g}_{n,\infty},\widetilde{K}_{n,\infty};I_{\eta}^{\RB,(k)}\otimes F^\vee_{\eta}) @>>>\RH^{c_n}(\partial_{\RB_n}\CX_n,F^\vee_{\eta})\\
			@V\sigma VV @VV\sigma V\\
			\RH^{c_n}(\mathfrak{g}_{n,\infty},\widetilde{K}_{n,\infty};I_{{^{\sigma}\eta}}^{\RB,(k)}\otimes F^\vee_{{^{\sigma}\eta}}) @>>>\RH^{c_n}(\partial_{\RB_n}\CX_n, F^\vee_{{^{\sigma}\eta}})
		\end{CD}
	\end{equation}
	commutes for every $\sigma\in\mathrm{Aut}(\C)$.
\end{prp}

\begin{proof}
By \cite[Theorem I]{Har90} (see also \cite[Proposition 4.2]{HRbook} and \cite[Proposition 2.23]{R-imaginary}), for every $q\in\Z$, we have an identification
	\begin{equation}
		\label{boundarystructure}
		\RH^q(\partial_{\RB_n}\mathcal{X}_n,F^\vee_{\eta})=\bigoplus_{0\leq p\leq q}{^{\mathrm{u}}\mathrm{Ind}^{\GL_n(\A_{\mathrm{f}})}_{\RB_n(\A_{\mathrm{f}})}}\left(\RH^{q-p}(\mathcal{X}_n^{\RT_n},\RH^p(\mathfrak{n}_{n,\infty};F^\vee_{\eta}))\right),
	\end{equation}
	which is $\mathrm{Aut}(\C)$-equivariant in the sense that the diagram
	\begin{equation}\label{sigmaboundary}
		\begin{CD}
			\RH^q(\partial_{\RB_n}\mathcal{X}_n,F^\vee_{\eta}) @>=>>\bigoplus_{0\leq p\leq q}{^{\mathrm{u}}\mathrm{Ind}^{\GL_n(\A_{\mathrm{f}})}_{\RB_n(\A_{\mathrm{f}})}}\left(\RH^{q-p}(\mathcal{X}_n^{\RT_n},\RH^p(\mathfrak{n}_{n,\infty};F^\vee_{\eta}))\right)\\
			@V\sigma VV @VV\sigma V\\
			\RH^q(\partial_{\RB_n}\mathcal{X}_n,F^\vee_{{^{\sigma}\eta}}) @>=>>\bigoplus_{0\leq p\leq q}{^{\mathrm{u}}\mathrm{Ind}^{\GL_n(\A_{\mathrm{f}})}_{\RB_n(\A_{\mathrm{f}})}}\left(\RH^{q-p}(\mathcal{X}_n^{\RT_n},\RH^p(\mathfrak{n}_{n,\infty};F^\vee_{{^{\sigma}\eta}}))\right)
		\end{CD}
	\end{equation}
	commutes for every $\sigma\in \mathrm{Aut}(\C)$. Here 
	\[
	\mathcal{X}_n^{\RT_n}:=\RT_n(\rk)\backslash\RT_n(\A)/(\RT_n(\rk_{\infty})\cap\widetilde{K}_{n,\infty}),
	\]
$\RH^p(\mathfrak{n}_{n,\infty};F^\vee_{{}^\sigma\eta})$ is naturally a representation of $\RT_n(\rk\otimes_{\BQ}\C)$ which defines the cohomology space $\RH^{\ast}(\mathcal{X}_n^{\RT_n},\RH^p(\mathfrak{n}_{n,\infty};F^\vee_{{^{\sigma}\eta}}))$ as in \eqref{defbcoh}, and the vertical arrows are the obvious $\sigma$-linear maps. 

Applying Kostant's theorem \eqref{kostant}, we rewrite \eqref{boundarystructure} as
	\begin{equation} \label{dechq0}
		\RH^q(\partial_{\RB_n}\mathcal{X}_n,F^\vee_{\eta})=\bigoplus_{w\in\RW_{n,\infty}}{^{\mathrm{u}}\mathrm{Ind}^{\GL_n(\A_{\mathrm{f}})}_{\RB_n(\A_{\mathrm{f}})}}\left(\RH^{q-l(w)}(\mathcal{X}_n^{\RT_n}, \wedge^{l(w)}\frak n_w^*\otimes \BC w. v^{\vee}_\eta)
		\right).
	\end{equation}

Recall the element $w^{(k)}\in\RW_{n,\infty}$ of length $c_n$ in \eqref{wk}. We have  
\begin{eqnarray}
    \label{inch00} &&\RH^{c_n}(\partial_{\RB_n}\mathcal{X}_n,F^\vee_{\eta})\\
    \nonumber &\supset & {^{\mathrm{u}}\mathrm{Ind}^{\GL_n(\A_{\mathrm{f}})}_{\RB_n(\A_{\mathrm{f}})}}\left(\RH^{0}(\mathcal{X}_n^{\RT_n},\wedge^{c_n}\frak n_{w^{(k)}}^*\otimes \BC w^{(k)}. v^{\vee}_\eta)\right) \quad (\textrm{by \eqref{dechq0}})\\
   \nonumber &=&\bigoplus_{\Lambda}{^{\mathrm{u}}\mathrm{Ind}^{\GL_n(\A_{\mathrm{f}})}_{\RB_n(\A_{\mathrm{f}})}}\Lambda_{\mathrm{f}}\otimes(\wedge^{c_n}\frak{n}_{w^{(k)}}^*\otimes \BC w^{(k)}. v^{\vee}_\eta)\quad (\textrm{see \cite[Sections 2.5, 2.6]{Har87}})\\
    \nonumber &\cong&\bigoplus_{\Lambda}{^{\mathrm{u}}\mathrm{Ind}^{\GL_n(\A_{\mathrm{f}})}_{\RB_n(\A_{\mathrm{f}})}}\Lambda_{\mathrm{f}}\otimes\RH^{c_n}(\mathfrak{g}_{n,\infty},\widetilde{K}_{n,\infty};{^{\mathrm{u}}\mathrm{Ind}}^{\GL_n(\rk_{\infty})}_{\RB_n(\rk_{\infty})}\Lambda_{\infty}\otimes F^\vee_{\eta})   \quad (\textrm{by \eqref{delorme}})\\
		\nonumber &\cong&\bigoplus_{\Lambda}\RH^{c_n}(\mathfrak{g}_{n,\infty},\widetilde{K}_{n,\infty};{^{\mathrm{u}}\mathrm{Ind}}^{\GL_n(\A)}_{\RB_n(\A)}\Lambda\otimes F^\vee_{\eta})\\
        \nonumber &\supset &\RH^{c_n}(\mathfrak{g}_{n,\infty},\widetilde{K}_{n,\infty};I_{\eta}^{\RB,(k)}\otimes F^\vee_{\eta}),
\end{eqnarray}
where the direct sum is taken over all  
    characters $\Lambda=\Lambda_{\mathrm{f}}\otimes\Lambda_{\infty}:\RT_n(\rk)\backslash\RT_n(\A)\to\C^{\times}$ such that $\Lambda_{\infty}=\Lambda_{\eta,\infty}^{(k)}$.
   
From \cite[Theorem 2.7]{Sch83}, we see that the map \eqref{prp4.1} agrees with the map 
\begin{equation}\label{prp4.1p}
	\RH^{c_n}(\mathfrak{g}_{n,\infty},\widetilde{K}_{n,\infty};I_{\eta}^{\RB,(k)}\otimes F^\vee_{\eta})\hookrightarrow\RH^{c_n}(\partial_{\RB_n}\CX_n,F^\vee_{\eta})
	\end{equation}
obtained from \eqref{inch00}. Thus the map \eqref{prp4.1} is injective, and is $\mathrm{Aut}(\C)$-equivariant in view of \eqref{delormesigma} and \eqref{sigmaboundary}.  
 \end{proof}

	\section{The Eisenstein cohomology}\label{sec:5}

  In view of the identification 
	\begin{equation}\label{sheafLie}
		\RH^{q}(\mathcal{X}_n,F^\vee_{\mu})=\RH^q(\mathfrak{g}_{n,\infty},\widetilde{K}_{n,\infty};\CA(\GL_n(\rk)\backslash\GL_n(\A))\otimes F^\vee_{\mu})\qquad (q\in\Z),
	\end{equation}
	the Eisenstein map \eqref{Eisenstein} induces 
    a linear map 
	\begin{equation}
		\mathrm{Eis}_{\eta}:\RH^{c_n}(\mathfrak{g}_{n,\infty},\widetilde{K}_{n,\infty};I_{\eta}\otimes F^\vee_{\eta})\to\RH^{c_n}(\mathcal{X}_n,F^\vee_{\eta}).
	\end{equation}
    In this section we will prove Theorem \ref{thm:Eis} on rationality of the Eisenstein cohomology.

\subsection{The explicit generator of $\RH^{n-1}(\mathfrak{g}_{n,v},\widetilde{K}_{n,v};I_{\eta,v}\otimes F^\vee_{\eta,v})$} \label{sec:gen}

We write $F_{\eta}^{\vee}=\otimes_{v|\infty}F_{\eta,v}^{\vee}$, where $F_{\eta,v}^{\vee}$ is an irreducible holomorphic representation of $\GL_n(\rk_v\otimes_{\R}\C)$ realized as an algebraic parabolic induction as in \eqref{algebraicinduction}.  Fix an archimedean place $v$ of $\rk$. As in \eqref{dimhcn}, we have that  
  \[ 
	\mathrm{dim}\,\RH^{n-1}(\mathfrak{g}_{n,v},\widetilde{K}_{n,v};I_{\eta,v}\otimes F^\vee_{\eta,v})=1.
	\]
    In what follows we construct an explicit generator 
\[
\kappa_{\eta_v}\in \RH^{n-1}(\mathfrak{g}_{n,v},\widetilde{K}_{n,v};I_{\eta,v}\otimes F^\vee_{\eta,v}).
\]


Let $\iota_v'$ be the unique element in $\{\iota_v, \bar\iota_v\}$ such that $\eta_{\iota_v'}\geq n$. 
 Set 
 \[
\mathfrak{p}_{n,v}:=\mathfrak{g}_{n,v}/\widetilde{\mathfrak{k}}_{n,v},
\]
    where $\widetilde{\mathfrak{k}}_{n,v}$ is the complexified Lie algebra of $\widetilde{K}_{n,v}$. 
Following \cite[Lemmas 4.3, 4.1 and (40)]{DX}, let
    \[
   \tau_{n,v}\subset \wedge^{n-1}\mathfrak{p}_{n,v}, \qquad \tau_{\eta_v}\subset I_{\eta,v},\qquad\textrm{and}\qquad \tau'_{\eta_v}\subset F_{\eta,v}^\vee
    \]
    be the unique irreducible $K_{n,v}$-subrepresentations  of highest weights
    \[
 (\iota_v',\dots,\iota_v', (\iota_v')^{1-n}), \ \,  ((\iota_v')^0,\dots,(\iota_v')^0,(\iota_v')^{\eta_{\overline{\iota_v'}}\,-\eta_{\iota_v'}}), \ \, \textrm{and}\ \, 
 ((\iota_v')^{\eta_{\iota_v'}-\eta_{\overline{\iota_v'}}\,+1-n},\iota_v',\dots,\iota_v'),
    \]
    respectively. 

Write $\C_{\iota_v'}:=\C$ (resp. $\C_{\overline{\iota_v'}}:=\C$), viewed as a $\rk_v$-algebra via $\iota_v'$ (resp. $\overline{\iota_v'}$). Note that $\tau'_{\eta_v}$ is isomorphic to the Cartan product of
	$F_{\eta,\iota_v'}^{\vee}|_{K_{n,v}}$ and $F_{\eta,\overline{\iota_v'}}^{\vee}|_{K_{n,v}}$, where $F_{\eta,\iota_v'}$ and  $F_{\eta,\overline{\iota_v'}}$ are irreducible holomorphic finite-dimensional representations of   $\GL_n(\C_{\iota_v'})$ and  $\GL_n(\C_{\overline{\iota_v'}})$ respectively such that $F_{\eta,v}^{\vee}=F_{\eta,\iota_v'}^{\vee}\otimes F_{\eta,\overline{\iota_v'}}^{\ve}$. 
     Also note that $\tau_{n,v}$ is isomorphic to the PRV-component (see \cite{PRV67}) of $\tau_{\eta_v}\otimes \tau_{\eta_v}'$, and 
     \be\label{dim11}
     \dim \mathrm{Hom}_{K_{n,v}}(\tau_{n,v},\tau_{\eta_v}\otimes \tau_{\eta_v}')=\dim \mathrm{Hom}_{K_{n,v}}(\tau_{n,v},\tau_{\eta_v}\otimes F_{\eta_v}^\vee)=1.
     \ee
     
Define a character  
    \[
	\eta^{\circ}_{v}:=(\iota_v')^n: \rk_v^\times \rightarrow \C^\times,
    \]
    and set 
    \[
    I_{\eta_v^\circ}:={}^{\mathrm{u}}\mathrm{Ind}_{\RP_n(\rk_v)}^{\GL_n(\rk_v)} (\mathbf{1} \otimes (\eta_v^\circ)^{-1}).
   \]
   Denote by $F_{\eta^\circ_v}$ the coefficient system of $I_{\eta_v^\circ}$. Then all the preceding discussions apply to $\eta_v^\circ$.

  We have a Jantzen--Zuckerman translation map (see \cite{ Jan79, Zu77,Vo81,VZ84})
		\[
		\jmath_{\eta,v}:I_{\eta_v^\circ}\otimes F^\vee_{\eta_v^\circ}\to I_{\eta,v}\otimes F^\vee_{\eta,v}
		\]
        explicitly constructed in \cite[Section 2.3]{local}. It is an injective homomorphism of representations of $\GL_n(\rk_v)$ which induces   an isomorphism	\[	\jmath_{\eta,v}:\RH^{n-1}(\mathfrak{g}_{n,v},\widetilde{K}_{n,v};I_{\eta_v^\circ}\otimes F^\vee_{\eta_v^\circ})\xrightarrow{\sim}\RH^{n-1}(\mathfrak{g}_{n,v},\widetilde{K}_{n,v};I_{\eta,v}\otimes F^\vee_{\eta,v}).	\]

\begin{lemp}\label{cohomi}
    There is a well-defined linear isomorphism 
    \[
    \begin{array}{rcl}
    \mathrm{Hom}_{K_{n,v}}(\tau_{n,v},\tau_{\eta_v}\otimes \tau_{\eta_v}')&\rightarrow& \RH^{n-1}(\mathfrak{g}_{n,v},\widetilde{K}_{n,v};I_{\eta,v}\otimes F^\vee_{\eta,v}),\smallskip\\
    \phi&\mapsto & \textrm{the cohomology class of $\tilde \phi$ },
    \end{array}
    \]
    where 
    $\tilde \phi$ is the composition 
    \[
    \tilde \phi: \wedge^{n-1}\mathfrak{p}_{n,v}\xrightarrow{\textrm{the $K_{n,v}$-equivariant projection}}\tau_{n,v}\xrightarrow{\phi} \tau_{\eta_v}\otimes \tau_{\eta_v}'\xrightarrow{\subset} I_{\eta,v}\otimes F^\vee_{\eta,v}.
    \]
	\end{lemp}
	
	\begin{proof}
     First consider the special case when 
    $\eta_v=\eta^{\circ}_{v}$.
   Note that $I_{\eta_v^\circ}$
   becomes unitarizable after twisting by the character $\abs{\det}_v^{\frac{1}{2}}$. Therefore in this special case, the lemma holds by \cite[Proposition 9.4.3]{Wallach}.

Note that \[
F^\vee_{\eta^\circ_v}=\tau_{\eta^\circ_v}'={\det}_{\iota_v'}\otimes 1_{\overline{\iota_v'}}\qquad \textrm{and}\qquad \tau_{\eta^\circ_v}\otimes \tau_{\eta^\circ_v}'\cong \tau_{n,v},
\]
where $\det_{\iota_v'}$  is the determinant character of $\GL_n(\C_{\iota_v'})$ and $1_{\overline{\iota_v'}}$ is the trivial representation of $\GL_n(\C_{\overline{\iota_v'}})$.  

    In general, the construction of $\jmath_{\eta,v}$ implies that
    \[
    \jmath_{\eta,v}(\tau_{\eta_v^\circ}\otimes F_{\eta_v^\circ}^{\vee})\subset\tau_{\eta_v}\otimes F_{\eta_v}^{\vee}
    \]
     which, in view of \eqref{dim11},  further implies that 
   \[
    \jmath_{\eta,v}(\tau_{\eta_v^\circ}\otimes \tau_{\eta_v^\circ}')\subset\tau_{\eta_v}\otimes \tau_{\eta_v}'.
    \]

   The lemma then follows by considering the commutative diagram
\[
\begin{CD}
		\tau_{n,v}@>\phi_0>>	\tau_{\eta_v^\circ}\otimes \tau_{\eta_v^\circ}' @>\subset >>I_{\eta_v^\circ}\otimes F^\vee_{\eta_v^\circ}\\
			@ V = VV @V\jmath_{\eta,v} VV @VV\jmath_{\eta,v} V\\
		\tau_{n,v}@>\phi >>	\tau_{\eta_v}\otimes \tau_{\eta_v}' @>\subset >>I_{\eta,v}\otimes F^\vee_{\eta,v}
		\end{CD},
    \]
    where $\phi_0$ is the map that makes the left square commutative. 
	\end{proof}

\begin{remarkp}\label{gr1}
The lemma above was originally stated in \cite[Proposition 8]{Gr}. We follow the idea in that paper, which uses the Jantzen--Zuckerman translation functor (see \cite{Jan79, Zu77,Vo81,VZ84} for further details). However, constructing the translation map requires more work than was done in \cite{Gr}. In Grenié's proof (paragraph 2 on p.~296), he selects an extremal weight vector $v_s$ of a certain weight and claims that it is the desired eigenvector for the action of $P(\C)$. This is not generally valid, because an extremal weight vector need not be a highest weight vector. The correct translation map is now constructed in \cite{local}.
\end{remarkp}

We have the obvious identification  
    \[
    \mathfrak{p}_{n,v}=\mathfrak{t}_{n,v}/\mathfrak{t}_{n,v}^{K}\oplus\mathfrak{n}_{n,v},
    \]
    where $\mathfrak{n}_{n,v}$, $\mathfrak{t}_{n,v}$, and $\mathfrak{t}_{n,v}^{K}$ are the complexified Lie algebras of $\RN_n(\rk_v)$, $\RT_n(\rk_v)$,  and $\RT_n(\rk_v)\cap\widetilde{K}_{n,v}$, respectively. Note that $\mathfrak{n}_{n,v}=\mathfrak{n}_{n}^{\iota_v'}\oplus\mathfrak{n}^{\overline{\iota_v'}}_{n}$, and
	\[
	e_{1,n}^{\iota_v'}\wedge e_{2,n}^{\iota_v'}\wedge\dots\wedge e_{n-1,n}^{\iota_v'}\in\wedge^{n-1}\mathfrak{n}^{\iota_v'}_{n}\subset\wedge^{n-1}\mathfrak{n}_{n,v}\subset \wedge^{n-1}\mathfrak{p}_{n,v}
	\]
	    is a highest weight vector of $\tau_{n,v}$.

Let
	\[
	\Upsilon(\eta_v):=\{\beta=(\beta_1,\dots,\beta_n)\in\Z_{\geq 0}^n:\beta_1+\dots+\beta_n=\eta_{\iota_v'}-\eta_{\overline{\iota_v'}}\}
	\]
	and for every $\beta\in\Upsilon(\eta_v)$ we define $\varphi_{\beta}\in I_{\eta,v}$ by
	\be \label{phibeta}
	\varphi_{\beta}(g):=\frac{\overline{\iota_v'}(g_{n,1})^{\beta_1}\cdots\overline{\iota_v'}(g_{n,n})^{\beta_n}}{\left(|g_{n,1}|_v+\dots+|g_{n,n}|_v\right)^{\eta_{\iota_v'}}},\qquad g=[g_{i,j}]_{1\leq i,j\leq n}\in\GL_n(\rk_v).
	\ee
	Then $\{\varphi_{\beta}\}_{\beta\in\Upsilon(\eta_v)}$ forms a basis of $\tau_{\eta_v}$ (see \cite[page 731]{IM22}). Specifically, $\varphi_{\beta_0}$ with 
    \[
    \beta_0:=(0,\dots,0,\eta_{\iota_v'}-\eta_{\overline{\iota_v'}})
    \]
    is a highest weight vector in $\tau_{\eta_v}$. Fix a highest weight vector $v^{\vee}_{\eta_v}\in F^{\vee}_{\eta,v}$ in a similar way to \eqref{vetavee}. 
    Since $\tau_{\eta_v}$ is isomorphic to the Cartan component of $\tau_{n,v}\otimes\tau'^{\vee}_{\eta_v}$, there is a unique highest weight vector  of the PRV-component 
		of $\tau_{\eta_v}\otimes \tau_{\eta_v}'$ of the form
		\begin{equation}\label{lwv}
			\Psi_{\eta_v}:=\varphi_{\beta_0}\otimes w_v^{(n)}.v^{\vee}_{\eta_v}+\sum_{\beta\in\Upsilon(\eta_v),\, \beta\neq \beta_0}\varphi_{\beta}\otimes v_{\beta},\qquad v_{\beta}\in F_{\eta,v}^{\vee}.
		\end{equation}

In view of Lemma \ref{cohomi}, define a generator
	\[
	\kappa_{\eta_v}\in\mathrm{Hom}_{K_{n,v}}(\tau_{n,v},\tau_{\eta_v}\otimes \tau_{\eta_v}' )=\RH^{n-1}(\mathfrak{g}_{n,v},\widetilde{K}_{n,v};I_{\eta,v}\otimes F^\vee_{\eta,v})
	\]
	 such that
	\begin{equation}\label{kappa}
		\kappa_{\eta_v}\left(e_{1,n}^{\iota_v'}\wedge e_{2,n}^{\iota_v'}\wedge\dots\wedge e_{n-1,n}^{\iota_v'}\right)=\Psi_{\eta_v}.
	\end{equation}
Specializing to the case when $\eta_v=\eta^\circ_v$, we obtain  a generator
\[
\kappa_{\eta_v^\circ}\in\mathrm{Hom}_{K_{n,v}}(\tau_{n,v},\tau_{\eta^\circ_v}\otimes \tau_{\eta_v^\circ}' )=\RH^{n-1}(\mathfrak{g}_{n,v},\widetilde{K}_{n,v};I_{\eta_v^\circ}\otimes F^\vee_{\eta_v^\circ}).
\]

\begin{lemp}\label{lem:ind-trans}
    It holds that $\jmath_{\eta,v}(\kappa_{\eta_v^\circ}) = \kappa_{\eta_v}$.
\end{lemp}

\begin{proof}
    The explicit construction of $\jmath_{\eta,v}$ in \cite[Section 2.3]{local} shows that 
\[\jmath_{\eta,v}(\Psi_{\eta_v^\circ}) = \Psi_{\eta_v},
\]
which implies the lemma.
\end{proof}

\subsection{The explicit generator of $\RH^{c_n}(\mathfrak{g}_{n,\infty},\widetilde{K}_{n,\infty};I_{\eta,\infty}\otimes F^\vee_{\eta})$} \label{sec5.2}

Recall from the Introduction that $\rk_1$ is the maximal CM subfield of $\rk$, and we have fixed compatible CM types of $\rk_1$ and $\rk$. The field $\rk_1$ is a totally  imaginary quadratic extension of the maximal totally real subfield $\rk_0$ in $\rk$. Denote by $\CE_{\rk_0}$ and $\CE_{\rk_1}$ the sets of field embeddings $\rk_0\hookrightarrow\C$ and $\rk_1\hookrightarrow\C$, respectively.

Fix total orders $\prec$ on $\CE_{\rk_0}$ 
and $\CE_{\rk}$ such that 
\begin{itemize}
    \item 
    the restriction map $\CE_{\rk}\to \CE_{\rk_0}$ is order-preserving, namely  for all $\iota, \iota'\in \CE_{\rk}$,
    \[ 
    \iota \prec \iota'\quad \textrm{implies}\quad \iota|_{\rk_0}=\iota'|_{\rk_0}\textrm{ or } \iota|_{\rk_0}\prec \iota'|_{\rk_0};
    \]
    \item $\iota_v\prec \bar\iota_v$ for every archimedean place $v$ of $\rk$, and there is no $\iota\in \CE_\rk$ such that $\iota_v\prec \iota \prec \bar\iota_v$.
    \end{itemize}
For every $\tau\in\CE_{\rk_1}$, we denote 
\[
\CE_{\rk}(\tau):=\{\iota\in\CE_{\rk}:\iota|_{\rk_1}=\tau\}.
\]
Then it is clear that the map 
\[
	\CE_{\rk}(\tau)\to\CE_{\rk}(\overline{\tau}),\qquad \iota\mapsto\overline{\iota}
	\]
    is order-preserving.

Note that the set $\{v| \infty\}$ of archimedean places of $\rk$ is naturally in bijection with  $\{\iota_v :  v|\infty\}$. We endow $\{v |\infty\}$ with the total order induced from that on $\{\iota_v :  v|\infty\}$. With respect to this fixed order, the Künneth formula yields an identification
\[
\RH^{c_n}(\mathfrak{g}_{n,\infty},\widetilde{K}_{n,\infty};I_{\eta,\infty}\otimes F^\vee_{\eta})
=
\bigotimes_{v \mid \infty} \RH^{n-1}(\mathfrak{g}_{n,v},\widetilde{K}_{n,v};I_{\eta,v}\otimes F^\vee_{\eta,v}).
\]
We further define a generator
    \begin{equation}
    \label{generator'}
    \begin{aligned}
\kappa_{\eta}:=\otimes_{v|\infty}\kappa_{\eta_v}\in\,&\bigotimes_{v|\infty}\RH^{n-1}(\mathfrak{g}_{n,v},\widetilde{K}_{n,v};I_{\eta,v}\otimes F^\vee_{\eta,v})\\
        =\,&\RH^{c_n}(\mathfrak{g}_{n,\infty},\widetilde{K}_{n,\infty};I_{\eta,\infty}\otimes F^\vee_{\eta}).
    \end{aligned}
    \end{equation}

Put $\eta^\circ:=\otimes_{v\mid\infty} \eta^\circ_v$, $I_{\eta^\circ}:=\widehat\otimes_{v\mid\infty} I_{\eta_v^\circ}$, $F_{\eta^\circ}:=\otimes_{v\mid\infty} F_{\eta_v^\circ}$, and
\begin{equation}\label{generator0}
\begin{aligned}
\kappa_{\eta^{\circ}}:=\otimes_{v|\infty}\kappa_{\eta^{\circ}_v}\in\,&\bigotimes_{v|\infty}\RH^{n-1}(\mathfrak{g}_{n,v},\widetilde{K}_{n,v};I_{\eta^{\circ}_v}\otimes F^\vee_{\eta^{\circ}_v}).
    \end{aligned}
\end{equation}
We have a translation map
        \[
        \jmath_\eta :=\otimes_{v\mid\infty}\jmath_{\eta,v}: I_{\eta^\circ}\otimes F_{\eta^\circ}^\vee \to I_{\eta,\infty}\otimes F_{\eta}^\vee,
        \]
and we will also write $\jmath_\eta$ for the induced isomorphism between the corresponding cohomology spaces.  Then by Lemma \ref{lem:ind-trans} we have 
   \be \label{ind-trans}
   \jmath_{\eta}(\kappa_{\eta^{\circ}})=\kappa_{\eta}.
\ee

Denote by $\{e_{i,j}^{\iota\ast}\}_{1\leq i<j\leq n,\iota\in\CE_{\rk}}$ the dual basis of $\mathfrak{n}_{n,\infty}^{\ast}$ with respect to the basis $\{e_{i,j}^{\iota}\}_{1\leq i<j\leq n,\iota\in\CE_{\rk}}$ of $\mathfrak{n}_{n,\infty}$. Recall from \eqref{wk} 
the element 
\[
  w^{(n)}:=\{w^{(n)\iota}\}_{\iota\in\mathcal{E}_{\rk}}=\{w^{(n)}_v\}_{v\mid \infty}\in\RW_{n,\infty}. 
\]
Using \eqref{kappa}, the following lemma is routine to verify. 

    \begin{lemp}\label{kappaeta}
        Under the isomorphism (see \eqref{delorme}) 
        \[
\RH^{c_n}(\mathfrak{g}_{n,\infty},\widetilde{K}_{n,\infty};I_{\eta,\infty}\otimes F^\vee_{\eta})\xrightarrow{\sim} \wedge^{c_n}\frak n_{ w^{(n)}}^*\otimes \BC w^{(n)}.v_{\eta}^{\vee},
\]
$\kappa_\eta$ corresponds to the element 
\[
  \omega(\kappa_\eta):=\bigwedge_{\substack{\iota\in\CE_{\rk}\\\eta_{\iota}\geq n}}\left(e_{1,n}^{\iota\ast}\wedge\cdots\wedge e^{\iota\ast}_{n-1,n}\right)\otimes w^{(n)}.v_{\eta}^{\vee}\in\wedge^{c_n}\mathfrak{n}^{\ast}_{w^{(n)}}\otimes\C w^{(n)}.v_{\eta}^{\vee},
\]
where the wedge products are taken with respect to the fixed total order of $\CE_{\rk}$ . 
    \end{lemp} 

\subsection{$\mathrm{Aut}(\C)$-action on $\kappa_{\eta}$}

Let $\sigma\in\mathrm{Aut}(\C)$. We have defined in \eqref{sigmaI} a $\sigma$-linear map 
\[
\sigma:\RH^{c_n}(\mathfrak{g}_{n,\infty},\widetilde{K}_{n,\infty};I_{\eta,\infty}\otimes F^\vee_{\eta})\to\RH^{c_n}(\mathfrak{g}_{n,\infty},\widetilde{K}_{n,\infty};I_{^\sigma\eta,\infty}\otimes F^\vee_{^\sigma\eta}).
\]
In this subsection, we describe the image $\sigma(\kappa_{\eta})$ of our fixed generator
\[
\kappa_{\eta}\in\RH^{c_n}(\mathfrak{g}_{n,\infty},\widetilde{K}_{n,\infty};I_{\eta,\infty}\otimes F^\vee_{\eta}).
\]

Note that the diagram
\[
\begin{CD}
	\CE_{\rk} @>>>\CE_{\rk_1}@>>>\CE_{\rk_0}\\
	@V\sigma\circ(\cdot) VV  @V \sigma\circ(\cdot)VV @VV \sigma\circ(\cdot) V\\
	\CE_{\rk} @>>>\CE_{\rk_1}@>>>\CE_{\rk_0}\\
\end{CD}
\]
commutes for every $\sigma\in\mathrm{Aut}(\C)$, where the horizontal arrows are the restrictions (which are obviously surjective). We view $\sigma\in\mathrm{Aut}(\C)$ as the permutation $\sigma\circ(\cdot)$ on $\CE_{\rk}$, and (uniquely) decompose it as a product $\sigma=\sigma_1\sigma_2$ of permutations such that
\begin{itemize}
	\item $\sigma_1$ descends to the permutation $\sigma\circ(\cdot)$ on $\CE_{\rk_1}$, and induces an order-preserving bijection $\sigma_1:\CE_{\rk}(\tau)\mapsto\CE_{\rk}(\sigma\circ\tau)$ for each $\tau\in\CE_{\rk_1}$;
	\item $\sigma_2$ descends to the trivial permutation on $\CE_{\rk_1}$.
\end{itemize}
Denote by $p_0(\sigma)$ the permutation number of the permutation $\sigma\circ(\cdot)$ on the totally ordered set $\CE_{\rk_0}$. For every $\tau\in\CE_{\rk_1}$, we define $p(\sigma,\tau)$ to be the permutation number of $\sigma_2$ on the totally ordered set $\CE_{\rk}(\tau)$.

For all number field extensions $\rk''/\rk'$, we define the discriminant of $\rk''$ over $\rk'$ to be
    \[
\delta_{\rk''/\rk'}:=\det\left[\mathrm{tr}_{\rk''/\rk'}(x_ix_j)\right]_{1\leq i,j\leq [\rk'':\rk']}\in\rk'^{\times},
    \]
where $\{x_i\}_{1\leq i\leq [\rk'':\rk']}$ is a fixed basis of $\rk''$ over $\rk'$, and $\mathrm{tr}_{\rk''/\rk'}$ is the trace map. Up to multiplication by scalars in $(\rk'^{\times})^2$, this is independent of the choice of the basis.

Define 
	\begin{equation}\label{Delta}
		\Delta_{\rk}:=\sqrt{\RN_{\rk_0/\BQ}(\delta_{\rk_1/\rk_0})}^{[\rk:\rk_1]},\qquad\nabla_{\rk}:=\sqrt{\RN_{\rk_1/\BQ}(\delta_{\rk/\rk_1})},
	\end{equation}
    where $\RN_{\rk_0/\BQ}:\rk_0\rightarrow\BQ$ and  $\RN_{\rk_1/\BQ}:\rk_1\to\BQ$ denote the norm maps.

Denote by $\delta_{\rk_0}$ and $\delta_{\rk}$ the integral discriminants of $\rk_0$ and $\rk$ respectively. By \cite[Lemma 5.27]{R-imaginary}, we have that
\begin{equation}\label{discriminant}
	|\delta_{\rk}|^{\frac{1}{2}}=c\cdot\mathrm{i}^{\frac{[\rk:\BQ]}{2}}\cdot\Delta_{\rk}\cdot\nabla_{\rk}
\end{equation}
for some $c\in\BQ^{\times}$. Fix a basis $\alpha_1,\alpha_2,\dots,\alpha_{[\rk_0:\BQ]}$ of $\rk_0$ over $\BQ$. We form a matrix
\[
D_0:=\begin{bmatrix}
    \tau_i(\alpha_j)
\end{bmatrix}_{1\leq i,j\leq [\rk_0:\BQ]},\qquad\text{where}\qquad\CE_{\rk_0}=\left\{\tau_1\prec\tau_2\prec\cdots\prec\tau_{[\rk_0:\BQ]}\right\},
\]
and define
\[
\delta_0:=\det(D_0)\in\C^{\times}.
\]
Fix a basis $\beta_1,\beta_2,\dots,\beta_{[\rk:\rk_1]}$ of $\rk$ over $\rk_1$. For each $\tau\in\CE_{\rk_1}$, we form a matrix
\[
D(\rk/\rk_1,\tau):=\begin{bmatrix}
	\iota_i(\beta_j)
\end{bmatrix}_{1\leq i,j\leq [\rk:\rk_1]},\qquad\text{where}\qquad\mathcal{E}_{\rk}(\tau)=\{\iota_1\prec\iota_2\prec\cdots\prec\iota_{[\rk:\rk_1]}\},
\]
and define
\[
\delta(\rk/\rk_1,\tau):=\det\left(D(\rk/\rk_1,\tau)\right)\in\C^{\times}.
\]
It is easy to see that
\begin{equation}\label{dk0}
\frac{\sigma(\delta_0)}{\delta_0}=(-1)^{p_0(\sigma)},\qquad\text{and}\qquad \frac{\sigma(\delta(\rk/\rk_1,\tau))}{\delta(\rk/\rk_1,\sigma\circ\tau)}=(-1)^{p(\sigma,\tau)},\,\tau\in\CE_{\rk_1}.
\end{equation}
By \cite[Sublemma 5.35]{R-imaginary}, we have that
\begin{equation}\label{nabla}
	\frac{\sigma(\nabla_{\rk})}{\nabla_{\rk}}=(-1)^{\sum_{\tau\in\CE_{\rk_1}}p(\sigma,\tau)}.
\end{equation}

 For every $\tau\in\CE_{\rk_1}$, set
\[
\eta_{\tau}:=\eta_{\iota}\in\Z, \quad \textrm{where }\ \iota\in\CE_{\rk}(\tau).
\]
As a consequence of the purity lemma (\cite{We55} and \cite[Lemma 7]{R-Hecke}), this is independent of the choice of $\iota \in \CE_{\rk}(\tau)$. The regularity condition \eqref{regular} implies that \[
\textrm{$\eta_{\tau}\leq 0\quad $ or $\quad \eta_{\tau}\geq n$. }
\]
Define
\begin{equation}\label{omegaeta}
\delta_{\infty}(\eta):=\delta^{[\rk:\rk_1]}_{0}\cdot\prod_{\substack{\tau\in\CE_{\rk_1}\\\eta_{\tau}\geq n}}\delta(\rk/\rk_1,\tau).
\end{equation}

\begin{remarkp}
    Define $\delta'_{\infty}(\eta)$ as above with respect to another basis $\alpha_1',\alpha_2',\dots,\alpha_{[\rk_0:\BQ]}'$ of $\rk_0$ over $\BQ$ and another basis $\beta_1',\beta_2',\dots,\beta_{[\rk:\rk_1]}'$ of $\rk$. Then there are matrices $A\in\GL_{[\rk_0:\BQ]}(\BQ)$ and $B\in\GL_{[\rk:\rk_1]}(\rk_1)$ such that
    \[
    \begin{aligned}
        \begin{bmatrix}
        \alpha'_1 & \alpha'_2 & \cdots & \alpha'_{[\rk_0:\BQ]}
    \end{bmatrix}&=\begin{bmatrix}
        \alpha_1 & \alpha_2 & \cdots & \alpha_{[\rk_0:\BQ]}
    \end{bmatrix}A,\\
    \begin{bmatrix}
        \beta'_1 & \beta'_2 & \cdots & \beta'_{[\rk:\rk_1]}
    \end{bmatrix}&=\begin{bmatrix}
        \beta_1 & \beta_2 & \cdots & \beta_{[\rk:\rk_1]}
    \end{bmatrix}B,
    \end{aligned}
    \]
    and 
    \[
    \delta_{\infty}'(\eta)=\delta_{\eta}\cdot\delta_{\infty}(\eta),\qquad \delta_{\eta}:=(\det A)^{[\rk:\rk_1]}\cdot\prod_{\substack{
\tau\in\CE_{\rk_1}\\
\eta_{\tau}\geq n
    }}\tau(\det B).
    \]
    We have that $\delta_{\eta}\in\BQ(\eta)^{\times}$ and $\sigma(\delta_{\eta})=\delta_{^\sigma\eta}$ for every $\sigma\in\Aut(\C)$.
\end{remarkp}

\begin{examplep}
    Let $\rk=\rk_1=\BQ(\sqrt{-1},\sqrt{5})$ and $\CE_{\rk}:=\{\iota_1\prec\iota_2\prec\iota_3\prec\iota_4\}$ with
    \[
    \begin{aligned}
        &\iota_1:\sqrt{-1}\mapsto\sqrt{-1},&\ &\sqrt{5}\mapsto\sqrt{5},\\
    &\iota_2:\sqrt{-1}\mapsto-\sqrt{-1},&\ &\sqrt{5}\mapsto\sqrt{5},\\
    &\iota_3:\sqrt{-1}\mapsto\sqrt{-1},&\ & \sqrt{5}\mapsto-\sqrt{5},\\
    &\iota_4:\sqrt{-1}\mapsto-\sqrt{-1},&\ & \sqrt{5}\mapsto-\sqrt{5}.
    \end{aligned}
    \]
    Then $\rk_0=\BQ(\sqrt{5})$ with a basis $1,\frac{1+\sqrt{5}}{2}$ of $\rk_0$ over $\BQ$. Let $n=2$ and $\eta$ be a character with
    \[
    \eta_{\iota_1}=4,\qquad\eta_{\iota_2}=0,\qquad\eta_{\iota_3}=4,\qquad\eta_{\iota_4}=0.
    \]
    Then we have
    \[
    \delta_{\infty}(\eta)=-\sqrt{5}.
    \]
\end{examplep}

\begin{lemp}\label{sigmakappaeta}
	For every $\sigma\in\mathrm{Aut}(\C)$, it holds that 
	\[
	\sigma\left(\frac{\kappa_{\eta}}{\delta_{\infty}(\eta)^{n-1}}\right)=\frac{\kappa_{^\sigma\eta}}{\delta_{\infty}(^\sigma\eta)^{n-1}}.
	\]
\end{lemp}

\begin{proof}
Recall the element ${^\sigma w^{(n)}}$ from \eqref{w317}.  
In the notation of Lemma \ref{kappaeta}, we have that 
\[
\sigma(\omega(\kappa_{\eta}))=\bigwedge_{\substack{\iota\in\CE_{\rk}\\\eta_{\iota}\geq n}}\left(e_{1,n}^{\sigma\circ \iota \ast}\wedge\cdots\wedge e^{\sigma\circ \iota\ast}_{n-1,n}\right)\otimes {^\sigma w^{(n)}}.v_{^\sigma\eta}^{\vee},
	\]
	which differs from $\omega(\kappa_{^\sigma\eta})$ by a sign. To calculate this sign, we view $\sigma$ as a permutation on $\CE_{\rk}$ and decompose it as a product $\sigma=\sigma_1\sigma_2$ of permutations as before. It is straightforward to check that
	\[
	\bigwedge_{\substack{\iota\in\CE_{\rk}\\\eta_{\iota}\geq n}}\left(e_{1,n}^{\sigma_1(\iota)\ast}\wedge\cdots\wedge e^{\sigma_1(\iota)\ast}_{n-1,n}\right)=(-1)^{(n-1)[\rk:\rk_1]p_0(\sigma)}\cdot\bigwedge_{\substack{\iota\in\CE_{\rk}\\\eta_{\sigma_1^{-1}(\iota)}\geq n}}\left(e_{1,n}^{\iota\ast}\wedge\cdots\wedge e^{\iota\ast}_{n-1,n}\right)
	\]
	and
	\[
	\bigwedge_{\substack{\iota\in\CE_{\rk}\\\eta_{\iota}\geq n}}\left(e_{1,n}^{\sigma_2(\iota)\ast}\wedge\cdots\wedge e^{\sigma_2(\iota)\ast}_{n-1,n}\right)=\prod_{\substack{\tau\in\CE_{\rk_1}\\\eta_{\tau}\geq n}}(-1)^{(n-1)p(\sigma,\tau)}\cdot\bigwedge_{\substack{\iota\in\CE_{\rk}\\\eta_{\sigma_2^{-1}(\iota)}\geq n}}\left(e_{1,n}^{\iota\ast}\wedge\cdots\wedge e^{\iota\ast}_{n-1,n}\right).
	\]
	Therefore,
	\[
	\sigma(\omega(\kappa_{\eta}))=(-1)^{(n-1)[\rk:\rk_1]p_0(\sigma)}\cdot\prod_{\substack{\tau\in\CE_{\rk_1}\\\eta_{\tau}\geq n}}(-1)^{(n-1)p(\sigma,\tau)}\cdot\omega(\kappa_{^\sigma\eta}),
	\]
	and the lemma follows by \eqref{dk0} and \eqref{delormesigma}.
\end{proof}

	\subsection{Cohomological intertwining operators}
	
	Recall the local normalized intertwining operator
    \[
    \CN_v(w_k):=\CN_v(0,w_k):I_{\eta,v}\to I_{\eta,v}^{\RB,(k)}\qquad (1\leq k\leq n)
    \]
	defined in \eqref{normalize}. Put
    \[
    \CN(w_k):=\otimes_{v}\CN_v(w_k),\qquad\CN_{\mathrm{f}}(w_k):=\otimes_{v\nmid\infty}\CN_v(w_k),\qquad\CN_{\infty}(w_k):=\otimes_{v|\infty}\CN_v(w_k).
    \]
    We further define
\[
\begin{aligned}
    \CN'(w_k)&:=(\mathrm{i}^{\frac{[\rk:\BQ]}{2}}\cdot\Delta_{\rk})^{k-n}\cdot\CN(w_k):I_{\eta}\to I_{\eta}^{\RB,(k)},\\
    \CN'_{\infty}(w_k)&:=(\mathrm{i}^{\frac{[\rk:\BQ]}{2}}\cdot\Delta_{\rk})^{k-n}\cdot\CN_{\infty}(w_k):I_{\eta,\infty}\to I_{\eta,\infty}^{\RB,(k)}.
\end{aligned}
\]
	On the cohomological level, the normalized intertwining operators $\CN(w_k)$ and $\CN_{\infty}(w_k)$ induce
	\begin{equation}\label{cohomologyintertwine}
    \begin{aligned}
        \mathcal{N}(w_k):	&\RH^{c_n}(\mathfrak{g}_{n,\infty},\widetilde{K}_{n,\infty};I_{\eta}\otimes F^\vee_{\eta})\to	\RH^{c_n}(\mathfrak{g}_{n,\infty},\widetilde{K}_{n,\infty};I_{\eta}^{\RB,(k)}\otimes F^\vee_{\eta}),\\
        \mathcal{N}_{\infty}(w_k):	&\RH^{c_n}(\mathfrak{g}_{n,\infty},\widetilde{K}_{n,\infty};I_{\eta,\infty}\otimes F^\vee_{\eta})\to	\RH^{c_n}(\mathfrak{g}_{n,\infty},\widetilde{K}_{n,\infty};I_{\eta,\infty}^{\RB,(k)}\otimes F^\vee_{\eta}),
    \end{aligned}
	\end{equation}
and similarly for $\CN'(w_k)$ and  $\CN_{\infty}'(w_k)$. This subsection is devoted to a proof of the following proposition.

	\begin{prp}
		\label{prop-intertwine}
		The map $\CN'(w_k)$ is $\mathrm{Aut}(\C)$-equivariant. That is, the diagram
		\begin{equation}
			\begin{CD}
				\RH^{c_n}(\mathfrak{g}_{n,\infty},\widetilde{K}_{n,\infty};I_{\eta}\otimes F^\vee_{\eta}) @>\CN'(w_k)>>\RH^{c_n}(\mathfrak{g}_{n,\infty},\widetilde{K}_{n,\infty};I_{\eta}^{\RB,(k)}\otimes F^\vee_{\eta})\\
				@V\sigma VV @VV\sigma V\\
				\RH^{c_n}(\mathfrak{g}_{n,\infty},\widetilde{K}_{n,\infty};I_{{^{\sigma}\eta}}\otimes F^\vee_{{^{\sigma}\eta}}) @>\CN'(w_k)>>\RH^{c_n}(\mathfrak{g}_{n,\infty},\widetilde{K}_{n,\infty};I_{{^{\sigma}\eta}}^{\RB,(k)}\otimes F^\vee_{{^{\sigma}\eta}})
			\end{CD}
		\end{equation}
		commutes for every $\sigma\in\mathrm{Aut}(\C)$.
	\end{prp}

	It suffices to show the  $\mathrm{Aut}(\C)$-equivariance of
	\begin{equation}\label{archimedean}
		\CN'_{\infty}(w_k):\RH^{c_n}(\mathfrak{g}_{n,\infty},\widetilde{K}_{n,\infty};I_{\eta,\infty}\otimes F^\vee_{\eta})\to\RH^{c_n}(\mathfrak{g}_{n,\infty},\widetilde{K}_{n,\infty};I_{\eta,\infty}^{\RB,(k)}\otimes F^\vee_{\eta})
	\end{equation}
	and
	\begin{equation}\label{nonarchimedean}
		\CN_{\mathrm{f}}(w_k):I_{\eta,\mathrm{f}}\to I_{\eta,\mathrm{f}}^{\RB,(k)}.
	\end{equation}

    The $\mathrm{Aut}(\mathbb{C})$-equivariance of \eqref{nonarchimedean} follows by the same argument as in \cite[Section 7.3.2.1]{HRbook}, using Waldspurger's rationality result \cite[Theorem IV.1.1]{Wa03}. It remains to prove the $\mathrm{Aut}(\mathbb{C})$-equivariance of \eqref{archimedean}.


	\begin{lemp}\label{lem-intertwine}
   Let $v$ be an archimedean place of $\rk$.  For all $\beta\in\Upsilon(\eta_v)$ and $k=1,2,\dots,n$, it holds that 
		\[
		(\mathcal{N}_v(w_k)\varphi_{\beta})(w_k) =  \begin{cases}
			c_v^{n-k},\quad& \textrm{if }\beta=\beta_0:=(0,\dots,0,\eta_{\iota_v'}-\eta_{\overline{\iota_v'}}),\\
			0,\quad& \textrm{otherwise,}
		\end{cases}
		\]
		where
\be \label{cv}
c_v:= \frac{1}{2\pi}\mathrm{vol}\left(\{ x\in \rk_v: |x|_v\leq 1\}\right),
\ee
and `$\,\mathrm{vol}$' indicates the volume with respect to the fixed Haar measure of $\rk_v$. 
	\end{lemp}
	
	\begin{proof}
For each $s\in \C$, define $\varphi_{\beta,s}\in I_{\eta,s,v}$ by 
\[
\varphi_{\beta,s}(g):=\frac{\overline{\iota_v'}(g_{n,1})^{\beta_1}\cdots\overline{\iota_v'}(g_{n,n})^{\beta_n}}{\left(|g_{n,1}|_v+\cdots+|g_{n,n}|_v\right)^{\eta_{\iota_v'}+ns}},\qquad   g=[g_{i,j}]_{1\leq i,j\leq n}\in\GL_n(\rk_v),
\]
so that 
\[
(\mathcal{N}_v(w_k)\varphi_{\beta})(w_k)=(\mathcal{N}_v(s,w_k)\varphi_{\beta,s})(w_k)|_{s=0}.
\]

By definition, we have that
		\[
		(\mathcal{M}_v(s,w_k)\varphi_{\beta,s})(w_k)=\int_{\rk_v}\dots\int_{\rk_v}\varphi_{\beta,s}\begin{pmatrix}
			1_{k-1}\\
			& 1\\
			& & \ddots \\
			& & & 1\\
			& u_k & \cdots & u_{n-1} & 1	\end{pmatrix}\mathrm{d}u_k\cdots\mathrm{d}u_{n-1}
		\]
        when $\Re(s)$ is sufficiently large such that the integral converges absolutely. Write $\beta=(\beta_1,\beta_2, \dots, \beta_n)$. 
       It is clear that the above integrand vanishes unless $\beta_1=\dots=\beta_{k-1}=0$. Assume that $\beta_1=\dots=\beta_{k-1}=0$. Then 
		\[
		(\mathcal{M}_v(s,w_k)\varphi_{\beta,s})(w_k)=\int_{\rk_v}\dots\int_{\rk_v}\frac{\overline{\iota_v'}(u_k)^{\beta_k}\cdots\overline{\iota_v'}(u_{n-1})^{\beta_{n-1}}}{\left(|u_k|_v+\cdots+|u_{n-1}|_v+1\right)^{\eta_{\iota_v'}+ns}}\mathrm{d}u_k\cdots\mathrm{d}u_{n-1}.
		\]
		Using polar coordinates $\overline{\iota_v'}(u_k)=re^{\mathrm{i}\theta}$, we have 
		\[
        \begin{aligned}
         & \int_{\rk_v}\frac{\overline{\iota_v'}(u_k)^{\beta_k}}{\left(|u_k|_v+\cdots+|u_{n-1}|_v+1\right)^{\eta_{\iota_v'}+ns}}\mathrm{d}u_k \\
	    =\, & c_v\cdot\int_0^{2\pi}\int_0^{\infty}\frac{(re^{\mathrm{i}\theta})^{\beta_k}\cdot 2r\mathrm{d}r\mathrm{d}\theta}{(r^2+|u_{k+1}|_v+\dots+|u_{n-1}|_v+1)^{\eta_{\iota_v'}+ns}} \\
        =\, & \begin{cases} c_v\cdot\dfrac{\RL(ns-1,\eta_v)}{\RL(ns,\eta_v)}\cdot\dfrac{1}{\left(|u_{k+1}|_v+\dots+|u_{n-1}|_v+1\right)^{\eta_{\iota_v'}-1+ns}}, & \text{if }\beta_k=0,\\
        0, & \text{if }\beta_k>0.
        \end{cases}
        \end{aligned}
        \]
           Repeating the same calculation, we find that 
		\[
			(\mathcal{M}_v(s,w_k)\varphi_{\beta,s})(w_k)= \begin{cases} c_v^{n-k}\cdot\dfrac{\RL(ns+k-n,\eta_v)}{\RL(ns,\eta_v)}, & \text{if }\beta = \beta_0, \\
            0, & \text{otherwise}.
	  \end{cases}
		\]
		This proves the lemma.
	\end{proof}

Recall the generator $\kappa_{\eta}$ defined in \eqref{generator'}. For $k=1,2,\dots,n$, we define
\[
\kappa_{\eta}^{(k)}:=\mathcal{N}_{\infty}(w_k)(\kappa_{\eta})\in\RH^{c_n}(\mathfrak{g}_{n,\infty},\widetilde{K}_{n,\infty};I^{\RB,(k)}_{\eta,\infty}\otimes F^\vee_{\eta}).
\]
Note that $\kappa_{\eta}^{(n)}=\kappa_{\eta}$, under the identification given by Lemma \ref{isoiib}.

    \begin{lemp}\label{kappaetak}
   For $k=1,2,\dots,n$, the image of $\kappa^{(k)}_{\eta}$ under \eqref{delorme} is
  \[
\begin{aligned}
   \omega(\kappa^{(k)}_{\eta})
   :& = |\delta_{\rk}|^{\frac{k-n}{2}}\cdot\bigwedge_{\substack{\iota\in\CE_{\rk}\\\eta_{\iota}\geq n}}\left( e_{1,k}^{\iota\ast}\wedge\cdots\wedge e_{k-1,k}^{\iota\ast}\wedge e_{k,k+1}^{\bar\iota\ast}\wedge\cdots\wedge e_{k,n}^{\bar\iota\ast}\right)\otimes w^{(k)}.v_{\eta}^{\vee}\\
   & \in\wedge^{c_n}\frak{n}_{w^{(k)}}^*\otimes \C w^{(k)}.v_{\eta}^{\vee},
\end{aligned}
\]
where the wedge product is taken with respect to the fixed total order of $\CE_{\rk}$. 
    \end{lemp}

\begin{proof}
For each archimedean place $v$ of $\rk$, let $\kappa_{\eta_v}$ be the generator defined by \eqref{kappa}, and define
    \[
    \kappa_{\eta_v}^{(k)}:=\CN_v(w_k)(\kappa_{\eta_v})\in\RH^{n-1}(\mathfrak{g}_{n,v},\widetilde{K}_{n,v};I_{\eta,v}^{\RB,(k)}\otimes F^\vee_{\eta,v})
    \]
    to be the image of $\kappa_{\eta_v}$ under the intertwining operator
    \[
    \CN_v(w_k):\RH^{n-1}(\mathfrak{g}_{n,v},\widetilde{K}_{n,v};I_{\eta,v}\otimes F^\vee_{\eta,v})\to\RH^{n-1}(\mathfrak{g}_{n,v},\widetilde{K}_{n,v};I_{\eta,v}^{\RB,(k)}\otimes F^\vee_{\eta,v}).
    \]
    Then $\kappa_{\eta}^{(k)}=\otimes_{v|\infty}\kappa_{\eta_v}^{(k)}$. By Lemma \ref{cohomi} and \eqref{kappa},
    $\kappa_{\eta_v}^{(k)}$ is represented by a cocycle 
    \[
    \tilde\kappa_{\eta_v}^{(k)}\in\mathrm{Hom}_{K_{n,v}}\left(\wedge^{n-1}\mathfrak{p}_{n,v},I_{\eta,v}^{\RB,(k)}\otimes F^\vee_{\eta,v}\right)
    \]
    with the property that
    \begin{equation}\label{kappa'}
		\begin{aligned}
		    & \tilde\kappa_{\eta_v}^{(k)}\left(e_{1,n}^{\iota_v'}\wedge e_{2,n}^{\iota_v'}\wedge\dots\wedge e_{n-1,n}^{\iota_v'}\right)\\
            =\, &\CN_v(w_k)\varphi_{\beta_0}\otimes w_v^{(n)}.v^{\vee}_{\eta_v}+\sum_{\beta\in\Upsilon(\eta_v),\,\beta\neq\beta_0}\CN_v(w_k)\varphi_{\beta}\otimes v_{\beta}.
		\end{aligned}
	\end{equation}

Note that
\[
\bigwedge_{\substack{\iota\in\CE_{\rk}\\\eta_{\iota}\geq n}}\left( e_{1,k}^{\iota\ast}\wedge\cdots\wedge e_{k-1,k}^{\iota\ast}\wedge e_{k,k+1}^{\bar\iota\ast}\wedge\cdots\wedge e_{k,n}^{\bar\iota\ast}\right)\otimes w^{(k)}.v_{\eta}^{\vee}
\]
is a generator of $\wedge^{c_n}\frak{n}_{w^{(k)}}^*\otimes \C w^{(k)}.v_{\eta}^{\vee}$ and we calculate that
\[
\begin{aligned}
		    & \tilde\kappa_{\eta_v}^{(k)}\left( e_{1,k}^{\iota_v'}\wedge\cdots\wedge e_{k-1,k}^{\iota_v'}\wedge e_{k,k+1}^{\overline{\iota_v'}}\wedge\cdots\wedge e_{k,n}^{\overline{\iota_v'}}\right)\\
            =\, &\tilde\kappa_{\eta_v}^{(k)}\left(w_k.(e_{1,n}^{\iota_v'}\wedge e_{2,n}^{\iota_v'}\wedge\dots\wedge e_{n-1,n}^{\iota_v'})\right)\\
            =\, &w_k.(\CN_v(w_k)\varphi_{\beta_0})\otimes w_kw_v^{(n)}.v^{\vee}_{\eta_v}+\sum_{\beta\in\Upsilon(\eta_v),\,\beta\neq\beta_0}w_k.(\CN_v(w_k)\varphi_{\beta})\otimes w_k.v_{\beta}.
		\end{aligned}
\]
Using Lemma~\ref{lem-intertwine}, it is routine to verify that the image of $\kappa^{(k)}_{\eta}$ under \eqref{delorme} is
\[
\begin{aligned}
\prod_{v|\infty}c_v^{n-k}\cdot\bigwedge_{\substack{\iota\in\CE_{\rk}\\\eta_{\iota}\geq n}}\left( e_{1,k}^{\iota\ast}\wedge\cdots\wedge e_{k-1,k}^{\iota\ast}\wedge e_{k,k+1}^{\bar\iota\ast}\wedge\cdots\wedge e_{k,n}^{\bar\iota\ast}\right)\otimes w^{(k)}.v_{\eta}^{\vee},
\end{aligned}
\]
where $c_v$ is defined in \eqref{cv}.

Recall the measure on $\rk_v$ ($v|\infty$) fixed in Section~\ref{sec2.2}. A direct computation gives
\[
\prod_{v\mid\infty} c_v = |\delta_{\rk}|^{-\frac{1}{2}}.
\]
This completes the proof of the lemma.
\end{proof}

\begin{lemp}\label{sigmakappaetak}
	For every $\sigma\in\mathrm{Aut}(\C)$, it holds that 
	\[
	\sigma\left(\frac{|\delta_{\rk}|^{\frac{n-k}{2}}\cdot\kappa^{(k)}_{\eta}}{\nabla_{\rk}^{n-k}\cdot\delta_{\infty}(\eta)^{n-1}}\right)=\frac{|\delta_{\rk}|^{\frac{n-k}{2}}\cdot\kappa^{(k)}_{^\sigma\eta}}{\nabla_{\rk}^{n-k}\cdot\delta_{\infty}(^\sigma\eta)^{n-1}}.
	\]
\end{lemp}

\begin{proof}
    Using Lemma \ref{kappaetak}, the same computation as in the proof of Lemma \ref{sigmakappaeta} shows that
	\[
    \begin{aligned}
	\sigma(|\delta_{\rk}|^{\frac{n-k}{2}}\cdot\kappa_{\eta}^{(k)})=(-1)^{(n-1)[\rk:\rk_1]p_0(\sigma)}\cdot&\prod_{\substack{\tau\in\CE_{\rk_1}\\\eta_{\tau}\geq n}}(-1)^{(n-1)p(\sigma,\tau)}\cdot\\
    \cdot&\prod_{\tau\in\CE_{\rk_1}}(-1)^{(n-k)p(\sigma,\tau)}\cdot|\delta_{\rk}|^{\frac{n-k}{2}}\cdot\kappa_{^\sigma\eta}^{(k)}.
    \end{aligned}
	\]
	The lemma then follows by \eqref{dk0}, \eqref{nabla} and \eqref{delormesigma}.
\end{proof}

Comparing Lemmas \ref{sigmakappaeta} and \ref{sigmakappaetak}, and using \eqref{discriminant}, we obtain that the diagram
\[
	\begin{CD}
	\RH^{c_n}(\mathfrak{g}_{n,\infty},\widetilde{K}_{n,\infty};I_{\eta,\infty}\otimes F^\vee_{\eta}) @>\CN_{\infty}'(w_k)>>\RH^{c_n}(\mathfrak{g}_{n,\infty},\widetilde{K}_{n,\infty};I_{\eta,\infty}^{\RB,(k)}\otimes F^\vee_{\eta})\\
	@V\sigma VV @VV\sigma V\\
	\RH^{c_n}(\mathfrak{g}_{n,\infty},\widetilde{K}_{n,\infty};I_{{^{\sigma}\eta,\infty}}\otimes F^\vee_{{^{\sigma}\eta}}) @>\CN_{\infty}'(w_k)>>\RH^{c_n}(\mathfrak{g}_{n,\infty},\widetilde{K}_{n,\infty};I_{{^{\sigma}\eta,\infty}}^{\RB,(k)}\otimes F^\vee_{{^{\sigma}\eta}})
\end{CD}
\]
commutes for every $\sigma\in\mathrm{Aut}(\C)$, which completes the proof of Proposition \ref{prop-intertwine}.

	\subsection{Rationality of the Eisenstein cohomology}
	
Consider the composition of
    \[
	\mathrm{Eis}_{\eta}:\RH^{c_n}(\mathfrak{g}_{n,\infty},\widetilde{K}_{n,\infty};I_{\eta}\otimes F^\vee_{\eta})\to\RH^{c_n}(\mathcal{X}_n,F^\vee_{\eta})
    \]
 with the restriction map
	\begin{equation}\label{rbn}
		r_{\RB_n}:\RH^{c_n}(\mathcal{X}_n,F^\vee_{\eta})\to\RH^{c_n}(\partial_{\RB_n}\mathcal{X}_n,F^\vee_{\eta}).
	\end{equation}
	
	\begin{prp}\label{mainprop}
		The diagram
		\begin{equation}\label{diagramproof}
			\begin{CD}
				\RH^{c_n}(\mathfrak{g}_{n,\infty},\widetilde{K}_{n,\infty};I_{\eta}\otimes F^\vee_{\eta}) @>r_{\RB_n}\circ\mathrm{Eis}_{\eta}>>\RH^{c_n}(\partial_{\RB_n}\mathcal{X}_n,F^\vee_{\eta})\\
				@V\sigma VV @VV\sigma V\\
				\RH^{c_n}(\mathfrak{g}_{n,\infty},\widetilde{K}_{n,\infty};I_{^{\sigma}\eta}\otimes F^\vee_{^{\sigma}\eta}) @>r_{\RB_n}\circ\mathrm{Eis}_{^{\sigma}\eta}>>\RH^{c_n}(\partial_{\RB_n}\mathcal{X}_n,F^\vee_{^{\sigma}\eta})
			\end{CD}
		\end{equation}
		commutes for every $\sigma\in\mathrm{Aut}(\C)$.
	\end{prp}
	
	\begin{proof}

Define a map
\begin{equation}\label{318A}
\begin{aligned}
I_{\eta}&\to\CA(\RB_n(\rk)\RN_n(\A)\backslash\GL_n(\A)),\\
\varphi&\mapsto\left(g\mapsto\int_{\RN_n(\rk)\backslash\RN_n(\A)}\mathrm{Eis}_{\eta}(\varphi)(ug)\mathrm{d}u\right).
\end{aligned}
\end{equation}
It follows from Lemma \ref{lemct} that its image lies in 
\[
\bigoplus_{k=1}^nI_{\eta}^{\RB,(k)}\subset\CA(\RB_n(\rk)\RN_n(\A)\backslash\GL_n(\A)),
\]
and hence \eqref{318A} induces a map
\begin{equation}\label{318B}
\RH^{c_n}(\mathfrak{g}_{n,\infty},\widetilde{K}_{n,\infty};I_{\eta}\otimes F^\vee_{\eta})\to\bigoplus_{k=1}^n\RH^{c_n}(\mathfrak{g}_{n,\infty},\widetilde{K}_{n,\infty};I_{\eta}^{\RB,(k)}\otimes F^\vee_{\eta}).
\end{equation}

By \cite[Satz 1.10]{Sch83}, the map $r_{\RB_n}\circ\mathrm{Eis}_{\eta}$ in \eqref{diagramproof} equals the composition of \eqref{318B} with the map
\[
		\bigoplus_{k=1}^n\RH^{c_n}(\mathfrak{g}_{n,\infty},\widetilde{K}_{n,\infty};I_{\eta}^{\RB,(k)}\otimes F^\vee_{\eta})\to\RH^{c_n}(\partial_{\RB_n}\mathcal{X}_n,F_{\eta}^{\vee})
		\]
        induced by \eqref{boundaryZ}. In view of \eqref{boundaryautC}, it suffices to show the $\mathrm{Aut}(\C)$-equivariance of \eqref{318B}. 
		
		Recall from the proof of Proposition \ref{prp-Eisenstein} that the L-function $\RL(s,\eta)$  has no pole and $\RL(0,\eta)\neq 0$. By the constant term formula \eqref{constantterm}, the $k$-th component of the map \eqref{318B} equals
        \[
        (\mathrm{i}^{\frac{[\rk:\BQ]}{2}}\cdot\Delta_{\rk})^{n-k}\frac{\RL(k-n,\eta)}{\RL(0,\eta)}\cdot\CN'(w_k).
        \]
We have proved in Proposition \ref{prop-intertwine} that $\mathcal{N}'(w_k)$ is $\mathrm{Aut}(\C)$-equivariant. Note that $0,-1,\dots,1-n$ are critical places for the Hecke L-function $\RL(s,\eta)$. By a theorem of Harder \cite{Har87} (see also \cite[Theorem 25]{R-Hecke}) we have
		\[
		\sigma\left((\mathrm{i}^{\frac{[\rk:\BQ]}{2}}\cdot\Delta_{\rk})^{k-n}\frac{\RL(k-n,\eta)}{\RL(0,\eta)}\right)=(\mathrm{i}^{\frac{[\rk:\BQ]}{2}}\cdot\Delta_{\rk})^{k-n}\frac{\RL(k-n,{^{\sigma}\eta})}{\RL(0,{^{\sigma}\eta})}.
		\]
        This proves the $\mathrm{Aut}(\C)$-equivariance of \eqref{318B}, which completes the proof of the proposition.
	\end{proof}

    We have the following lemma from the proof of \cite[Theorem II]{Har90}.

    \begin{lemp}\label{lem-Harder}
    Under the assumption \eqref{etaassumption}, it holds that
    \[
    \mathrm{Hom}_{\GL_n(\A_{\mathrm{f}})}(I_{\eta,\mathrm{f}},\mathrm{ker}(r_{\RB_n}))=\{0\},
    \]
    where $r_{\RB_n}$ is the restriction map \eqref{rbn}.
    \end{lemp}

    \begin{proof}
     Let 
    \[
    \Phi:I_{\eta,\rm f}\to\RH^{c_n}(\CX_n,F_{\eta}^{\vee})
    \]
    be a non-zero $\GL_n(\A_{\rm f})$-equivariant map. For each standard parabolic subgroup $\RP\subset\GL_n$, denote by
    \[
    r_{\RP}:\RH^{c_n}(\CX_n,F_{\eta}^{\vee})\to\RH^{c_n}(\partial_{\RP}\mathcal{X}_n,F_{\eta}^{\vee})
    \]
    the restriction map to the boundary stratum $\partial_{\RP}\CX_n$, and we understand $r_{\GL_n}$ to be the identity map. Assume that $\RP\subset\GL_n$ is minimal among all standard parabolic subgroups satisfying $r_{\RP}\circ\Phi\neq 0$. If $\RP\neq \RB_n$, then Harder's argument in \cite[Proof of Theorem II, pages 126--128]{Har90} gives a proper standard parabolic subgroup $Q\subset P$ such that $r_{\RQ}\circ\Phi\neq 0$, contradicting the minimality. Therefore, $\RP=\RB_n$ and $r_{\RB_n}\circ\Phi\neq 0$. This proves the lemma.
    \end{proof}

	We now finish the proof of Theorem \ref{thm:Eis}. For each $\sigma\in\mathrm{Aut}(\C)$, taking the difference between the two compositions around the diagram \eqref{maindiagram} provides a $\GL_n(\A_{\mathrm{f}})$-intertwining map
	\[
	\begin{aligned}
		\Delta_{\sigma}:\RH^{c_n}(\mathfrak{g}_{n,\infty},\widetilde{K}_{n,\infty};I_{\eta}\otimes F^\vee_{\eta})&\to \RH^{c_n}(\mathcal{X}_n,F^\vee_{\eta}),\\
		\kappa&\mapsto \mathrm{Eis}_{\eta}(\kappa)-\sigma^{-1}(\mathrm{Eis}_{^{\sigma}\eta}(\sigma(\kappa))).
	\end{aligned}
	\]
	By \eqref{diagramproof}, the image of $\Delta_{\sigma}$ lies in the kernel of $r_{\mathrm{B}_n}$ and we obtain a $\GL_n(\A_{\mathrm{f}})$-equivariant map
    \[
    \Delta_{\sigma}:I_{\eta,\mathrm{f}}\otimes \RH^{c_n}(\mathfrak{g}_{n,\infty},\widetilde{K}_{n,\infty};I_{\eta,\infty}\otimes F^\vee_{\eta})\to\mathrm{ker}(r_{\RB_n}).
    \]
    By Lemma \ref{lem-Harder} and \eqref{dimhcn}, we have that $\Delta_{\sigma}=0$. Therefore the diagram \eqref{maindiagram} commutes, completing the proof of Theorem \ref{thm:Eis}.

	\section{Cohomological tamely isobaric representations}
	
	The rest of the paper is devoted to a proof of our main theorem (Theorem \ref{mainthm}) on special values of Rankin-Selberg L-functions. We introduce in this section the cohomological tamely isobaric representations, extending the ones considered in \cite{LLS}. See \cite[Section 1.2 and Section 4]{dual} for more details.

\subsection{
Tamely isobaric automorphic representations}

We fix a partition 
    \[
     n=n_1+n_2+\dots+n_r \qquad ( r, n_1,n_2,\dots\,, n_r\geq 1)
    \]
    and write 
    \[
     \RM=\GL_{n_1}\times \GL_{n_2}\times \dots \times \GL_{n_r}, 
     \]
    which is viewed as a Levi subgroup of $\GL_n$ as usual. Denote by $\RP$ the block upper-triangular parabolic subgroup of $\GL_n$ with Levi factor $\RM$. 
    
Let $\Sigma$ be an isobaric automorphic representation of $\GL_n(\A)$: 
\be\label{isob0}
\Sigma:=\Sigma_{1}\boxplus  \Sigma_{2}\boxplus  \dots \boxplus \Sigma_{r}, 
\ee
where  $\Sigma_i$ ($1\leq i\leq r$) is an irreducible cuspidal smooth automorphic representation of $\GL_{n_i}(\A)$. We suppose that the exponents satisfy 
		\[
		\mathrm{ep}(\Sigma_1)\geq\mathrm{ep}(\Sigma_2)\geq\dots\geq\mathrm{ep}(\Sigma_r).
		\]
 Here the exponent $\mathrm{ep}(\Sigma_i)$ is the unique real number such that $\Sigma_i\otimes\abs{\det}_{\A}^{-\mathrm{ep}(\Sigma_i)}$ is unitarizable. Note that the cuspidal support of $\Sigma$
 is $[(\RM, \Sigma_{1}\widehat \otimes  \Sigma_{2}\widehat \otimes  \dots\widehat \otimes \Sigma_{r})]$ (the $\GL_n(\rk)$-conjugacy  class of the pair $(\RM, \Sigma_{1}\widehat \otimes  \Sigma_{2}\widehat \otimes  \dots\widehat \otimes \Sigma_{r})$). See \cite[Section 3]{F}, \cite[Section 1]{FS}, and \cite[Section 3.2]{dual} for the notion of cuspidal support.
 
 We have the factorizations
\[
\Sigma=\widehat{\otimes}_v'\Sigma_v=\Sigma_{\mathrm{f}}\otimes\Sigma_{\infty}.
\]
By definition, 
 \[
 \Sigma_v\cong \Sigma_{1,v}\boxplus  \Sigma_{2,v}\boxplus  \dots \boxplus \Sigma_{r,v},
 \]
 namely, the Langlands parameter of $\Sigma_v$ (which is an $n$-dimensional completely reducible representation of the Weil-Deligne group) is isomorphic to the direct sum of those of $\Sigma_{1,v}, \Sigma_{2,v},\dots, \Sigma_{r,v}$.  
Define the induced representation
\[
 \mathrm I(\Sigma):=\mathrm{Ind}^{\GL_n(\A)}_{\RP(\A)}\left(\Sigma_1\,\widehat{\otimes}\,\cdots\,\widehat{\otimes}\,\Sigma_r\right)=\widehat{\otimes}'_v\RI_v(\Sigma)=\RI_{\mathrm{f}}(\Sigma)\otimes\RI_{\infty}(\Sigma),\]
where  \[
  \RI_{\mathrm{f}}(\Sigma):=\otimes_{v\nmid\infty}'\RI_v(\Sigma),\qquad \  \RI_{\infty}(\Sigma):=\widehat{\otimes}_{v|\infty}\RI_v(\Sigma), 
\]
and
\[
\mathrm I_v(\Sigma):=\mathrm{Ind}^{\GL_n(\rk_v)}_{\RP(\rk_v)}\left(\Sigma_{1,v}\,\widehat{\otimes}\cdots\widehat{\otimes}\,\Sigma_{r,v}\right).
\] 

\begin{dfnp}\label{def:isobaric}
	An automorphic representation  of $\GL_n(\A)$ is said to be tamely isobaric if it is an isobaric automorphic representation $\Sigma$ as above such that  $\Sigma_i\neq  \Sigma_j$ and $\Sigma_i\neq\Sigma_j\otimes \abs{\det}_{\A}$ for all distinct $i,j\in\{1,2,\dots,r\}$ with $n_i=n_j$.
	\end{dfnp}

Assume that $\Sigma$ is tamely isobaric. Then the Eisenstein series yields an embedding 
\[
 \RI(\Sigma)\hookrightarrow \mathcal{A}\left(\GL_n(\rk)\backslash\GL_n(\A)\right).
\]
Denote by $\RI^{\mathrm{aut}}(\Sigma)$ the image of this embedding, which is a subrepresentation of $\mathcal{A}\left(\GL_n(\rk)\backslash\GL_n(\A)\right)$ that  depends only on $\Sigma$. See \cite[Section 5.2]{F}, \cite{BL24, Lap08}, and \cite[Section 4.2]{dual}.

\subsection{Whittaker models}

Let $\Sigma'$ be an irreducible cuspidal smooth automorphic  representation of $\GL_n(\A)$. It is obvious that  $\Sigma'$ is tamely isobaric and $\RI(\Sigma')=\Sigma'$. Put \[
    \Pi:=\Sigma\,\widehat{\otimes}\,\Sigma',\qquad  \RI(\Pi):=\RI(\Sigma)\,\widehat{\otimes}\, \RI(\Sigma'), \]
    and
    \[
    \RI^{\mathrm{aut}}(\Pi):=\RI^{\mathrm{aut}}(\Sigma)\,\widehat \otimes\, \RI^{\mathrm{aut}}(\Sigma')\subset \CA\left(\left(\GL_n(\rk)\backslash\GL_n(\A)\right)\times \left(\GL_n(\rk)\backslash\GL_n(\A)\right)\right). 
    \]

Recall the additive character $\psi:\rk\backslash\A\to\C^{\times}$ defined in \eqref{additivechar}. This further induces an automorphic character
	\begin{equation}\label{psin}
		\begin{aligned}
			\psi_n:\RN_n(\A)\times\RN_n(\A)&\to\C^{\times},\\
			\left([x_{i,j}]_{1\leq i,j\leq n},[x'_{i,j}]_{1\leq i,j\leq n}\right)&\mapsto\psi\left(\sum_{i=1}^{n-1}(x_{i,i+1}-x_{i,i+1}')\right).
		\end{aligned}
	\end{equation}
	Define the global Whittaker functional
	\begin{equation}
		\lambda_{\Pi}\in\mathrm{Hom}_{\RN_n(\A)\times\RN_n(\A)}\left(\RI^{\rm aut}(\Pi),\psi_n\right)
	\end{equation}
    by 
	\[
	\la \lambda_{\Pi}, \phi\ra:=\int_{\RN_n(\rk)\backslash\RN_n(\A)\times\RN_n(\rk)\backslash\RN_n(\A)}\phi(u)\overline{\psi_n}(u)\mathrm{d}u,\qquad \phi\in \RI^{\rm aut}(\Pi).
	\]
     This is nonzero by  \cite[Chapter 7]{Sh10} (see also \cite[Proposition 4.2]{dual}). 

Let $v$ be a place of $\rk$. Recall that by \cite[Theorem 4.1]{dual}, $\RI_v(\Pi)$ is isomorphic to the standard module whose Langlands quotient is isomorphic to $\Pi_v$.  Hence by \cite[Lemma 2.5]{J09} and \cite{JS83}, there is a unique subrepresentation 
\begin{equation}\label{whittakermodel}
\RI_v^{\rm whi}(\Pi)\subset\mathrm{Ind}^{\GL_n(\rk_v)\times\GL_n(\rk_v)}_{\RN_n(\rk_v)\times\RN_n(\rk_v)}\psi_{n,v}\qquad(\text{the Whittaker model})
\end{equation}
that is isomorphic to $\RI_v(\Pi)$. We put
\[
\RI^{\rm whi}(\Pi):=\widehat \otimes_v'\RI_v^{\rm whi}(\Pi),\qquad \RI_{\rm f}^{\rm whi}(\Pi):=\otimes_{v\nmid\infty}'\RI_v^{\rm whi}(\Pi),\qquad\RI^{\rm whi}_{\infty}(\Pi):=\widehat\otimes_{v|\infty}\RI_v^{\rm whi}(\Pi).
\]
Then $\lambda_{\Pi}$ induces a $\GL_n(\A)\times\GL_n(\A)$-equivariant isomorphism
\begin{equation}\label{lambdapi}
\Lambda_{\Pi}:\RI^{\rm aut}(\Pi)\to\RI^{\rm whi}(\Pi).
\end{equation}

\subsection{Cohomology spaces}
Denote by $\CR^{\rm coh}$ the set of all isomorphism classes of irreducible automorphic representations of $\GL_n(\A)\times \GL_n(\A)$ of the form 
\[
\Pi=\Sigma \,\widehat \otimes\, \Sigma',
\]
where
\begin{itemize}
    \item  $\Sigma$ is a cohomological tamely isobaric automorphic representation of $\GL_n(\A)$ with generic infinite part; and 
    \item $\Sigma'$ is a cohomological irreducible cuspidal smooth  automorphic representation of $\GL_n(\A)$.
\end{itemize}
Note that an isobaric automorphic representation $\Sigma$ has generic infinite part if and only if $\RI_{\infty}(\Sigma)$ is irreducible (see \cite{BZ77, JS83, Vo78, Z80}).

Let $\Pi=\Sigma\,\widehat \otimes\, \Sigma'\in \CR^{\rm coh}$.  Then there exist dominant weights $\mu$ and $\nu$ such that $F_{\mu}\otimes F_{\nu}$ (the coefficient system) has the same infinitesimal character as $\Pi_{\infty}$. We consider the relative Lie algebra cohomology spaces
	\begin{equation}\label{HPi}
		\begin{aligned}
			\CH(\Pi)&:=\RH^{2b_n}\left(\mathfrak{g}_{n,\infty}\times\mathfrak{g}_{n,\infty},\widetilde{K}_{n,\infty}\times\widetilde{K}_{n,\infty};(F_{\mu}^{\vee}\otimes F_{\nu}^{\vee})\otimes\RI^{\rm aut}(\Pi)\right),\\
			\CH(\Pi_{\infty})&:=\RH^{2b_n}\left(\mathfrak{g}_{n,\infty}\times\mathfrak{g}_{n,\infty},\widetilde{K}_{n,\infty}\times\widetilde{K}_{n,\infty};(F_{\mu}^{\vee}\otimes F_{\nu}^{\vee})\otimes\RI_{\infty}^{\rm whi}(\Pi)\right),
		\end{aligned}
	\end{equation}
    where
    \[
	b_n:=\frac{[\rk:\BQ]}{2}\cdot\frac{ n(n-1)}{2}.
	\]
    
The inverse of \eqref{lambdapi} defines a $\GL_n(\A)\times\GL_n(\A)$-equivariant isomorphism
\[
\RI_{\rm f}^{\rm whi}(\Pi)\otimes\RI_{\infty}^{\rm whi}(\Pi)=\RI^{\rm whi}(\Pi)\xrightarrow{\Lambda_{\Pi}^{-1}}\RI^{\rm aut}(\Pi)
\]
which induces a canonical isomorphism
	\begin{equation}\label{canpi}
		\imath_{\mathrm{can}}:\CH(\Pi_{\infty})\otimes\RI^{\rm whi}_{\mathrm{f}}(\Pi)\xrightarrow{\sim}\CH(\Pi).
	\end{equation}
	As a consequence of the Künneth formula and Delorme's lemma \cite[Theorem III.3.3]{BW} (see also \cite[Section 1.1]{local}), we have that 
    \[
    \dim\CH(\Pi_{\infty})=1.
	\]
	By \cite[Lemma 3.15]{Clo} and \cite[Satz 4.11]{Sch83} (see also \cite[Proposition 1.6]{G} and \cite[Lemma 4.11]{dual}), we have a canonical $\GL_n(\A_{\mathrm{f}})\times\GL_n(\A_{\rm f})$-equivariant embedding
	\begin{equation}\label{iotapi}
		\imath_{\Pi}:\CH(\Pi)\hookrightarrow\RH^{b_n}(\CX_n,F_{\mu}^{\vee})\otimes\RH_{\mathrm{c}}^{b_n}(\CX_n,F_{\nu}^{\vee}).
	\end{equation}
    Here $\RH_{\mathrm{c}}^{b_n}$ indicates the compactly supported cohomology (of degree $b_n$).	As in  \eqref{sheafsigma},  for every $\sigma\in\mathrm{Aut}(\C)$, we have the $\sigma$-linear isomorphism
\begin{equation}\label{sigmasheaf}
		\sigma:\RH^{b_n}(\CX_n,F_{\mu}^{\vee})\otimes\RH_{\mathrm{c}}^{b_n}(\CX_n,F_{\nu}^{\vee})\to\RH^{b_n}(\CX_n,F_{{^{\sigma}\mu}}^{\vee})\otimes\RH_{\mathrm{c}}^{b_n}(\CX_n,F_{{^{\sigma}\nu}}^{\vee}).
	\end{equation}

	\subsection{$\Aut(\C)$-twists} \label{sec:sigmapi}

    Let $\sigma\in\mathrm{Aut}(\C)$. For all $v\nmid \infty$,  we have a $\sigma$-linear isomorphism
	\be\label{sigmawhi}
	\sigma:\mathrm{Ind}^{\GL_n(\rk_v)\times\GL_n(\rk_v)}_{\RN_n(\rk_v)\times\RN_n(\rk_v)}\psi_{n,v}\to\mathrm{Ind}^{\GL_n(\rk_v)\times\GL_n(\rk_v)}_{\RN_n(\rk_v)\times\RN_n(\rk_v)}\psi_{n,v}
	\ee
	defined as in \cite[p.594]{Ma05} (see also \cite[Section 2.7]{dual}). 
    
    By \cite[Lemma 4.9]{dual} and \cite[Proposition 4.10]{dual},  there exists a unique  representation $\,^\sigma \Pi\in \CR^{\rm coh}$ such that 
    \[
    \RI_v^{\rm whi}(\,^\sigma \Pi)=\sigma(\RI_v^{\rm whi}(\Pi))\quad \textrm{ for all $v\nmid \infty$.}
 \]
The map \eqref{sigmawhi} for various $v\nmid \infty$ induces a $\sigma$-linear isomorphism
    \begin{equation}\label{sigmapi}
		\sigma:\RI^{\rm whi}_{\mathrm{f}}(\Pi)\to \RI^{\rm whi}_{\mathrm{f}}(^{\sigma}\Pi).
	\end{equation}

Note that $^\sigma\Pi$  has coefficient system $F_{^\sigma\mu}\otimes F_{^\sigma\nu}$. The rationality field of $\Pi$, denoted by $\BQ(\Pi)$, is by definition the fixed field of the group of field automorphisms $\sigma\in\mathrm{Aut}(\C)$ such that ${}^\sigma\Pi \cong \Pi$. We remark that $\BQ(\Pi)$ is a number field.

	\section{Modular symbols and special values of L-functions}

    We now study the special values of L-functions via modular symbols and prove the main theorem. We retain the notation of the Introduction. In particular, $\Pi = \Sigma \,\widehat{\otimes}\, \Sigma' \in \CR^{\rm coh}$ is an irreducible smooth automorphic representation of $\GL_n(\A) \times \GL_n(\A)$, cohomological with coefficient system $F_{\mu} \otimes F_{\nu}$, where $\mu$ and $\nu$ are the dominant weights in \eqref{writemu} and \eqref{writenu}, respectively. Moreover, $\Sigma$ is tamely isobaric with generic infinite part, and $\Sigma'$ is cuspidal. Let $\chi \in \RB(\mu,\nu)$, and let $\eta$ be as in \eqref{eta}. Throughout Sections~\ref{sec7.1}--\ref{sec7.5}, we assume $\chi \in \RB(\mu,\nu)^+$ and first prove Theorem~\ref{mainthm} under this hypothesis. The remaining case will then be treated in Section~\ref{sec7.6} via the functional equation.

	\subsection{Rankin-Selberg integrals}\label{sec7.1}
	
As usual, write $\mathrm{PGL}_n:=\RZ_n\backslash \GL_n$ for the projective general linear group. Then $\RN_n$ is obviously identified as an algebraic subgroup of $\mathrm{PGL}_n$. 
For each place $v$ of $\rk$, let $\mathfrak{M}_{n,v}$ denote the one-dimensional space of left invariant Borel measures on $\mathrm{PGL}_n(\rk_v)$. When $v$ is non-archimedean, it has a distinguished element with respect to which $\PGL_n(\CO_v)$ has total volume $1$.
Denote by $\mathfrak{M}_n$ the one-dimensional space of left-invariant measures on $\mathrm{PGL}_n(\A)$. Then we have the usual decomposition 
	\[
	\mathfrak{M}_n=\otimes_v'\mathfrak{M}_{n,v}=\mathfrak{M}_{n,\infty}\otimes\mathfrak{M}_{n,\mathrm{f}},
	\]
	where the restricted tensor product is defined with respect to those distinguished elements. 

     Equip $\mathrm{PGL}_n(\rk)$ with the counting measure, and recall the measures on $\RN_n(\A)$ and $\RN_n(\rk_v)$ fixed in Section \ref{sec2.2}. With respect to these measures, for each $\mathrm{d}g\in\mathfrak{M}_n$ or $\mathfrak{M}_{n,v}$, the quotient measures on $\RN_n(\A)\backslash\mathrm{PGL}_n(\A)$, $\mathrm{PGL}_n(\rk)\backslash\mathrm{PGL}_n(\A)$, or $\RN_n(\rk_v)\backslash\mathrm{PGL}_n(\rk_v)$ are all denoted by $\overline{\mathrm{d}}g$.

We view $\GL_n$ as an algebraic subgroup of $\GL_n\times\GL_n$ via the diagonal embedding. The global Rankin-Selberg integral is defined by
	\begin{equation}\label{globalRS}
		\RZ(s,\phi,\varphi;\chi,\mathrm{d}g):=\int_{\mathrm{PGL}_n(\rk)\backslash\mathrm{PGL}_n(\A)}\phi(g)\RE(g;\varphi_s,\eta)\chi(\det g)\abs{\det g}^s_{\A}\overline{\mathrm{d}}g,
	\end{equation}
 where $\phi\in \RI^{\rm aut}(\Pi)$, $\varphi\in I_{\eta}$, and $\mathrm{d}g\in\mathfrak{M}_n$. Here we recall that $\varphi_s\in I_{\eta,s}$ is defined as in \eqref{varphis}. When $\mathrm{Re}(s)$ is sufficiently large so that the Eisenstein series $\RE(g;\varphi_s,\eta)$ absolutely converges, we unfold the Eisenstein series in \eqref{globalRS} to obtain
    \[
\RZ(s,\phi,\varphi;\chi,\mathrm{d}g)=\int_{\RN_n(\A)\backslash\mathrm{PGL}_n(\A)}\left\langle \lambda_{\Pi},g.\phi\right\rangle\varphi_s(g)\chi(\det g)\abs{\det g}_{\A}^s\overline{\mathrm{d}}g.
	\] 
	Therefore, for decomposable data
    \[
    \Lambda_\Pi(\phi) = \otimes_v\phi_v, \quad \phi_v\in\RI^{\rm whi}_v(\Pi),
    \]
    and
    \[
        \varphi = \otimes_v \varphi_v, \quad \varphi_v \in I_{\eta, v},
    \]
    the global Rankin-Selberg integral admits an Euler product factorization into the local Rankin-Selberg integrals
	\begin{equation}\label{localRS}
		\RZ_v(s,\phi_v,\varphi_v;\chi_v,\mathrm{d}g):=\int_{\RN_n(\rk_v)\backslash\mathrm{PGL}_n(\rk_v)}\phi_v(g)\varphi_{s,v}(g)\chi_v(\det g)\abs{\det g}^s_v\overline{\mathrm{d}}g, 
	\end{equation}
    where $\mathrm{d}g\in\mathfrak{M}_{n,v}$. We further define the normalized local Rankin-Selberg integrals
	\begin{equation}\label{normalizedRS}
		\RZ^{\circ}_v(s,\phi_v,\varphi_v;\chi_v,\mathrm{d}g):=\frac{\RL(ns,\omega_{\Pi_v}\chi_v^n)}{\RL(s,\Pi_v\times\chi_v)}\cdot\RZ_v(s,\phi_v,\varphi_v;\chi_v,\mathrm{d}g).
	\end{equation}

By \cite{JS1, JPSS83, Cog} and \cite[Theorem 4.1]{dual}, the normalized local Rankin-Selberg integral \eqref{normalizedRS} is holomorphic at $s=0$, and evaluating $\RZ_v^{\circ}$ at $s=0$ yields a nonzero element 	        \begin{equation}\label{functionalRS}
\RZ^{\circ}_{\Pi_v,\chi_v}\in\mathrm{Hom}_{\GL_n(\rk_v)}\left(\RI^{\rm whi}_v(\Pi)\otimes I_{\eta,v}\otimes \chi_v\otimes\mathfrak{M}_{n,v},\C\right).
\end{equation}
Denote 
\[
\begin{aligned}
\RZ^{\circ}_{\Pi_{\infty},\chi_{\infty}}&:=\bigotimes_{v\mid\infty}\RZ^{\circ}_{\Pi_v,\chi_v}\in\mathrm{Hom}_{\GL_n(\rk_\infty)}\left(\RI^{\rm whi}_\infty(\Pi)\otimes I_{\eta,\infty}\otimes \chi_\infty\otimes\mathfrak{M}_{n,\infty},\C\right),\\
\RZ^{\circ}_{\Pi_{\mathrm{f}},\chi_{\mathrm{f}}}&:=\bigotimes_{v\nmid\infty}\RZ^{\circ}_{\Pi_v,\chi_v}\in\mathrm{Hom}_{\GL_n(\A_{\rm f})}\left(\RI^{\rm whi}_{\rm f}(\Pi)\otimes I_{\eta,\rm f}\otimes \chi_{\rm f}\otimes\mathfrak{M}_{n,\rm f},\C\right).
\end{aligned}
\]

	\subsection{Definition of modular symbols}

    Denote by $\mathfrak{O}_{n,\infty}$ the one-dimensional space of $\GL_n(\rk_{\infty})$-invariant sections of the orientation line bundle of $\GL_n(\rk_{\infty})/\widetilde{K}_{n,\infty}$ with complex coefficients. Set $\mathfrak{M}_n^{\natural}:=\mathfrak{M}_{n,\mathrm{f}}\otimes\mathfrak{O}_{n,\infty}$. 
    By using the push-forward of measures through the natural proper map 
    \[
    \mathrm{PGL}_n(\rk_{\infty})\to\GL_n(\rk_{\infty})/\widetilde K_{n,\infty},
    \]
    we identify $\mathfrak{M}_{n,\infty}$ with the space of $\GL_n(\rk_{\infty})$-invariant Borel measures on $\GL_n(\rk_{\infty})/\widetilde{K}_{n,\infty}$, which is further identified with the space 
    \[
\wedge^{d_n}\left(\mathfrak{g}_{n,\infty}/\widetilde{\mathfrak{k}}_{n,\infty}\right)^{\ast} \otimes\mathfrak{O}_{n,\infty}.
    \]
    See \cite[Section 3.1]{LLS} for more details. Here we use the superscript `$^\ast$' to indicate the dual vector space and
    \[
    d_n:=\dim \left(\GL_n(\rk_{\infty})/\widetilde{K}_{n,\infty}\right)=\frac{[\rk:\BQ]}{2}\cdot(n^2-1).
    \]

	Recall the cohomology spaces $\CH(\Pi)$ and $\CH(\Pi_{\infty})$ defined in \eqref{HPi}. We similarly write for short
	\begin{equation}
		\label{HI}
		\begin{aligned}
			\CH(I_{\eta})&:=\RH^{c_n}\left(\mathfrak{g}_{n,\infty},\widetilde{K}_{n,\infty};F_{\eta}^{\vee}\otimes I_{\eta}\right),\\
			\CH(I_{\eta_{\infty}})&:=\RH^{c_n}\left(\mathfrak{g}_{n,\infty},\widetilde{K}_{n,\infty};F_{\eta}^{\vee}\otimes I_{\eta_{\infty}}\right).
		\end{aligned}
	\end{equation}
	    Denote  
	\[
	\CH(\chi_{\infty}):=\RH^0(\mathfrak{g}_{n,\infty},\widetilde{K}_{n,\infty};F_{\chi}^{\vee}\otimes\chi_{\infty}),
	\]
	which is canonically identified with $\C$, and denote
	\[
\CH(\chi):=\RH^0(\mathfrak{g}_{n,\infty},\widetilde{K}_{n,\infty};F_{\chi}^{\vee}\otimes\chi)\cong\chi_{\mathrm{f}}.
	\]
	We further denote 
	\[
	\begin{aligned}
		\CH(\Pi,\chi)_{\mathrm{glob}}&:=\CH(\Pi)\otimes\CH(I_{\eta})\otimes\CH(\chi)\otimes\mathfrak{M}_n^{\natural},\\
		\CH(\Pi_{\infty},\chi_{\infty})&:=\CH(\Pi_{\infty})\otimes\CH(I_{\eta_{\infty}})\otimes\CH(\chi_{\infty})\otimes\mathfrak{O}_{n,\infty},\\
		\CH(\Pi_{\mathrm{f}},\chi_{\mathrm{f}})&:=\RI^{\rm whi}_{\rm f}(\Pi)\otimes I_{\eta,{\mathrm{f}}}\otimes\chi_{\mathrm{f}}\otimes\mathfrak{M}_{n,\mathrm{f}},\\
		\CH(\Pi,\chi)_{\mathrm{loc}}&:=\CH(\Pi_{\infty},\chi_{\infty})\otimes\CH(\Pi_{\mathrm{f}},\chi_{\mathrm{f}}),
	\end{aligned}
	\]
	and the isomorphism \eqref{canpi} induces  the canonical isomorphism
	\begin{equation}
		\imath_{\mathrm{can}}:\CH(\Pi,\chi)_{\mathrm{loc}}\xrightarrow{\sim}\CH(\Pi,\chi)_{\mathrm{glob}}.
	\end{equation}
	
    We define an invariant linear functional
	\begin{equation}
		\begin{aligned}
			\phi_{\mu,\nu,\chi}&\in\mathrm{Hom}_{\GL_n(\rk\otimes_{\BQ}\C)}\left(F_{\mu}^{\vee}\otimes F_{\nu}^{\vee}\otimes F_{\eta}^{\vee}\otimes F_{\chi}^{\vee},\C\right)\\
			&=\left(F_{\mu}\otimes F_{\nu}\otimes F_{\eta}\otimes F_{\chi}\right)^{\GL_n(\rk\otimes_{\BQ}\C)}
		\end{aligned}
	\end{equation}
	as in \cite[Section 1.3]{local}.

	The  pairing with the fundamental class yields a linear map (see \cite[Section 5.1]{Ma05} for more explanation)
	\[
	\int_{\CX_n}:\RH_{\mathrm{c}}^{d_n}(\CX_n,\C)\otimes\mathfrak{M}_n^{\natural}\to\C.
	\]
	Note that
	\[
	d_n=2b_n+c_n.
	\]
	We define the global modular symbol
	\begin{equation}\label{globalmodularsymbol}
		\CP:=\CP_{\Pi,\chi}:\CH(\Pi,\chi)_{\mathrm{glob}}\to\C
	\end{equation}
	to be the composition of
	\[
	\begin{aligned}
		\CP_{\Pi,\chi}:\ &\CH(\Pi,\chi)_{\mathrm{glob}}\\
		\xrightarrow{\imath_{\Pi}\otimes\mathrm{Eis}_{\eta}\otimes \imath_\chi\otimes \rm id}\ &\RH^{b_n}(\CX_n,F_{\mu}^{\vee})\otimes\RH_{\mathrm{c}}^{b_n}(\CX_n,F_{\nu}^{\vee})\otimes\RH^{c_n}(\CX_n,F_{\eta}^{\vee})\otimes\RH^0(\CX_n,F_{\chi}^{\vee})\otimes\mathfrak{M}_n^{\natural}\\
		\xrightarrow{\phi_{\mu,\nu,\chi}}\ &\RH_{\mathrm{c}}^{d_n}(\CX_n,\C)\otimes\mathfrak{M}_n^{\natural}\xrightarrow{\int_{\CX_n}}\C.
	\end{aligned}
	\]
	Here the first arrow is the tensor product of the map $\imath_\Pi$ given in \eqref{iotapi}, the map $\mathrm{Eis}_{\eta}$ given in \eqref{Eis}, the natural map 
    \[
    \imath_\chi: \CH(\chi)\to\RH^0(\CX_n,F_{\chi}^{\vee}),
    \]
    and the identity map $\rm id:\mathfrak{M}_n^{\natural}\to\mathfrak{M}_n^{\natural}$.
	On the other hand, we define the archimedean modular symbol
	\begin{equation}\label{archimedeanmodularsymbol}
		\CP_{\infty}:=\CP_{\Pi_{\infty},\chi_{\infty}}:\CH(\Pi_{\infty},\chi_{\infty})\to\C
	\end{equation}
	to be the composition of
	\[
	\begin{aligned}
		\CP_{\Pi_{\infty},\chi_{\infty}}:\ &\CH(\Pi_{\infty},\chi_{\infty})\\
		\xrightarrow{\text{res}}&\,\RH^{d_n}\left(\mathfrak{g}_{n,\infty},\widetilde{K}_{n,\infty};F_{\mu}^{\vee}\otimes F_{\nu}^{\vee}\otimes F_{\eta}^{\vee}\otimes F_{\chi}^{\vee}\otimes\RI_{\infty}^{\rm whi}(\Pi)\otimes I_{\eta_{\infty}}\otimes\chi_{\infty}\right)\otimes\mathfrak{O}_{n,\infty}\\
		\xrightarrow{\phi_{\mu,\nu,\chi}\otimes\RZ^{\circ}_{\Pi_{\infty},\chi_{\infty}}}\,&\RH^{d_n}(\mathfrak{g}_{n,\infty},\widetilde{K}_{n,\infty};\mathfrak{M}_{n,\infty}^{\ast})\otimes\mathfrak{O}_{n,\infty}=\C,
	\end{aligned}
	\]
	where the first arrow is induced by the restriction of cohomology. See \cite[Section 3.1]{LLS} and \cite[Section 1.5]{local} for more details.  We also have the non-archimedean normalized Rankin-Selberg period
	\begin{equation}\label{nonarchimedeanmodularsymbol}
		\CP_{\mathrm{f}}:=\RZ^{\circ}_{\Pi_{\mathrm{f}},\chi_{\mathrm{f}}}:\CH(\Pi_{\mathrm{f}},\chi_{\mathrm{f}})\to\C.
	\end{equation}
	
	Note that $\RL(0,\omega_{\Pi}{\chi}^n)\in \BC^\times$ since $\chi \in \RB(\mu,\nu)^+$. The global modular symbol is related to the archimedean modular symbol and the non-archimedean Rankin-Selberg period  by the following proposition. 
	
	\begin{prp}
        The diagram
		\begin{equation}\label{frontback}
			\begin{CD}
				\CH(\Pi,\chi)_{\mathrm{loc}} @>\CP_{\infty}\otimes\CP_{\mathrm{f}}>>\C\\
				@V\imath_{\mathrm{can}}VV @VV\CL(\Pi,\chi)V\\
				\CH(\Pi,\chi)_{\mathrm{glob}}@>\CP>>\C
			\end{CD}
		\end{equation}
		commutes, where the right vertical arrow is multiplication by 
		\[
		\CL(\Pi,\chi):=\frac{\RL(0,\Pi\times\chi)}{\RL(0,\omega_{\Pi}\chi^n)}.
      		\]
               \end{prp}

\begin{proof}
This is well known and the proof is the same as that of \cite[Proposition 7.2]{LLS}. 
\end{proof}

	\subsection{Non-archimedean period relations}

	We say that an invariant measure in $\mathfrak{M}_{n,\mathrm{f}}$  is rational if every open compact subgroup of $\mathrm{PGL}_n(\mathbb{A}_{\mathrm{f}})$ has rational volume with respect to it. All these rational measures  form a rational structure of $\mathfrak{M}_{n,\mathrm{f}}$. With respect to it, we have the $\sigma$-linear isomorphism
	\begin{equation}\label{sigmaM}
		\sigma:\mathfrak{M}_{n,\mathrm{f}}\to\mathfrak{M}_{n,\mathrm{f}}
	\end{equation}
    for   every $\sigma\in\mathrm{Aut}(\C)$. From  \eqref{sigmaIf}, \eqref{sigmapi} and \eqref{sigmaM}, we obtain a $\sigma$-linear isomorphism
	\begin{equation}
		\sigma:\CH(\Pi_{\mathrm{f}},\chi_{\mathrm{f}})\to\CH({^{\sigma}\Pi}_{\mathrm{f}},{^{\sigma}\chi_{\mathrm{f}}}).
	\end{equation}
	
	For each non-archimedean place $v$, denote by $\mathfrak{c}(\chi_v)$ and $\mathfrak{c}(\psi_v)$ the conductors of $\chi_v$ and $\psi_v$ respectively, which are fractional ideals of $\rk_v$. Fix $y_v\in\rk_v^{\times}$ such that $\mathfrak{c}(\psi_v)=y_v\cdot\mathfrak{c}(\chi_v)$. Define the local Gauss sum by
	\begin{equation}\label{localgauss}
		\CG(\chi_v):=\CG(\chi_v,\psi_v):=\int_{\CO_v^{\times}}\chi_v(y_vx)^{-1}\cdot\psi_v(y_v x)\mathrm{d}x,
	\end{equation}
	where $\mathrm{d}x$ is the normalized Haar measure so that $\CO_{v}^{\times}$ has volume $1$. We note that the above definition is independent of the choice of $y_v$. Define the Gauss sum of $\chi$ by
	\begin{equation}\label{globalgauss}
		\CG(\chi):=\prod_{v\nmid\infty}\CG(\chi_v).
	\end{equation}
	This is a well-defined nonzero complex number, and the definition obviously applies to all characters of $\rk^\times \backslash \A^\times$. 
	
	Let $\sigma\in\mathrm{Aut}(\C)$. Define a  linear functional
	\[
	{^{\sigma}\CP}_{\mathrm{f}}^{\circ}:=\CG({^{\sigma}\chi})^{\frac{n(n-1)}{2}}\cdot\RZ^{\circ}_{{^{\sigma}\Pi_{\mathrm{f}}},{^{\sigma}\chi_{\mathrm{f}}}}:\CH({^{\sigma}\Pi_{\mathrm{f}}},{^{\sigma}\chi_{\mathrm{f}}})\to\C.
	\]
	When $\sigma$ is the identity map, we simply write $\CP_{\mathrm{f}}^{\circ}$ for ${^{\sigma}\CP}_{\mathrm{f}}^{\circ}$. The non-archimedean period relation is the following commutative diagram
	\begin{equation}
		\label{nonarchimedeanperioerelation}
		\begin{CD}
			\CH(\Pi_{\mathrm{f}},\chi_{\mathrm{f}}) @>\CP_{\mathrm{f}}^{\circ}>>\C\\
			@V\sigma VV @VV\sigma V\\
			\CH({^{\sigma}\Pi_{\mathrm{f}}},{^{\sigma}\chi_{\mathrm{f}}}) @>{^{\sigma}\CP}_{\mathrm{f}}^{\circ}>>\C.
		\end{CD}
	\end{equation}
This is essentially due to Raghuram and Shahidi (\cite{RS08}). In view of \cite[Theorem 4.1]{dual}, it can be obtained by the same proof as in \cite[Propositions 5.1  and 7.4]{LLS}, 
following the ideas of \cite[Section III]{Har83} (for $n=2$), \cite[Section 3.4]{Ma05}, and \cite[Section 3.3]{RS08}.
    

	\subsection{Archimedean period relations}
	
	Note that $\mathfrak{O}_{n,\infty}$ has a natural $\BQ$-rational structure and we obtain a $\sigma$-linear isomorphism
	\begin{equation}\label{sigmao}
		\sigma:\mathfrak{O}_{n,\infty}\to\mathfrak{O}_{n,\infty}.
	\end{equation}
	We fix an orientation $\mathbf{o}_n$ of $\GL_n(\rk_{\infty})/\widetilde{K}_{n,\infty}$ which is a $\BQ$-rational generator of the one-dimensional space $\mathfrak{O}_{n,\infty}$.

	Recall that $\dim\CH(\Pi_{\infty})=1$. We are going to fix a generator $\kappa_{\mu,\nu}$ of this one-dimensional space. We first consider the case when $\mu=\nu=0_n$, where $0_n$ is the zero element of the abelian group $(\Z^n)^{\CE_{\rk}}$. Recall the character $\eta^\circ = \otimes_{v\mid\infty} \eta_v^\circ$ of $\rk_\infty^\times$ in Section \ref{sec5.2}.
    Then $\eta^\circ = (\chi^\circ)^n$ with 
    \[
 \chi^\circ:=\prod_{v|\infty}\iota_v': \rk_{\infty}^{\times}\to\C^{\times}.
    \]
	 Let $\pi_n$ be the irreducible subrepresentation of $\mathrm{Ind}^{\GL_n(\rk_\infty)\times\GL_n(\rk_\infty)}_{\RN_n(\rk_\infty)\times\RN_n(\rk_\infty)}\psi_{n,\infty}$ that has the same infinitesimal character as the trivial representation.  
 
    \begin{lemp}\label{lem:wd}
        There exists a generator 
        \[
    \kappa_{0_n,0_n}\in \CH(\pi_n):=\RH^{2b_n}\left(\mathfrak{g}_{n,\infty}\times\mathfrak{g}_{n,\infty},\widetilde{K}_{n,\infty}\times\widetilde{K}_{n,\infty};\pi_n\right)
    \]
    such that
    \begin{equation}\label{P1}
	\CP_{\pi_n, \chi^\circ}\left(\kappa_{0_n,0_n}\otimes\kappa_{\eta^\circ}\otimes 1\otimes\mathbf{o}_n\right)=\prod_{v|\infty}\varepsilon_{\Pi,\chi,v}^{\frac{n(n-1)}{2}}.
	\end{equation}
    Moreover, the generator $\kappa_{0_n,0_n}\in\CH(\pi_n)$ is independent of the character $\chi^\circ$.
    \end{lemp}

    \begin{proof}
    The existence of $\kappa_{0_n,0_n}$ follows from the non-vanishing hypothesis of the archimedean modular symbol (\cite[Theorem 1.6 (b)]{local}). Let
\[
\chi':=\prod_{v|\infty}\iota_v'':\rk_{\infty}^{\times}\to\C^{\times}
\]
be an arbitrary character with $\iota_v''\in\CE_{\rk_v}$ for each archimedean place $v$ of $\rk$, and put $\eta'=(\chi')^n$. By using \cite[Theorems 1.2 and 4.5]{JT26}, we have that
\[
\CP_{\pi_n, \chi'}\left(\kappa_{0_n,0_n}\otimes\kappa_{\eta'}\otimes 1\otimes\mathbf{o}_n\right)=\prod_{\substack{v|\infty\\\iota_v''=\iota_v}}(-1)^{\frac{n(n-1)}{2}},
\]
which completes the proof of the lemma.
    \end{proof}

    In general, we define a generator
    \[
    \kappa_{\mu,\nu}:=\jmath_{\mu,\nu}(\kappa_{0_n,0_n})\in\CH(\Pi_{\infty}),
    \]
where 
\[
\jmath_{\mu,\nu}:=\jmath_{\mu}\otimes\jmath_{\nu}:\CH(\pi_n)\to\CH(\Pi_{\infty})
\]
is the translation map defined in \cite[Proposition 1.1]{local}. 
	
    With respect to the above choice of generators, we define a $\sigma$-linear isomorphism
	\begin{equation}\label{sigmaaA}
		\sigma:\CH(\Pi_{\infty})\to\CH({^{\sigma}\Pi_{\infty}})
	\end{equation}
	such that $\sigma(\kappa_{\mu,\nu})=\kappa_{{^{\sigma}\mu},{^{\sigma}\nu}}$.  
    From  \eqref{sigmaI}, \eqref{sigmao} and \eqref{sigmaaA}, we obtain a $\sigma$-linear isomorphism
    \begin{equation}
    \sigma:\CH(\Pi_{\infty},\chi_{\infty})\to\CH({^{\sigma}\Pi_{\infty}},{^{\sigma}\chi_{\infty}}).
    \end{equation}
    Define the normalized archimedean modular symbol
	\[
	{^{\sigma}\CP}_{\infty}^{\circ}:=\Omega_{\infty}({^{\sigma}\Pi},{^{\sigma}\chi})\cdot\CP_{{^{\sigma}\Pi_{\infty}},{^{\sigma}\chi_{\infty}}}:\CH({^{\sigma}\Pi_{\infty}},{^{\sigma}\chi_{\infty}})\to\C,
	\]
	where the constant $\Omega_{\infty}({^{\sigma}\Pi},{^{\sigma}\chi})\in\C^{\times}$ is defined as in \eqref{omegainfty}. When $\sigma$ is the identity map, we omit $\sigma$ and simply write  ${\CP}_{\infty}^{\circ}$ for ${^{\sigma}\CP}_{\infty}^{\circ}$. 
    
    Then we have the following archimedean period relation.

    \begin{prp} \label{prop:APR} 
    The diagram
	\begin{equation}\label{archimedeanperiodrelation}
		\begin{CD}
			\CH(\Pi_{\infty},\chi_{\infty}) @>\CP_{\infty}^{\circ}>>\C\\
			@V\sigma VV @VV\sigma V\\
			\CH({^{\sigma}\Pi_{\infty}},{^{\sigma}\chi_{\infty}}) @>{^{\sigma}\CP}_{\infty}^{\circ}>>\C
		\end{CD}
	\end{equation}
    commutes for every $\sigma\in\Aut(\C)$.
    \end{prp}

    \begin{proof}
Let $\sigma\in\Aut(\C)$ and define
\[
\begin{aligned}
 &   \Omega'_{\infty}({^\sigma\Pi},{^\sigma\chi}):=\left(\prod_{v|\infty}\varepsilon_{{^\sigma\Pi},{^\sigma\chi},v}\right)^{\frac{n(n-1)}{2}}\cdot\prod_{v|\infty}\prod_{i+k\leq n}\mathrm{i}^{1-\sum_{\iota\in\CE_{\rk_v}}(\mu_i^{\sigma^{-1}\circ\iota}+\nu_k^{\sigma^{-1}\circ\iota}+\chi_{\sigma^{-1}\circ\iota})}\\
    &\cdot\prod_{\substack{\iota\in\CE_{\rk}\\ \eta_{\sigma^{-1}\circ\iota}\geq n}}\left((-1)^{\frac{n(n-1)}{2}(\chi_{\sigma^{-1}\circ\iota}-1)}\cdot\prod_{i=1}^n(-1)^{(n+1)\nu_i^{\sigma^{-1}\circ\iota}}\right)\\
    &\cdot\prod_{\substack{i> k\\ i+k\leq n}}\left(\prod_{\substack{\iota\in\CE_{\rk}\\\eta_{\sigma^{-1}\circ\iota}\leq 0}}(-1)^{\mu_i^{\sigma^{-1}\circ\iota}+\nu_k^{\sigma^{-1}\circ\iota}+\chi_{\sigma^{-1}\circ\iota}}\prod_{\substack{\iota\in\CE_{\rk}\\\eta_{\sigma^{-1}\circ\iota}\geq n}}(-1)^{\mu_{n+1-i}^{\sigma^{-1}\circ\iota}+\nu_{n+1-k}^{\sigma^{-1}\circ\iota}+\chi_{\sigma^{-1}\circ\iota}-1}\right).
\end{aligned}
\]
     Recall from \eqref{ind-trans} that $\jmath_{\eta}(\kappa_{\eta^{\circ}})=\kappa_{\eta}$. By \cite[Theorem 1.6(a)]{local}, we have that
\[
\CP_{^\sigma\Pi_{\infty},{^\sigma\chi_{\infty}}}\left(\kappa_{{^\sigma\mu},{^\sigma\nu}}\otimes\kappa_{^\sigma\eta}\otimes 1\otimes\mathbf{o}_n\right)=\Omega'_{\infty}({^\sigma\Pi},{^\sigma\chi}).
\]
Note that
 \[
    \sigma(\kappa_\eta) =\frac{\sigma(\delta_{\infty}(\eta)^{n-1})}{\delta_{\infty}(^\sigma\eta)^{n-1}}\cdot \kappa_{^\sigma\eta}
    \]
    by Lemma \ref{sigmakappaeta}. We verify that 
    \[
    \frac{\sigma\left(\Omega_{\infty}(\Pi,\chi)\right)}{\Omega_{\infty}({^\sigma\Pi},{^\sigma\chi})}=\frac{\sigma\left(\Omega'_{\infty}(\Pi,\chi)\right)}{\Omega'_{\infty}({^\sigma\Pi},{^\sigma\chi})}\cdot\frac{\sigma(\delta_{\infty}(\eta)^{n-1})}{\delta_{\infty}(^\sigma\eta)^{n-1}},
    \]
    which completes the proof of the proposition.
    \end{proof}

	\subsection{Global period relations}\label{sec7.5}
We prove Theorem \ref{mainthm} when $\chi\in\RB(\mu,\nu)^+$ using modular symbols.

\subsubsection{Deligne periods associated to critical algebraic Hecke characters}\label{751}

Recall that an algebraic Hecke character $\xi:\rk^\times \backslash \A^\times \rightarrow \C^\times$ is said to be critical if $s=0$ is not a pole of $\RL(s,\xi_{\infty})$ or $\RL(1-s,\xi_\infty^{-1})$, where $\xi_\infty$ is the archimedean part of $\xi$. Note that if $\xi$ is critical, then $^\sigma\xi$ is critical for every $\sigma\in\mathrm{Aut}(\C)$. 

Suppose that $\RE$ is a number field and  $\widetilde{\xi}:\A^{\times}\to\RE^{\times}$ is a locally constant homomorphism that is algebraic in the following sense: the restriction $\widetilde{\xi}|_{\rk^\times}: \rk^\times \rightarrow \RE^\times$ (uniquely) extends to an algebraic homomorphism 
    \[ 
   w_{\widetilde{\xi}}=\{w_{\widetilde{\xi},\rho}\}_{\rho\in \CE_\RE}: (\rk\otimes_\mathbb Q \C)^\times=(\C^\times)^{\CE_{\rk} }\rightarrow (\RE\otimes_\mathbb Q \C)^\times =(\C^\times)^{\CE_{\RE}}
   \] 
    of complex algebraic tori. For every  $\rho\in\CE_{\RE}$, we have an algebraic Hecke character ${^{\rho}\widetilde{\xi}}: \rk^\times\backslash \A^\times\rightarrow \C^\times$ defined by
        \[ 
        {^{\rho}\widetilde{\xi}}(x)=((\rho\circ \widetilde{\xi})(x))\cdot w_{\widetilde{\xi},\rho}^{-1}(x_\infty)
        \]
        for all $x=(x_\infty,x_{\mathrm{f}})\in \A^\times$, where $x_\infty\in\rk_{\infty}$ and $x_\mathrm f\in \A_{\mathrm{f}}$ are  respectively the archimedean and non-archimedean components of $x$. See \cite[Chapter 0]{S88} for more details. If $\widetilde{\xi}$ is critical, in the sense that ${^\rho\widetilde{\xi}}$ is critical for all $\rho\in\CE_{\RE}$, then we have the Deligne period
 \[
 \mathrm{c}^+(\widetilde{\xi})=\{\mathrm{c}^+(\widetilde{\xi})_{\rho}\}_{\rho\in\CE_{\RE}}\in(\RE\otimes_{\BQ}\C)^{\times}=(\C^{\times})^{\CE_{\RE}},
 \] 
 which is uniquely defined up to multiplication by scalars in $\RE^{\times}$. See \cite{D, S88, Ku24} for the precise definition. Deligne's conjecture for $\widetilde{\xi}$ (which is completely proved in \cite{Ku24}) asserts that
\begin{equation}\label{deligneA}
    \frac{\mathbb{L}(0,\widetilde{\xi}_{\mathrm{f}})}{\mathrm{c}^+(\widetilde{\xi})}\in\RE\subset\RE\otimes_{\BQ}\C,\qquad \text{where}\quad \mathbb{L}(0,\widetilde{\xi}_{\mathrm{f}}):=\{\RL(0,{^\rho\widetilde{\xi}}_{\mathrm{f}})\}_{\rho\in\CE_{\RE}}.
\end{equation}
Here $\widetilde{\xi}_{\mathrm{f}}$ is the non-archimedean part of $\widetilde{\xi}$.

\begin{lemp}\label{lemcplus}
    There is a map
    \be\label{cplus}
    \mathrm{c}^+: \{\textrm{critical algebraic Hecke characters $\rk^\times \backslash \A^\times \rightarrow \C^\times$} \}\rightarrow \C^\times
    \ee
    such that 
    \[
    \frac{\{\mathrm{c}^+(^\rho \widetilde{\xi})\}_{\rho\in\CE_{\RE}}}{\mathrm{c}^+( \widetilde{\xi})} \in \RE^\times\subset (\RE\otimes_\BQ\C)^\times,
    \]
    for every number field $\RE$ and every  locally constant homomorphism $\widetilde{\xi}:\A^{\times}\to\RE^{\times}$ that is algebraic and critical.
    Moreover, such a map is unique up to multiplication by rational scalars in the following sense: if $\tilde c^+$ is another map that has the same property as $c^+$, then the map
    \[
    \frac{\widetilde{\mathrm{c}}^+}{\mathrm{c}^+}: \{\textrm{critical algebraic Hecke characters $\rk^\times \backslash \A^\times \rightarrow \C^\times$} \}\rightarrow \C^\times
    \]
    is $\mathrm{Aut}(\C)$-equivariant. 
\end{lemp}
\begin{proof}
    This follows from the fact that  the Deligne periods are preserved under extensions of the  coefficient fields. We omit the details.  
\end{proof}

For every critical algebraic Hecke character $\xi: \rk^\times \backslash \A^\times \rightarrow \C^\times$, we also call the number $\mathrm{c}^+(\xi)\in \C^\times$ given by \eqref{cplus} the Deligne period of $\xi$. Denote 
\[\mathrm{c}(\xi):=\mathrm{c}^+(\xi)\cdot\RL(0,\xi_{\infty}).
\]
We remark that \eqref{deligneA} is equivalent to saying that 
\begin{equation}\label{deligneB}
    \sigma\left(\frac{\RL(0,\xi)}{\mathrm{c}(\xi)}\right)=\frac{\RL(0,{^\sigma\xi})}{\mathrm{c}({^\sigma}\xi)},\qquad \text{for every }\sigma\in\mathrm{Aut}(\C).
\end{equation}

\subsubsection{Whittaker periods}\label{sec:whittaker}

We recall the definition of Whittaker periods in \cite[Sections 3.2, 3.3]{Ma05} and \cite{dual} (see also \cite{RS08}). 

We know that there exists a family $\{\varpi_{\Pi}\}_{\Pi\in\CR^{\rm coh}}$ of $\GL_n(\A_{\mathrm{f}})\times\GL_n(\A_{\mathrm{f}})$-equivariant linear embeddings
\begin{equation}\label{varpisigma}
\varpi_{\Pi}:{\RI^{\rm whi}_{\mathrm{f}}(\Pi)}\to\RH^{b_n}(\CX_n,F_{{\mu}}^\vee)\otimes\RH_{\mathrm{c}}^{b_n}(\CX_n,F_{\nu}^\vee),
\end{equation}
where $F_\mu\otimes F_\nu$ is the coefficient system of $\Pi$,
such that the diagram (\cite[(71)]{LLS})
\[
\begin{CD}
    \RI^{\rm whi}_{\mathrm{f}}(\Pi) @>\varpi_{\Pi}>>\RH^{b_n}(\CX_n,F_{{\mu}}^{\vee})\otimes\RH_{\mathrm{c}}^{b_n}(\CX_n,F_{\nu}^{\vee})\\
    @V\sigma VV @VV\sigma V\\
    \RI^{\rm whi}_{\mathrm{f}}(^\sigma\Pi) @>\varpi_{^\sigma\Pi}>>\RH^{b_n}(\CX_n,F_{{^{\sigma}\mu}}^\vee)\otimes\RH_{\mathrm{c}}^{b_n}(\CX_n,F_{^{\sigma}\nu}^\vee)
\end{CD}
\]
commutes for every $\sigma\in\mathrm{Aut}(\C)$ and $\Pi\in\CR^{\rm coh}$. See \cite[Lemma 3.15 and Theorem 3.19]{Clo}, \cite[Lemma 6.4]{LLS} and \cite[Theorem 4.13]{dual}.

For every $\Pi\in\CR^{\rm coh}$, the Whittaker period $\Omega(\Pi)\in\C^{\times}$ is defined to be the unique scalar such that the diagram
\begin{equation}\label{def:whittaker}
\begin{CD}
   \RI^{\rm whi}_{\mathrm{f}}(\Pi) @>\mathrm{id}\otimes\kappa_{{\mu},{\nu}}>> \RI^{\rm whi}_{\mathrm{f}}(\Pi)\otimes\mathcal{H}(\Pi_{\infty})\\
    @V\Omega(\Pi)VV @VV\imath_{\Pi}\circ\imath_{\mathrm{can}}V\\
   \RI^{\rm whi}_{\mathrm{f}}(\Pi) @>\varpi_{\Pi}>>\RH^{b_n}(\CX_n,F_{{\mu}}^\vee)\otimes\RH_{\mathrm{c}}^{b_n}(\CX_n,F_{\nu}^\vee)
\end{CD}
\end{equation}
commutes. By Lemma \ref{lem:wd}, $\kappa_{\mu,\nu}$ and thus $\Omega(\Pi)$ are independent of the CM type $\{\iota'_v\}_{v\mid \infty}$. 

As a consequence of the definition, the diagram
	\begin{equation}\label{diagrampi}
		\begin{CD}
			\CH(\Pi_{\infty})\otimes\RI_{\mathrm{f}}^{\rm whi}(\Pi) @>\imath_{\mathrm{can}}>> \CH(\Pi) @>\Omega(\Pi)^{-1}\cdot\imath_{\Pi}>>\RH^{b_n}(\CX_n,F_{\mu}^{\vee})\otimes\RH_{\mathrm{c}}^{b_n}(\CX_n,F_{\nu}^{\vee})\\
			@V\sigma VV @VV\sigma V @VV\sigma V\\
			\CH({^{\sigma}\Pi_{\infty}})\otimes{\RI^{\rm whi}_{\mathrm{f}}(^\sigma\Pi)} @>\imath_{\mathrm{can}}>>\CH({^{\sigma}\Pi}) @>\Omega({^{\sigma}\Pi})^{-1}\cdot\imath_{{^{\sigma}\Pi}}>>\RH^{b_n}(\CX_n,F_{{^{\sigma}\mu}}^{\vee})\otimes\RH_{\mathrm{c}}^{b_n}(\CX_n,F_{{^{\sigma}\nu}}^{\vee})
		\end{CD}
	\end{equation}
	commutes for every $\sigma\in\mathrm{Aut}(\C)$. Here the middle vertical arrow is defined such that the left square is commutative. 

Let $\Pi\in\CR^{\rm coh}$. Denote by $\Pi^{\vee}$ the contragredient of $\Pi$. Note that $\Pi^{\vee}\in\CR^{\rm coh}$, and 
\[
^\sigma\Pi^{\vee}:=(^\sigma\Pi)^{\vee}={^\sigma(\Pi^{\vee})} \quad\textrm{ for every $\sigma\in\mathrm{Aut}(\C)$}.
\]
We record the following proposition proved in \cite[Theorem 1.3]{dual}.

\begin{prp}
    For every $\sigma\in\mathrm{Aut}(\C)$, it holds that
\begin{equation}\label{whittakergalois}
    \sigma\left(\frac{\Omega(\Pi)}{\mathcal{G}(\omega_{\Pi})^{n-1}\cdot\Omega(\Pi^{\vee})}\right)=\frac{\Omega(^\sigma\Pi)}{\mathcal{G}(\omega_{^\sigma\Pi})^{n-1}\cdot\Omega(^\sigma\Pi^{\vee})}.
\end{equation}
\end{prp}

\subsubsection{Proof of Theorem \ref{mainthm} for $\chi\in\RB(\mu,\nu)^+$}
    
	For every $\sigma\in\mathrm{Aut}(\C)$, set 
	 \[
      {^{\sigma}\CL}:= \frac{\RL(0,{^{\sigma}\Pi}\times{^{\sigma}\chi})}{\RL(0,\omega_{^{\sigma}\Pi}\cdot{^{\sigma}\chi}^n)\cdot\CG({^{\sigma}\chi})^{\frac{n(n-1)}{2}}\cdot\Omega_{\infty}({^{\sigma}\Pi},{^{\sigma}\chi})\cdot\Omega({^{\sigma}\Pi})}.
    \]	
  	The symbol $\sigma$ will be omitted from the notation if it is the identity element of $\Aut(\C)$. 
       By \eqref{deligneB}, we have that
    \[
\sigma\left(\frac{\RL(0,\omega_{\Pi}\chi^n)}{\mathrm{c}(\omega_{\Pi}\chi^n)}\right) = \frac{\RL(0,\omega_{^{\sigma}\Pi}\cdot{^{\sigma}\chi}^n)}{\mathrm{c}(\omega_{^\sigma\Pi}\cdot{^\sigma\chi^n})}
    \]
    for every $\sigma\in\mathrm{Aut}(\C)$. Therefore, we only need to prove that 
	\begin{equation}
		\label{thmsigma'}
		\sigma(\CL)={^{\sigma}\CL},
	\end{equation}
	which is equivalent to \eqref{thmsigma}. Then \eqref{thmquotient} follows from \eqref{thmsigma} immediately.
	
	We form the following diagram
	\begin{equation}
		\label{diagram}
		\xymatrixrowsep{0.3in}
		\xymatrixcolsep{0.35in}
		\xymatrix{
			\mathcal{H}(\Pi,\chi)_{\mathrm{loc}} \ar[rd]^{\sigma} \ar[ddd]_{\imath_{\mathrm{can}}} \ar[rrr]^{   \CP^{ \circ}_{\infty}\otimes\CP^{ \circ}_{\mathrm{f}} } 
			&&& \BC \ar^\sigma[rrd] \ar^{ \sigma}[rrd] \ar[ddd]^{\mathcal{L}}|!{[dll];[dr]}\hole & \\
			&\mathcal{H}({^{\sigma}\Pi},{^{\sigma}\chi})_{\mathrm{loc}}  \ar[ddd]_{\imath_{\mathrm{can}}}  \ar[rrrr]^(0.4){   {^{\sigma}\CP}^{ \circ}_{\infty}\otimes{^{\sigma}\CP}^{\circ}_{\mathrm{f}} } 
			&&& & \BC  \ar[ddd]^{^{\sigma}\mathcal{L}}\\
			&&& & \\
			\mathcal{H}(\Pi,\chi)_{\mathrm{glob}} \ar[rd]_{\sigma} \ar[rrr]^(0.6){\Omega(\Pi)^{-1}\cdot\CP} |!{[ur];[dr]}\hole
			&& & \BC \ar[drr]^{\sigma}  & \\
			& 	\mathcal{H}({^{\sigma}\Pi},{^{\sigma}\chi})_{\mathrm{glob}}\ar[rrrr]^{\Omega({^{\sigma}\Pi})^{-1}\cdot{}^\sigma\CP} 
			&&& & \BC.
		}
	\end{equation}
	Then \eqref{thmsigma'} is the commutativity of the right square of \eqref{diagram}. Since the top horizontal arrow is surjective by the non-vanishing of the archimedean modular symbol (\cite[Theorem 1.6(b)]{local}), it suffices to show the commutativity of the other squares of the cubic diagram. The front and back squares are commutative by \eqref{frontback}. The left square commutes by  \eqref{diagrameis} and \eqref{diagrampi}. The bottom square is commutative by \cite[Proposition 3.14]{Rag1}, \eqref{maindiagram} and \eqref{diagrampi}. The top square is commutative by the non-archimedean period relation \eqref{nonarchimedeanperioerelation} and the archimedean period relation \eqref{archimedeanperiodrelation}. We therefore obtain \eqref{thmsigma'} and complete the proof of Theorem \ref{mainthm} for $\chi\in\RB(\mu,\nu)^+$.

\subsection{The functional equation}\label{sec7.6}
We finally complete the proof of Theorem \ref{mainthm} for all $\chi\in\RB(\mu,\nu)$ using the functional equation.
Assume that $\chi \in \RB(\mu,\nu)\setminus \RB(\mu,\nu)^+$. Then $\chi^{-1}\abs{\cdot}_\A\in \RB(\check\mu,\check\nu)^+$, where
\[
\check\mu:=\{(-\mu_n^{\iota},-\mu_{n-1}^{\iota},\dots,-\mu_1^{\iota})\}_{\iota\in\CE_{\rk}},\quad\check\nu:=\{(-\nu_n^{\iota},-\nu_{n-1}^{\iota},\dots,-\nu_1^{\iota})\}_{\iota\in\CE_{\rk}}.
\]
We have already proved that
\begin{equation}
    \label{chi-1}
    \begin{aligned}
        &\sigma\left(\frac{\RL(0,\Pi^{\vee}\times\chi^{-1}\abs{\cdot}_{\A})}{\mathrm{c}(\omega_{\Pi}^{-1}\chi^{-n}\abs{\cdot}^n_{\A})\cdot\mathcal{G}(\chi^{-1})^{\frac{n(n-1)}{2}}\cdot\Omega_{\infty}(\Pi^{\vee},\chi^{-1}\abs{\cdot}_{\A})\cdot\Omega(\Pi^{\vee})}\right)\\
        =\ &\frac{\RL(0,{^\sigma\Pi^{\vee}}\times{^\sigma\chi^{-1}}\abs{\cdot}_{\A})}{\mathrm{c}(\omega_{^\sigma\Pi}^{-1}\,{^\sigma\chi^{-n}\abs{\cdot}_{\A}^n})\cdot\mathcal{G}(^\sigma\chi^{-1})^{\frac{n(n-1)}{2}}\cdot\Omega_{\infty}(^\sigma\Pi^{\vee},{^\sigma\chi^{-1}}\abs{\cdot}_{\A})\cdot\Omega(^\sigma\Pi^{\vee})}
    \end{aligned}
\end{equation}
for every $\sigma\in\mathrm{Aut}(\C)$.

Recall the functional equation (see e.g. \cite{JPSS83} and  \cite[Lecture 9]{Cog})
\begin{equation}\label{FE}
\RL(s,\Pi\times\chi)=\varepsilon(s,\Pi\times\chi)\cdot\RL(1-s,\Pi^{\vee}\times\chi^{-1}),
\end{equation}
where $\varepsilon(s,\Pi\times\chi)$ is the epsilon factor whose infinite part is (see \cite{Kn91} and \cite[Lemma 3.13]{local}) 
\be \label{epsilon}
\begin{aligned}
    \varepsilon_{\infty}(\Pi\times\chi):=&\,\prod_{v|\infty}\varepsilon(0,\Pi_v\times\chi_v,\psi_v)\\
      =&\,\prod_{v|\infty}\prod_{i+k\leq n} (\varepsilon_{\Pi,\chi,v}\cdot {\rm i})^{\mu_i^{\iota_v}+\nu_k^{\iota_v}+\chi_{\iota_v}  - \mu^{\bar\iota_v}_{n+1-i} - \nu^{\bar\iota_v}_{n+1-k} -\chi_{\bar\iota_v}+ 2(n+1-i -k) } \\
    & \cdot \prod_{v\mid\infty} \prod_{i+k\geq n+1}(\varepsilon_{\Pi,\chi,v}\cdot {\rm i})^{- \mu_i^{\iota_v} - \nu_k^{\iota_v}- \chi_{\iota_v}  + \mu^{\bar\iota_v}_{n+1-i} + \nu^{\bar\iota_v}_{n+1-k} +\chi_{\bar\iota_v}+ 2(i+k-n-1) }\\
    &\cdot \prod_{\substack{
    v|\infty\\\eta_{\iota_v}\geq n
    }}(-1)^{(n+1)(\eta_{\iota_v}+\eta_{\overline{\iota}_v})}.
\end{aligned}
\ee
Note that 
\[
\varepsilon_{\infty}(^\sigma\Pi\times {^\sigma\chi})=\varepsilon_{\infty}(\Pi\times\chi)\qquad\text{for every }\sigma\in\mathrm{Aut}(\C).
\]

\begin{lemp} \label{lem:epsilon}
    For every $\sigma\in\mathrm{Aut}(\C)$, 
   it holds that  \begin{equation}
    \label{epsilongalois}
    \sigma\left(\frac{\varepsilon(0,\Pi\times\chi)}{|\delta_{\rk}|^{\frac{n}{2}}\cdot\mathcal{G}(\omega_{\Pi}\chi^n)^n\cdot\varepsilon_{\infty}(\Pi\times\chi)}\right)=\frac{\varepsilon(0,{^\sigma\Pi}\times{^\sigma\chi})}{|\delta_{\rk}|^{\frac{n}{2}}\cdot\mathcal{G}(\omega_{^\sigma\Pi}\,{^\sigma\chi}^n)^n\cdot\varepsilon_{\infty}(^\sigma\Pi\times {^\sigma\chi})}.
\end{equation}

\end{lemp}

\begin{proof}
This can be proved by using the same argument as in the proof of \cite[Lemma 3.1]{C24}. 
\end{proof}

On the other hand, by \cite[Proposition 5.1, Section 8]{D} (see also \cite[Section 5.4.1]{R-imaginary}) we have
\begin{equation}\label{deligneFE}
\begin{aligned}
&\sigma\left(\frac{\mathrm{c}(\omega_{\Pi}\chi^n)}{\mathrm{c}(\omega_{\Pi}^{-1}\chi^{-n}\abs{\cdot}_{\A}^n)\cdot|\delta_{\rk}|^{\frac{1}{2}}\cdot(\mathrm{i}^{\frac{[\rk:\BQ]}{2}}\cdot\Delta_{\rk})^{1-n}\cdot\mathcal{G}(\omega_{\Pi}\chi^n)\cdot\varepsilon_{\infty}(\omega_{\Pi}\chi^n)}\right)\\
=\,&\frac{\mathrm{c}(\omega_{^\sigma\Pi}{^\sigma\chi^n})}{\mathrm{c}(\omega_{^\sigma\Pi}^{-1}\,{^\sigma\chi^{-n}\abs{\cdot}_{\A}^n})\cdot|\delta_{\rk}|^{\frac{1}{2}}\cdot(\mathrm{i}^{\frac{[\rk:\BQ]}{2}}\cdot\Delta_{\rk})^{1-n}\cdot\mathcal{G}(\omega_{^\sigma\Pi}\,{^\sigma\chi^n})\cdot \varepsilon_{\infty}(\omega_{^\sigma\Pi}\,{^\sigma\chi^n})}
\end{aligned}
\end{equation}
for every $\sigma\in\mathrm{Aut}(\C)$, where
\be \label{cinf}
\varepsilon_{\infty}(\omega_{^\sigma\Pi}{^\sigma\chi^n})=\varepsilon_{\infty}(\omega_{\Pi}\chi^n):= \prod_{v\mid\infty} \varepsilon(0, \omega_{\Pi_v}\chi_v^n,\psi_v) = \prod_{v|\infty} (\varepsilon_{\Pi,\chi,v}\cdot \mathrm{i})^{\eta_{\bar\iota_v} - \eta_{\iota_v}}.
\ee

By noting that $\varepsilon_{\Pi^\vee,\chi^{-1}\abs{\cdot}_\A,v} = -\varepsilon_{\Pi,\chi,v}$ for every $v|\infty$ and
\[
\frac{\delta_{\infty}(\omega_{\Pi}^{-1}\chi^{-n}\abs{\cdot}^n_{\A})}{\delta_{\infty}(\omega_{\Pi}\chi^n)}=\frac{\prod_{\tau\in\CE_{\rk_1},\,\eta_{\tau}\leq 0}\delta(\rk/\rk_1,\tau)}{\prod_{\tau\in\CE_{\rk_1},\,\eta_{\tau}\geq n}\delta(\rk/\rk_1,\tau)}=:\nabla_{\rk}(\eta),
\]
we find that 
\be \label{omegainf}
\begin{aligned}
    \frac{\Omega_{\infty}(\Pi^{\vee},\chi^{-1}\abs{\cdot}_{\A})}{\Omega_{\infty}(\Pi,\chi)}=&(-1)^{\frac{[\rk:\BQ]}{4}n(n-1)}\cdot\nabla_{\rk}(\eta)^{n-1}\cdot\prod_{\substack{v|\infty\\
    \eta_{\iota_v}\geq n }}(-1)^{(n+1)(\eta_{\iota_v}+\eta_{\overline{\iota}_v})} \\
    &\cdot\prod_{v|\infty}\prod_{i+k\leq n}(\varepsilon_{\Pi,\chi,v}\cdot\mathrm{i})^{-\sum_{\iota\in\CE_{\rk_v}}(\mu_i^{\iota}+\nu_k^{\iota}+\chi_{\iota})+1}\\
    &\cdot\prod_{v|\infty}\prod_{
    i+k\geq n+2}(\varepsilon_{\Pi,\chi,v}\cdot\mathrm{i})^{-\sum_{\iota\in\CE_{\rk_v}}(\mu_i^{\iota}+\nu_k^{\iota}+\chi_{\iota})+1}.
\end{aligned}
\ee
Using \eqref{epsilon}, \eqref{cinf} and \eqref{omegainf}, it is straightforward to verify that
\be \label{inf}
\frac{\varepsilon_\infty(\Pi \times\chi)}{\varepsilon_{\infty}(\omega_{\Pi}{\chi^n})} \cdot \frac{\Omega_{\infty}(\Pi^{\vee},\chi^{-1}\abs{\cdot}_{\A})}{\Omega_{\infty}(\Pi,\chi)} = \nabla_{\rk}(\eta)^{n-1}\cdot\prod_{\iota\in\CE_{\rk}}\prod_{i+k\geq n+2}(-1)^{\mu_i^{\iota}+\nu_k^{\iota}+\chi_{\iota}}.
\ee

Finally, by combining \eqref{whittakergalois}, \eqref{epsilongalois}, \eqref{deligneFE}, \eqref{inf}, and noting, by \eqref{dk0} and \eqref{nabla}, that
\[
\frac{\sigma(\nabla_{\rk}(\eta))}{\nabla_{\rk}({^\sigma\eta})}=\frac{\sigma(\nabla_{\rk})}{\nabla_{\rk}},
\]
we obtain
\[
\begin{aligned}
    &\sigma\left(\varepsilon(0,\Pi\times\chi)\cdot\frac{\mathrm{c}(\omega_{\Pi}^{-1}\chi^{-n}\abs{\cdot}^n_{\A})\cdot\mathcal{G}(\chi^{-1})^{\frac{n(n-1)}{2}}\cdot\Omega_{\infty}(\Pi^{\vee},\chi^{-1}\abs{\cdot}_{\A})\cdot\Omega(\Pi^{\vee})}{\mathrm{c}(\omega_{\Pi}\chi^n)\cdot\mathcal{G}(\chi)^{\frac{n(n-1)}{2}}\cdot\Omega_{\infty}(\Pi,\chi)\cdot\Omega(\Pi)}\right)\\
    =\, &\varepsilon(0,^\sigma\Pi\times{^\sigma\chi})\cdot\frac{\mathrm{c}(\omega_{^\sigma\Pi}^{-1}\,{^\sigma\chi^{-n}\abs{\cdot}_{\A}^n})\cdot\mathcal{G}(^\sigma\chi^{-1})^{\frac{n(n-1)}{2}}\cdot\Omega_{\infty}(^\sigma\Pi^{\vee},{^\sigma\chi^{-1}}\abs{\cdot}_{\A})\cdot\Omega(^\sigma\Pi^{\vee})}{\mathrm{c}(\omega_{{^{\sigma}\Pi}}\, {}^{\sigma}\chi^n)\cdot
					\CG({^{\sigma}\chi})^{\frac{n(n-1)}{2}}\cdot\Omega_{\infty}({^{\sigma}\Pi},{^{\sigma}\chi})\cdot\Omega({^{\sigma}\Pi})}.
\end{aligned}
\]
By the functional equation \eqref{FE}, we obtain \eqref{thmsigma} from \eqref{chi-1}, which completes the proof of Theorem \ref{mainthm}.

	\section*{Acknowledgments}
    Y. Jin was supported by National Natural Science Foundation of China No. 12601062. D. Liu was supported in part by National Key R \& D Program of China No. 2022YFA1005300 and National Natural Science Foundation of China No. 12526208.  B. Sun was supported in part by National Key R \& D Program of China No. 2022YFA1005300  and New Cornerstone Science Foundation.

    After the paper was written, ChatGPT-6 Astra was used to proofread the manuscript and verify computations, helping to identify and correct various errors in earlier versions.

\end{document}